\documentclass[10pt]{amsart}
\usepackage[colorlinks]{hyperref}
\usepackage{lineno}
\usepackage{esint}
\usepackage[margin=1in]{geometry}
\usepackage[latin1]{inputenc}
\usepackage{amsfonts}
\usepackage{enumitem}
\usepackage{color}
\usepackage{bm}
\usepackage{textcomp}
\usepackage [spanish,english]{babel}
\usepackage{graphicx,dsfont}
\usepackage{amssymb,amsmath,amsthm}
\usepackage{lineno}

\newtheorem{theorem}{Theorem}[section]
\newtheorem{definition}{Definition}[section]
\newtheorem{proposition}{Proposition}[section]
\newtheorem{lemma}{Lemma}[section]
\newtheorem{corollary}{Corollary}[section]

\newtheorem{remark}{Remark}[section]
\theoremstyle{plain} 

\catcode`\@=11
\makeatletter

\@addtoreset{equation}{section}

\begin{document}


\title[Double-phase evolution problem]
{Global Calder\'on-Zygmund estimates for irregular double-phase evolution problem with non-divergence data}

\author{Rakesh Arora}
\address{Department of Mathematical Sciences, Indian Institute of Technology (IIT-BHU), Varanasi 221005, India}
\email{rakesh.mat@iitbhu.ac.in}

\author{Sergey Shmarev}
\address{Department of Mathematics, University of Oviedo, c/Federico Garc\'{i}a Lorca 18, 33007, Oviedo, Spain}
\email{shmarev@uniovi.es}
\thanks{The first author acknowledges the financial support from the Anusandhan National Research Foundation (ANRF), India, under Grant No. ANRF/ARGM/2025/000272/MTR}

\keywords{Nonlinear parabolic equation, Double-phase, Global Gradient estimates, Second-order regularity}
\subjclass[2020]{35K65, 35K67, 35B65,  35K55, 35K99}

\begin{abstract}
We study irregular double-phase parabolic equations with variable exponents and non-divergence data,
\[
u_t-\operatorname{div} \left(\mathcal{F}(z,\nabla u)\nabla u \right)=f(z),\quad z=(x,t)\in Q_T:=\Omega\times (0,T),
\]
 under the homogeneous Dirichlet boundary conditions. Here,
$\Omega \subset \mathbb{R}^N$, $N \geq 2$, is a bounded domain, $T>0$,
\[
\mathcal{F}(z,\nabla u)=a(z)|\nabla u|^{p(z)-2}  + b(z) |\nabla u |^{q(z)-2}
\]
with given Lipschitz-continuous exponents $p,q$ that satisfy a suitable balance condition. The nonnegative coefficients $a(z), b(z)$ satisfy the inequality
$a(z)+b(z)>0$ in $Q_T$, the space and time derivatives of $a$ and $b$ belong to $L^d(Q_T)$ with some $d$ depending on the data. If
\[
f\in L^\sigma(Q_T) \quad \text{for} \ \sigma \in (2, N+2] \quad \text{and} \quad \mathcal{F}((\cdot,0),\nabla u_0)\,|\nabla u_0|^{r+2}\in L^1(\Omega),
\]
where \(0\le r\le K(N,\sigma,p,q)\) if \(\sigma<N+2\), while \(r\ge0\) is arbitrary if \(\sigma=N+2\), then the problem has a unique strong solution, for which we prove the global transfer of integrability from the initial data and the forcing term to the double-phase flux in the spirit of Calder\'on-Zygmund theory, higher integrability of the gradient, and the second-order space regularity:
\[
\begin{split}
& \text{$\mathcal{F}((\cdot,t),\nabla u(\cdot,t))|\nabla u(\cdot,t)|^{r+2}\in L^1(\Omega)$ for a.e. $t\in (0,T)$},
\\
& \text{$|\nabla u|^{2(\min\{p(z),q(z)\}-1)+r+s}\in L^1(Q_T)$
 for every $s\in\left(0,\frac{4}{N+2}\right)$},
\\
&
\mathcal{F}(z,\nabla u)|\nabla u|^{\frac{r+2}{2}}
\in L^2(0,T;W^{1,2}(\Omega)).
\end{split}
\]
The results improve and complement the results in \cite{Arora-Shmarev-JGA-2026} and extend them to the full range $r \geq 0$.
\end{abstract}

\maketitle

\tableofcontents

\section{Introduction}
In this work, we study the irregular double-phase parabolic problem with non-divergence data
\begin{equation}
\label{eq:main}
\begin{split}
& u_t-\operatorname{div}\left(\mathcal{F}(z,\nabla u)\nabla u\right)= f(z), \quad z=(x,t)\in Q_T=\Omega\times (0,T),
\\
& \text{$u=0$ on $\partial\Omega\times (0,T)$},\qquad \text{$u(x,0)=u_0(x)$ in $\Omega$},
\end{split}
\end{equation}
where $\Omega \subset \mathbb{R}^N$ is a bounded domain with $\partial\Omega\in C^{2+\gamma}$, $\gamma\in (0,1)$, $N \geq 2$, $T>0$, and the flux function $\mathcal{F}$ has the form
\begin{equation}
\label{eq:flux-start}
\mathcal{F}(z,\xi)=a(z) |\xi|^{p(z)-2}+b(z) |\xi|^{q(z)-2}
\end{equation}
with given nonnegative modulating coefficients $a(z)$, $b(z)$ and variable exponents $p(z)$, $q(z)$. It is assumed that
\begin{equation}
\label{eq:coeff}
\begin{split}
&
\text{$\alpha\leq a(z)+b(z)$ in $\overline{Q}_T$ for some constant $\alpha>0$}.
\end{split}
\end{equation}
Equation \eqref{eq:main} belongs to the class of double-phase problems, characterized by the presence of multiple growth regimes governed by the exponents $p(z)$ and $q(z)$, with $z=(x,t)\in Q_T$. The growth of the flux is not uniform throughout the domain but varies according to the modulating coefficients $a(z)$ and $b(z)$, leading to a heterogeneous behavior. In regions where both coefficients are positive, the flux exhibits a mixed growth structure involving both exponents, whereas in regions where either coefficient vanishes, the problem reduces to a single-phase model. Thus, the problem exhibits transitions between mixed and single growth regimes across the interfaces determined by the supports of $a(z)$ and $b(z)$. Such equations naturally arise in nonlinear elasticity \cite{Zhikov-1986, Zhikov-1995}, where the coefficients $a$ and $b$ together with the exponents $p$ and $q$ describe materials with different hardening properties.

A typical example is given by $a=1$, $\operatorname{supp} b \Subset \Omega$, and $p<q$, reflecting the coexistence of distinct material responses and implying that solutions need not belong to $W^{1,q}(\Omega)$. Consequently, the natural functional framework for their study is provided by Musielak-Orlicz-Sobolev spaces; see Subsection \ref{subsec:spaces}.

The interplay between spatially-temporal varying coefficients and nonstandard growth conditions creates substantial analytical difficulties related to existence, regularity estimates, compactness methods and variational convergence. To address these difficulties, it is necessary to impose suitable gap conditions between the exponents $p(z)$ and $q(z)$ (see \cite{DeFilippis-2020, RACSAM-2023}), together with appropriate regularity assumptions on the modulating coefficients $a(z)$ and $b(z)$ (see \cite{Wontae-Kim-JMAA-2025, Wontae-Kim-NoDeA-2024}). Owing to these features, double-phase problems, both in stationary and evolutionary settings, have attracted considerable attention in recent years. We refer to \cite{Mingione-Radulescu-2021} for an overview of recent developments in nonstandard growth problems and nonuniformly elliptic equations and to \cite{Chl-Gw-Gw-Wr-2021} for the functional analytic foundations of Musielak-Orlicz spaces.




A central theme in the regularity theory of nonlinear partial differential equations is the Calder\'on-Zygmund principle, which seeks to transfer the integrability of the data to the gradient of solutions. In the elliptic setting, this theory is by now well developed for equations of $p$-Laplace type. For equations with divergence-form right-hand sides, fundamental contributions were obtained by Iwaniec \cite{Iwaniec-1983} and DiBenedetto and Manfredi \cite{DiBenedetto-Manfredi-AMJ-1993}. For elliptic equations, the case of non-divergence form data can be reduced to the divergence setting through the Bogovski\u{\i}  regularity results. However, such a technique is not applicable for parabolic equations (see \cite{Byun-Kim-2022, {Andrade-Bogelein-Duzaar-Moring-2026}}). We refer the reader to \cite{Kinnunen-Zhou-2001, Byun-Ok-Ryu-2013, Colombo-Mingione-2016, De-Filippis-Mingione-2020} for further developments and refinements of the elliptic Calder\'on-Zygmund theory.

For the model parabolic $p$-Laplace equation
\begin{equation}
\label{main:p-const}
\begin{split}
& u_t-\operatorname{div}\left(|\nabla u|^{p-2} \nabla u\right)
= \operatorname{div}\left(|F|^{p-2} F\right) + f
\quad \text{in} \ Q_T,
\\
& \text{$u=0$ on $\partial\Omega\times (0,T)$},
\qquad
\text{$u(x,0)=u_0(x)$ in $\Omega$},
\end{split}
\end{equation}
local and global Calder\'on-Zygmund type estimates have been the subject of intensive research over the last two decades. In the homogeneous case \(f=0\), Acerbi and Mingione \cite{Acerbi-Mingione-2007}, employing intrinsic scaling techniques and covering arguments, established the local Calder\'on-Zygmund theory for weak solutions to \eqref{main:p-const}. In particular, they proved that
\[
F \in L^q_{\mathrm{loc}}(Q_T)
\quad \Longrightarrow \quad
\nabla u \in L^q_{\mathrm{loc}}(Q_T),
\qquad q \ge p.
\]

Recently, these estimates have been extended to the global setting on the whole cylinder $Q_T$ for the case $f \neq 0$ by Byun and Kim \cite{Byun-Kim-2022} for parabolic equations and subsequently by Andrade, B\"ogelein, Duzaar, and Moring \cite{Andrade-Bogelein-Duzaar-Moring-2026} for parabolic systems. More precisely, they proved that
\[
|\nabla u_0|^q \in L^1(\Omega), \quad |F|^q, |f|^{\frac{(p^\ast)'q}{p}} \in L^1(Q_T)
\quad \Longrightarrow \quad
\nabla u \in L^q(Q_T),
\qquad q \geq p.
\]
where $p^\ast = \frac{p(N+2)}{N}$ is the parabolic conjugate of $p$ and $(p^\ast)'$ is the H\"older conjugate of $p^\ast.$ The proofs are based on comparison estimates, stopping-time arguments, and suitable applications of the Gagliardo-Nirenberg inequality. For further developments of the Calder\'on-Zygmund theory, we refer the reader to Acerbi, Mingione, and Seregin \cite{Seregin-Acerbi-Mignione-2004} for parabolic systems with polynomial growth, Crispo, Maremonti, and R\r{u}\u zi\u cka \cite{Crispo-Maremonti-Ruzicka-2019}, Baroni and B\"ogelein \cite{Baroni-Bogelin-2014} for problems with nonstandard growth, and Byun, Oh, and Wang \cite{Byun-Oh-Wang-2015} for asymptotically regular parabolic equations.

Very recently, Kim \cite{Kim-2025,Wontae-Kim-JMAA-2025} established Calder\'on-Zygmund type estimates for the double-phase problem
\begin{equation}
\label{main:p-q-const}
\begin{split}
& u_t-\operatorname{div}
\left(
|\nabla u|^{p-2}\nabla u
+a(z)|\nabla u|^{q-2}\nabla u
\right) =
\operatorname{div}
\left(
|F|^{p-2}F
+a(z)|F|^{q-2}F
\right)
\qquad \text{in } Q_T,
\\
& \text{$u=0$ on $\partial\Omega\times (0,T)$},
\qquad
\text{$u(x,0)=u_0(x)$ in $\Omega$}.
\end{split}
\end{equation}
Under suitable gap conditions on the exponents $p$ and $q$, together with appropriate regularity assumptions on the coefficient $a$, Kim \cite{Kim-2025,Wontae-Kim-JMAA-2025} exploited the intrinsic geometry associated with the double-phase structure and suitable comparison estimates to prove that
\[
|F|^{p-2}F+a(z)|F|^{q-2}F
\in L^\sigma(Q_T)
\Longrightarrow
|\nabla u|^{p-2}\nabla u
+a(z)|\nabla u|^{q-2}\nabla u
\in L^\sigma(Q_T), \quad \text{for every} \ \sigma\in(0,\infty).
\]
To the best of our knowledge, no analogous Calder\'on-Zygmund type results are currently available for double-phase parabolic equations with non-divergence form data. The present paper addresses this issue in the more general setting of variable exponents and irregular modulating coefficients.

Another major theme of nonlinear parabolic regularity concerns self-improving properties of weak solutions. For equation of the form \eqref{main:p-const}, a fundamental question is whether the natural energy estimate
\[
\nabla u \in L^{p}(Q_T)
\]
can be improved automatically to
\[
\nabla u \in
L^{p + \delta}_{\mathrm{loc}}(Q_T)
\]
for some $\delta>0$. Such higher integrability phenomena play a crucial role in obtaining compactness, Calder\'on-Zygmund estimates and strong convergence results. For variable-growth problems this theory goes back to the pioneering works of Antontsev and Zhikov \cite{Ant-Zhikov-2005} and Zhikov and Pastukhova \cite{Zhikov-Past-2010}, while recent developments include the works of H\"asto and Ok \cite{Hasto-OK-2021}, Kim and S\"arki\"o \cite{Wontae-Kim-NoDeA-2024}, Sen \cite{Sen-2025}, and Arora and Shmarev \cite{Ar-Sh-2021, Ar-Sh-2024, Arora-Shmarev-JMAA-2025}.

For nonlinear parabolic equations, an application of the self-improving properties are second-order estimates of the form
\[
|\nabla u|^{\lambda} \nabla u \in L^2(0,T;W^{1,2}(\Omega))^N,
\]
with $\lambda$ depending on $p$ and $N$, as demonstrated in~\cite{Seregin-Acerbi-Mignione-2004, Duzaar-Mignione-Steffen-2011, Feng-Parviainen-Sarsa-2023, DeFilippis-2020}, while global results are discussed in~\cite{Berselli-Ruzicka-2022}, see also references therein. The global second-order regularity of weak solutions to the Dirichlet problem for $p(z)$-Laplace evolution equation with different choices of the source terms and regularity of initial data was studied by Cianchi-Maz'ya \cite{Cianchi-Maz'ya-2019-1} and Arora-Shmarev \cite{Ar-Sh-2021,Ar-Sh-2024}. These type of estimates are closely connected with the existence theory of strong solutions and constitute one of the cornerstones of modern regularity theory.

In~\cite{Cianchi-Maz'ya-2019-1}, optimal second-order regularity is shown for approximable solutions of the homogeneous Dirichlet problem associated with the constant exponent $p$-Laplace evolution equation, proving that
\[
u_t \in L^2(Q_T), \quad |\nabla u|^{p-2} \nabla u \in L^2(0,T;W^{1,2}(\Omega))^N.
\]
The variable exponent analogue is addressed in~\cite{Ar-Sh-2021}, where Galerkin approximations are employed for $f\in L^2(0,T;W^{1,2}_0(\Omega))$, and $u_0\in W^{1,q(\cdot)}_0(\Omega)$ with $q(x)=\max\{2,p(x,0)\}$, yielding
\[
u_t \in L^2(Q_T), \quad D_{x_i}\left(|\nabla u|^{\frac{p(z)-2}{2}} D_{x_j} u\right) \in L^2(Q_T).
\]
A more flexible regularization framework was developed in~\cite{Ar-Sh-2024}, involving smooth approximations of both the data and the nonlinear terms. Classical solutions to these regularized equations are shown to converge to a weak solution of the original problem. Under the assumptions $p(z) > \frac{2(N+1)}{N+2}$, $u_0 \in W^{1,p(\cdot,0)}_0(\Omega)$, and $f \in L^2(Q_T)$, the limiting solution satisfies
\begin{equation}
\label{eq:approx-L-2}
\begin{split}
& |\nabla u|^{2(p(z)-1)+\delta} \in L^1(Q_T), \quad \text{for any } \delta \in \left(0,\frac{4}{N+2}\right), \\
& |\nabla u|^{p(z)-2} \nabla u \in L^2(0,T;W^{1,2}(\Omega)).
\end{split}
\end{equation}
These properties are derived from uniform estimates for the approximating classical solutions and the pointwise convergence of their gradients. Thus, the results of~\cite{Cianchi-Maz'ya-2019-1} are naturally extended to the variable exponent framework. In \cite{Arora-Shmarev-JMAA-2025}, it is shown that if $a(z) \equiv b(z) \equiv const$ and the data satisfy conditions \eqref{eq:coeff} with $p\equiv q$, then the gradient of the solution to problem \eqref{eq:main} with the nonlinear source $f(z)+F(z,u,\nabla u)$ maintains the initial order of integrability $r\geq \max\{2,p^+\}$, achieves higher integrability, and the solution acquires the second-order regularity:
\begin{equation}\label{eq:prelim-res}
\begin{split}
& \nabla u \in L^{\infty}(0,T;L^r(\Omega)),\quad |\nabla u|^{p(z)+r+\rho-2} \in L^1(Q_T), \ \text{for any} \ \rho \in \left(0,\frac{4}{N+2}\right)
\\
& |\nabla u|^{\frac{p(z)+r}{2}-2}\nabla u\in L^2(0,T;W^{1,2}(\Omega))^N.
\end{split}
\end{equation}
Extending the approach proposed in \cite{Arora-Shmarev-JMAA-2025} to double phase problem \eqref{eq:main}, Arora and Shmarev \cite{Arora-Shmarev-JGA-2026} proved the existence of a strong solution whose gradient maintains the initial order of integrability $r \geq \max\{2, p^+, q^+\}$ and achieves higher integrability, whereas the solution acquires the second-order regularity:
\begin{equation}\label{eq:prelim-res-double-phase}
\begin{split}
& \nabla u \in L^{\infty}(0,T;L^r(\Omega)),\quad |\nabla u|^{\min\{p(z),q(z)\}+r+\rho-2} \in L^1(Q_T),
\\
& |\nabla u|^{\frac{p(z)+r}{2}-2}\nabla u\in L^2(0,T;W^{1,2}(\Omega))^N,
\end{split}
\end{equation}
The aim of this paper is to extend the results of \cite{Arora-Shmarev-JGA-2026} to the full range \(r\ge 0\) and to strengthen several of the regularity conclusions obtained therein. To this end, we work under the following assumptions. We accept the notation
\[
\overline{s}(z)=\max\{p(z),q(z)\}, \qquad \underline{s}(z)=\min\{p(z),q(z)\}, \qquad D_i v=\frac{\partial v}{\partial x_i}, \quad v_t=\frac{\partial v}{\partial t},
\]
and agree to denote $w^+=\max w$, $w^-=\min w$ for the function $w$ on its domain of definition.

We assume that, for a given parameter $\beta \in [0,1]$, the variable exponents $p(z)$, $q(z)$ satisfy the regularity and gap conditions
\begin{equation}
\label{eq:balance-p-q}
\begin{split}
& \text{$p,q\in C^{0,1}(\overline{Q}_T)$ with the Lipschitz constant $L_{p,q}$}, \\
&\frac{2(N+1)}{N+2} < \underline{s}^- \quad \text{and} \quad \begin{cases}
\sup_{Q_T} |p(z)-q(z)| \leq \frac{2\beta}{N+2} \quad \text{if} \ \beta \in [0,1),\\
\sup_{Q_T} |p(z)-q(z)| < \frac{2}{N+2} \quad \text{if} \ \beta =1.
\end{cases}
\end{split}
\end{equation}
We distinguish between the cases $\beta=1$ and $\beta\in (0,1)$ because their study  requires different techniques. On the free term, we assume that
\begin{equation}\label{freeterm:cond}
    \text{$f \in L^\sigma(Q_T)$ for $\sigma > 2$}.
\end{equation}
Furthermore, for a given $r \geq 0$, we assume that 
\begin{equation}\label{initial:inte:cond}
    \int_{\Omega} \left(a(x,0) |\nabla u_0|^{p(x,0) + r} + b(x,0) |\nabla u_0|^{q(x,0) + r}\right) ~dx < +\infty.
\end{equation}
and the modulating coefficient $a(z)$, $b(z)$ satisfy one of the following conditions:
\begin{equation}
\label{eq:reg-data}
\begin{split}
{\rm (a)} \quad & \beta=1,\quad \text{$a$, $b$ are Lipschitz
continuous in $Q_T$ with the Lipschitz constant $L_{a,b}$},
\\
{\rm (b)} \quad & \beta <1, \quad D_i a, \,D_ib,\,a_t,\, b_t\in L^{d}(Q_T) \quad \text{with} \quad d> 2 + \frac{N+2}{2(1-\beta)} \left(\max\{\overline{s}^+, 2(\overline{s}^+-1)\} +r\right)
\end{split}
\end{equation}
In this work, we investigate the existence and global regularity of strong solutions to problem \eqref{eq:main} that inherit the initial data's integrability, in the spirit of Calder\'on-Zygmund regularity theory. In particular, we analyze how the gap condition \eqref{eq:balance-p-q} on the variable exponents $p(z)$ and $q(z)$, together with the integrability assumptions on the datum $f$ in \eqref{freeterm:cond}, governs the propagation of the integrability condition \eqref{initial:inte:cond} from the initial datum $u_0$ to the corresponding strong solution of the irregular double-phase evolution problem \eqref{eq:main}. In addition, we establish self-improving regularity properties of the constructed solution, including higher integrability and enhanced second-order regularity estimates.

The purpose of the present paper is to address these issues. Our analysis combines a regularization procedure with global higher integrability estimates and global second-order regularity estimates. Firstly, we construct a family of classical solutions to suitably regularized problems. The bulk of the work consists of establishing estimates independent of the regularization parameters, which yield global higher integrability of the gradients, transfer of integrability from the initial datum to the solution, and global second-order regularity. These estimates allow one to pass to the limit and obtain a strong solution to the original problem that exhibits the same properties. The approach to these issues follows a strategy similar to that of the paper \cite{Arora-Shmarev-JGA-2026}, but its technical implementation differs due to the choice of the main function space.

\section{Assumptions and results}
In this section, we recall the definition of the function spaces, list the assumptions on the data, and formulate the main results of the work. For a detailed presentation of the function spaces used throughout the text, we refer to \cite{DHHR-2011,HH-2019}.

\subsection{Function spaces}
\label{subsec:spaces}  The solution to problem \eqref{eq:main} will be sought as an element of the variable
Lebesgue and Sobolev spaces. Whereas the generalized Musielak-Orlicz spaces provide a natural analytical framework for double-phase problems, the approach based on the regularization of the equation and the data allows one to reduce the study of the regularized problem with smooth data to dealing with elements
of the variable Lebesgue and Sobolev spaces.
\vspace{0.1cm}\\
\textbf{1) Variable Lebesgue and Sobolev spaces:} Given a function $p:\Omega\mapsto (1,\infty)$, define the function $\rho_p(u)$ (the modular)
\[
\rho_p(u)=\int_{\Omega}|u|^{p(x)}\,dx
\]
and denote by $L^{p(\cdot)}(u)$ the collection of all functions from $L^1(\Omega)$ with $\rho_{p(\cdot)}(u)<\infty$. The set $L^{p(\cdot)}(\Omega)$ equipped with the norm
\[
\|u\|_{p(\cdot),\Omega}=\inf\left\{\lambda>0:\;\rho_{p(\cdot)}\left(\dfrac{u}{\lambda}\right)\leq 1\right\}
\]
is a Banach space. The variable Sobolev space $W^{1,p(\cdot)}_0(\Omega)$ is defined as the set
\[
W^{1,p(\cdot)}_0(\Omega):=\left\{u\in L^{p(\cdot)}(\Omega)\cap W^{1,1}_0(\Omega):\,|\nabla u|\in L^{p(\cdot)}(\Omega)\right\},\quad \|u\|_{W^{1,p(\cdot)}_0(\Omega)}=\|u\|_{p(\cdot)(\Omega)}+\|\nabla u\|_{p(\cdot)(\Omega)}.
\]
Since the functions from $W^{1,p(\cdot)}_0(\Omega)$ with Lipschitz-continuous $p(\cdot)$ satisfy the Poincar\'e inequality,
\[
\|u\|_{p(\cdot)(\Omega)}\leq C\|\nabla u\|_{p(\cdot)(\Omega)},\qquad C=C(p,N,\Omega),
\]
then $\|\nabla u\|_{p(\cdot)(\Omega)}$ is an equivalent norm of $W^{1,p(\cdot)}_0(\Omega)$. The spaces $L^{p(\cdot)}(\Omega)$ and $W^{1,p(\cdot)}(\Omega)$ are reflexive Banach spaces. For the functions defined on the cylinder $Q_T=\Omega\times (0,T)$, we introduce the space
\[
\mathbb{W}_{p(\cdot,\cdot)}(Q_T)=\left\{u: \,u\in L^2(Q_T), |\nabla u|^{p(x,t)}\in L^1(Q_T),\; \text{$u(\cdot,t)=0$ on $\partial\Omega$ for a.e $t\in (0,T)$}\right\},
\]
with the norm
\[
\|u\|_{\mathbb{W}_p}=\|u\|_{2,Q_T}+\|\nabla u\|_{p(\cdot,\cdot),Q_T}.
\]

\textbf{2)  Musielak-Orlicz-Sobolev spaces:} Given functions $a, b:\Omega\mapsto [0,\infty)$, and $p, q:\Omega\mapsto (1,\infty)$, we introduce the $\mathcal{N}$-function $\psi(x,|\xi|)$ and the modular $\mathcal{H}_{p,q}$:
\begin{equation}
\label{eq:psi}
\psi(x,|\xi|)=a(x)|\xi|^{p(x)}+b(x)|\xi|^{q(x)},\qquad \mathcal{H}_{p,q}(u)=\int_{\Omega}\psi(x,|u(x)|)\,dx. 
\end{equation}
Then, the set
\[
L^\psi(\Omega)=\left\{u\in L^1(\Omega):\,\mathcal{H}_{p,q}(u)<\infty\right\}
\]
is the Musielak-Orlicz space. The set $L^\psi(\Omega)$ equipped with the norm
\[
\|u\|_{\psi,\Omega}=\inf\left\{\lambda>0: \mathcal{H}_{p,q}\left(\dfrac{u}{\lambda}\right)\leq 1\right\}
\]
is a Banach space. The Musielak-Orlicz-Sobolev space $\mathcal{W}_0(\Omega)$ is defined as the closure of $C_c^{\infty}(\Omega)$ (smooth functions with compact support in $\Omega$) with respect to the norm
\[
\|u\|_{\mathcal{W}_0(\Omega)}=\|\nabla u\|_{\psi,\Omega}.
\]
%
By $\mathcal{W}_r(\Omega)$, $r\geq 0$, we denote the space generated by the $\mathcal{N}$-function
\[
\mathcal{F}((x,0),|\xi|)|\xi|^{r+2}\equiv a(x,0)|\xi|^{p(x,0)+r}+b(x,0)|\xi|^{q(x,0)+r}.
\]
We will consider problem \eqref{eq:main} with the initial data in the space $\mathcal{W}_r(\Omega)$.
By definition, for every $u_0\in \mathcal{W}_r(\Omega)$ there exists a sequence $\{u_{0m}\}$ such that $u_{0m}\in C_c^{\infty}(\Omega)$ and $u_{0m}\to u_0$ in $\mathcal{W}_r(\Omega)$.

Let $a,b,p,q$ be the coefficients and exponents defined on the cylinder $Q_T$. Assume that $a,b$ satisfy \eqref{eq:coeff}, $p,q$ satisfy \eqref{eq:balance-p-q}, and let $r\geq 0$ be a given number. Define
\[
\mathcal{V}_r(Q_T)=L^\infty(0,T;L^2(\Omega))\cap \left\{u\in L^1(0,T;W^{1,1}_0(\Omega)):\;a(z)|\nabla u|^{p(z)+r}+b(z)|\nabla u|^{q(z)+r}\in L^1(Q_T)\right\}.
\]
By Young's inequality and \eqref{eq:coeff}, for all $z\in \overline{Q}_T$ and $\xi\in \mathbb{R}^N$
\begin{equation}
\label{eq:alpha}
\alpha |\xi|^{\underline{s}(z)}\leq  \left(a(z)+b(z)\right)|\xi|^{\underline{s}(z)}\leq a(z)|\xi|^{p(z)}+b(z)|\xi|^{q(z)}+(a(z)+b(z)),
\end{equation}
therefore
\[
\mathcal{W}_0(\Omega)\subset W^{1,\underline{s}(\cdot,0)}_0(\Omega),\qquad
\text{$\mathcal{V}_r(Q_T)\subset\mathbb{W}_{\underline{s}(\cdot)+r}(Q_T)$ for $r\geq 0$}.
\]

\subsection{Main results}
We begin by introducing the notion of strong solution.
\begin{definition}
\label{def:sol}
A function $u$ is called a strong solution of problem \eqref{eq:main} if
\begin{itemize}
\item[{\rm (i)}] $u\in \mathcal{V}_0(Q_T)$, $u_t\in L^2(\Omega)$,
\item[{\rm (ii)}] for every $\phi\in \mathcal{V}_0(Q_T)$
\begin{equation}
\label{eq:def-1}
\int_{Q_T}\left(u_t\phi+\mathcal{F}(z,\nabla u)\nabla u\cdot \nabla \phi\right)\,dz= \int_{Q_T}f \phi\,dz ,
\end{equation}
\item[{\rm (iii)}] for every $\phi\in L^2(\Omega)$
\[
\int_{\Omega}(u(x,t)-u_0(x))\phi(x)\,dx\to 0\quad \text{as $t\to 0$}.
\]
\end{itemize}
\end{definition}

Denote
\[
K(N, \sigma, p, q):= (\sigma-2) \max\left\{\frac{N}{N+2-\sigma} \left(\underline{s}^--1) + \frac{2}{N(N+2)}\right), \underline{s}^- -1 + \frac{2}{N+2}\right\}
\]
\begin{theorem}
\label{th:1}
Assume that the exponents $p$, $q$, and the coefficients $a$, $b$ satisfy conditions \eqref{eq:coeff}, \eqref{eq:balance-p-q}, \eqref{initial:inte:cond}, and  \eqref{eq:reg-data}. For every $f\in L^{\sigma}(Q_T)$ with $\sigma> 2$ and $u_0$ satisfying \eqref{initial:inte:cond} for
\begin{equation}\label{bound-r-sigma}
\begin{cases}
    0 \leq r \leq K(N, \sigma, p, q) & \text{when} \ \sigma \in (2, N+2),\\
    0 \leq r < +\infty & \text{when} \ \sigma = N+2.
\end{cases}
\end{equation}
problem \eqref{eq:main} has a unique strong solution $u\in \mathcal{V}_r(Q_T)$.
\end{theorem}

\begin{theorem}
\label{th:2}
Under the conditions of Theorem \ref{th:1}
\[
u\in L^\infty(0,T;\mathcal{W}_r(\Omega)),\qquad \text{$|\nabla u|\in L^{\underline{s}(\cdot)+r+s}(Q_T)$ with every $s\in \left(0,\frac{4}{N+2}\right)$}
\]
and
\begin{equation}
\label{eq:th-2-est}
\begin{split}
\operatorname{ess}\sup_{(0,T)} & \|u(t)\|_{2,\Omega}^2 + \|u_{ t}\|_{2,Q_T}^2+\frac{1}{r+\overline{s}^+}\operatorname{ess}\sup_{(0,T)} \int_{\Omega}|\nabla u|^{r+2} \mathcal{F}(z,\nabla u)\,dx
\\
&
+
\int_{Q_T}|\nabla u|^{2(\underline{s}(z)-1)+r+s}\,dz
\leq C+C'\int_{\Omega}|\nabla u_{0}|^{r+2} \mathcal{F}((x,0),\nabla u_{0})\,dx+\|f\|_{\sigma,Q_T}^\sigma
\end{split}
\end{equation}
with constants $C$, $C'$ depending only on the \textbf{data}.
\end{theorem}

\begin{theorem}
\label{th:3}
Let the conditions of Theorem \ref{th:1} be fulfilled and $u(z)$ be the strong solution of problem \eqref{eq:main}. Then
\begin{equation}
\label{eq:reg-nonsmooth-data}
\begin{split}
{\rm (i)} \quad & a(z)|\nabla u|^{p(z)-1+\frac{r}{2}}, \,b(z)|\nabla u|^{q(z)-1+\frac{r}{2}}\in L^2(0,T;W^{1,2}(\Omega)),
\\
{\rm (ii)} \quad & |\nabla u|^{\frac{r}{2}}\mathcal{F}(z,\nabla u)\nabla u\in L^{2}(0,T;W^{1,2}(\Omega))^N,
\end{split}
\end{equation}
and the corresponding norms are bounded by a constant depending only on the \textbf{data}.
\end{theorem}
\begin{remark}
Note that the gradient of the solution enjoys the global higher integrability characterized by two independent parameters, \(r\) and \(s\). While the parameter \(r\) is inherited from the integrability of the initial datum, the parameter \(s\) is generated through interpolation arguments.
\end{remark}
\begin{remark}
    Note that the upper bound $K(N, \sigma, p, q)$ of the parameter $r$ in the case $\sigma \in (2, N+2)$ has the following asymptotics:
    \begin{itemize}
        \item $K(N, \sigma, p, q) \to 0$ as $\sigma \to 2^+$ and $K(N, \sigma, p, q) \to +\infty$ as $\sigma \to (N+2)^-$
        \item As $\sigma \to 2^+$, the higher integrability estimate and second order estimates in Theorem \ref{th:2} and \ref{th:3} coincides with the results in \cite{Arora-Shmarev-JMAA-2025, RACSAM-2023, Ar-Sh-2024}.
    \end{itemize}
\end{remark}

For the convenience of presentation, by \textbf{data} we denote the set of quantities that characterize the structure of equation \eqref{eq:main} and the regularity of $f$, $u_0$, $p$, $q$, $a$, $b$:

\[
\textbf{data}=\left\{r,p^\pm,q^\pm,L_{p,q},N,a^\pm,b^\pm,\beta,\|\nabla a\|_{d,Q_T},\|\nabla b\|_{d,Q_T},\|a_t\|_{d,Q_T},\| b_t\|_{d,Q_T},\|f\|_{\sigma,Q_T},\|u_0\|_{\mathcal{W}_r(\Omega)}\right\},
\]
where $d=\infty$ if $\beta=1$.

\section{Auxiliary propositions}
In this section, we collect several known results and adapt them to the problem we are interested in, derive the interpolation inequalities, and estimate the boundary integrals that appear after applying the Green formula to the regularized nonlinear flux function.

\subsection{Regularized problem}
Consider the problem with the regularized flux and smooth data
\begin{equation}
\label{eq:main-reg}
\begin{split}
& u_t-\operatorname{div}\mathcal{F}_{\epsilon}(z,\nabla u)\nabla u=f, \quad z=(x,t)\in Q_T,
\\
& \text{$u=0$ on $\partial\Omega\times (0,T)$},\qquad \text{$u(x,0)=u_0(x)$ in $\Omega$},
\end{split}
\end{equation}
where
\[
\begin{split}
& \mathcal{F}_\epsilon(z,\nabla u)=a_\epsilon w_{\epsilon}^{\frac{p(z)-2}{2}}+b_\epsilon w_{\epsilon}^{\frac{q(z)-2}{2}}, \qquad
w_\epsilon=\epsilon^2+|\nabla u|^2,
\\
& a_\epsilon=a(z)+\epsilon,\quad b_{\epsilon}=b(z)+\epsilon\quad \text{with $\epsilon \in (0,1)$}.
\end{split}
\]
The existence of classical solution of the problem \eqref{eq:main-reg} is given by the following lemma:
\begin{lemma}[Lemma 3.5, \cite{Arora-Shmarev-JGA-2026}]
\label{le:class-sol-existence}
Let the data of problem \eqref{eq:main-reg} satisfy the following conditions: $\partial\Omega\in C^{2+\gamma}$, $\gamma\in (0,1)$, $u_0\in C^\infty(\overline\Omega)$, $\operatorname{supp} u_0\Subset \Omega$, $f\in C^{\infty}(Q_T)$ with $\operatorname{supp} f(\cdot,t)\Subset \Omega$ for every $t\in (0,T)$, $a,b\in C^{\infty}(\overline Q_T)$, $p,q\in C^{\infty}(\overline Q_T)$. If $p^-, q^->  \frac{2N}{N+2}$, then for every $\epsilon \in (0,1)$ problem \eqref{eq:main-reg} admits a classical solution $u\in C^{2+\delta,1+\frac{\delta}{2}}(\overline Q_T)$ with $D_iu\in C^{2+\delta,1+\frac{\delta}{2}}_{loc}(Q_T)$ with some $\delta\in (0,1)$.
\end{lemma}

\subsection{Interpolation inequalities}
Accept the notation

\[
r^\sharp=\dfrac{4}{N+2},
\]
which will be repeatedly used in the rest of the paper.

\begin{lemma}[Corollary 4.1, \cite{Arora-Shmarev-JGA-2026}]
\label{le:d-r}
Let $\partial \Omega \in C^2$. Assume that $p$, $q$, $a$, $b$ are functions of $x \in \Omega$ and satisfy condition \eqref{eq:balance-p-q}, $a, b \in L^\infty(\Omega)$, $a(x) + b(x) \geq \alpha$ a.e. in $\Omega$, and

\begin{equation}
\label{eq:d-1}
\text{$|\nabla a|, \ |\nabla b| \in L^d(\Omega)$ \quad with $d > 2+\frac{N+2}{2}\left(\overline{s}^++r\right)$ and $r\geq 0$.}
\end{equation}
Then, for every $\delta>0$, $s\in (0,r^\sharp)$, and every $u\in C^{2}(\overline \Omega)$

\begin{equation}
\label{eq:d-r-ell}
\begin{split}
\int_\Omega w_\epsilon^{\frac{\underline{s}(z)-2 + r +s}{2}}|\nabla u|^2\,dx
&
\leq \delta \int_{\Omega} w_\epsilon^\frac{r}{2} \mathcal{F}_{\epsilon}(z,\nabla u)\vert u_{xx}\vert ^2\,dx +C
\end{split}
\end{equation}
with a constant $C$ depending on $\alpha$, $N$, $p^\pm$, $q^\pm$, $s$, $r$, $\delta$, $\|u\|_{2,\Omega}$.
\end{lemma}

\begin{lemma}
\label{le:d-r-mod}
Let in the conditions of Lemma \ref{le:d-r}

\begin{equation}
\label{eq:d-parab}
\underline{s}^->\dfrac{2(N+1)}{N+2}\quad \text{and}\quad d>2+\dfrac{N+2}{2}\left(2(\overline{s}^+-1)+r\right)\quad \text{with $r\geq 0$}.
\end{equation}
For every $\delta>0$, $s\in (0,r^\sharp)$, and every $u\in C^{2}(\overline \Omega)$

\begin{equation}
\label{eq:d-r-ell-mod}
\begin{split}
\int_\Omega w_\epsilon^{\frac{2(\underline{s}(z)-1)-2+s+r}{2}}|\nabla u|^2\,dx
&
\leq \delta\int_{\Omega}w_{\epsilon}^{\frac{r}{2}}\mathcal{F}^2_{\epsilon}(z,\nabla u)\vert u_{xx}\vert ^2\,dx +C
\end{split}
\end{equation}
with a constant $C$ depending on the same quantities as the constant in \eqref{eq:d-r-ell}.
\end{lemma}

\begin{proof}
Notice that
\begin{equation}
\label{eq:notice}
\notag
\begin{split}
& \alpha^2\leq (a_\epsilon+b_\epsilon)^2\leq 2a_\epsilon^2+2b_\epsilon^2,
\\
& a^2_\epsilon w_\epsilon^{\frac{2(p-2)}{2}} + b^2_\epsilon w_\epsilon^{\frac{2(q-2)}{2}}\leq \mathcal{F}^2_{\epsilon}(z,\nabla u)\leq 2\left( a^2_\epsilon w_\epsilon^{\frac{2(p-2)}{2}} + b^2_\epsilon w_\epsilon^{\frac{2(q-2)}{2}}\right),
\end{split}
\end{equation}
and apply \eqref{eq:d-r-ell} with $p$ substituted by $\widetilde p=2(p-1)$, $q$ by $\widetilde q=2(q-1)$, and $a_\epsilon$, $b_\epsilon$ substituted by $a_\epsilon^2$, $b_\epsilon^2$:

\begin{equation}
\label{eq:inter-rev}
\begin{split}
\int_\Omega w_\epsilon^{\frac{2(\underline{s}(z)-1)-2+s+r}{2}}|\nabla u|^2\,dx &
\leq \delta \int_{\Omega} \left(a_\epsilon^2w_{\epsilon}^{\frac{2(p-2)+r}{2}} + b_\epsilon^2w_{\epsilon}^{\frac{2(q-2)+r}{2}}\right)|u_{xx}|^2\,dx + C
\\
& \leq 2\delta\int_{\Omega}w_{\epsilon}^{\frac{r}{2}}\mathcal{F}^2_{\epsilon}(z,\nabla u)\vert u_{xx}\vert ^2\,dx +C,
\end{split}
\end{equation}
provided that $2(\underline{s}^--1)+r>\frac{2N}{N+2}$ for all $r\geq 0$, which is true if $\underline{s}^->\dfrac{2(N+1)}{N+2}$.
\end{proof}

\begin{corollary}
\label{cor:d-1-cor-mod}
Under the conditions of Lemma \ref{le:d-r-mod}

\begin{equation}
\label{eq:d-ell-cor-mod}
\int_\Omega w_\epsilon^{\frac{2(\underline{s}(z)-1)+s+r}{2}}\,dx
\leq \delta \int_{\Omega}w_\epsilon^{\frac{r}{2}}\mathcal{F}^2_{\epsilon}(x,\nabla u)\vert u_{xx}\vert ^2\,dx +C'
\end{equation}
with any $\delta>0$ and $C'$ depending on $\delta$ and the same quantities as the constant in \eqref{eq:d-r-ell-mod} and $|\Omega|$.
\end{corollary}

\begin{proof}
It is sufficient to notice that for every $\gamma\geq 0$ and $\epsilon\in (0,1)$

\[
w_\epsilon^{\gamma}=\begin{cases}
(\epsilon^2+|\nabla u|^2)^{\gamma} & \text{if $|\nabla u|^2\leq \epsilon^2$}
\\
w_{\epsilon}^{\gamma-1}(\epsilon^2+|\nabla u|^2) & \text{if $|\nabla u|^2\geq \epsilon^2$}
\end{cases}
\leq \begin{cases}
(2\epsilon^2)^\gamma & \text{if $|\nabla u|^2\leq \epsilon^2$}
\\
2w_\epsilon^{\gamma-1}|\nabla u|^2 & \text{if $|\nabla u|^2>\epsilon^2$}
\end{cases}
\leq \left(2+2^\gamma\right)\left(1+w_{\epsilon}^{\gamma-1}|\nabla u|^2\right)
\]
Inequality \eqref{eq:d-ell-cor-mod} follows by integration of $w_\epsilon^\gamma$ over $\Omega$ and application of \eqref{eq:d-r-ell-mod} with $\gamma= \frac{2(\underline{s}(z)-1)+s+r}{2}$.
\end{proof}

\begin{corollary}
\label{cor:inter-extra}
Let the conditions of Lemma \ref{le:d-r-mod} be fulfilled. For every $\theta\in \left(0,r^\sharp\right)$, $r\geq 0$, and $\delta>0$

\begin{equation}
\label{eq:d-ell-cor-mod-1}
\int_\Omega w_\epsilon^{\frac{\overline{s}(x)+r+\theta}{2}}\,dx\leq \delta \int_{\Omega}w_\epsilon^{\frac{r}{2}}\mathcal{F}^2_{\epsilon}(x,\nabla u)\vert u_{xx}\vert ^2\,dx +C
\end{equation}
with a constant $C$ depending on $\theta$, $r$, $\delta$, and the same quantities as the constant in \eqref{eq:d-r-ell-mod}.
\end{corollary}

\begin{proof} Since by \eqref{eq:balance-p-q}

\[
\overline{s}<2(\underline{s}-1)+r^\sharp\quad \Leftarrow \quad \overline{s}- \underline{s}< \frac{2}{N+2} =\frac{2(N+1)}{N+2}-2+\frac{4}{N+2} <\underline{s}-2+r^\sharp,
\]
the conclusion follows from \eqref{eq:d-ell-cor-mod}.
\end{proof}

\begin{corollary}
\label{cor:higher-int-parab}
 Let $\partial \Omega \in C^2$. Assume that $p$, $q$, $a$, $b$ satisfy condition \eqref{eq:balance-p-q}, $a, b \in L^\infty(Q_T)$, $a(z) + b(z) \geq \alpha$ a.e. in $Q_T$, and $|\nabla a|, |\nabla b|\in L^d(Q_T)$ with $d$ satisfuing \eqref{eq:d-parab}. Then

\begin{equation}
\label{eq:int-1-parab}
\int_{Q_T} w_\epsilon^{\frac{2(\underline{s}(z)-1)-2+s+r}{2}}|\nabla u|^2\,dz
\leq \delta\int_{Q_T}w_{\epsilon}^{\frac{r}{2}}\mathcal{F}^2_{\epsilon}(z,\nabla u)\vert u_{xx}\vert ^2\,dz +C,
\end{equation}

\begin{equation}
\label{eq:int-2-parab}
\int_{Q_T} w_\epsilon^{\frac{2(\underline{s}(z)-1)+s+r}{2}}\,dz
\leq \delta\int_{Q_T}w_{\epsilon}^{\frac{r}{2}}\mathcal{F}^2_{\epsilon}(z,\nabla u)\vert u_{xx}\vert ^2\,dz +C
\end{equation}
with any $\delta>0$, $s\in (0,r^\sharp)$, and constants $C$ depending on the same quantities as the constants in \eqref{eq:d-r-ell-mod} and $T$.
\end{corollary}

\subsection{The Green formula}
\begin{lemma}[Proposition 3.1, \cite{Arora-Shmarev-JMAA-2025}]
\label{le:Green-1}
    Let $\partial\Omega\in C^2$ and $\vec \alpha$, $\vec \beta$ be arbitrary vectors with the components
\begin{equation}
\label{eq:vec-reg}
\alpha_i \in C^0(\overline{\Omega})\cap C^1(\Omega), \qquad \beta_i\in C^1(\overline{\Omega})\cap C^2(\Omega).
\end{equation}
Denote by $\vec\nu$ the unit exterior normal to $\partial\Omega$ and represent $\vec \alpha=\alpha_\nu\vec\nu+ \vec\alpha_\tau$, $\vec \beta=\beta_\nu\vec\nu+ \vec\beta_\tau$, where $\alpha_{\nu}=(\vec \alpha,\vec \nu)$, $\beta_\nu=(\vec \beta,\vec \nu)$ are the normal components of $\vec \alpha$, $\vec \beta$, and $\vec\alpha_\tau$, $\vec\beta_{\tau}$ belong in the tangent plane to $\partial\Omega$. If $\vec\alpha_\tau=0$ and $\vec\beta_{\tau}=0$, then
\begin{equation}
\label{eq:by-parts}
\begin{split}
\int_\Omega & \operatorname{div}\vec \alpha \operatorname{div}\vec \beta\,dx  = - \int_{\partial\Omega}\alpha_{\nu}\beta_{\nu}\operatorname{trace}\mathcal{B} \,dS + \int_\Omega \sum_{i,j=1}^N D_j\alpha_i D_i\beta_j\,dx
\end{split}
\end{equation}
where $\vec \nu$ is the exterior unit normal to $\partial \Omega$
\end{lemma}

Given a function $u\in C^2(\overline{\Omega})\cap C^3(\Omega)$ and  $r\geq 0$, consider the vectors

\begin{equation}
\label{eq:p+r}
\begin{split}
& \vec \alpha= \mathcal{F}_\epsilon(x,\nabla u)\nabla u\equiv \left(a_{\epsilon}w_\epsilon^{\frac{p(x)-2}{2}} + b_\epsilon w_\epsilon ^{\frac{q(x)-2}{2}}\right)\nabla u,\qquad \vec \beta=  w_\epsilon^{\frac{r}{2}}\vec{\alpha}.
\end{split}
\end{equation}
By the straightforward computation

\[
\begin{split}
D_j \alpha_i & = \left(a_{\epsilon}w_\epsilon^{\frac{p-2}{2}}+b_\epsilon w_\epsilon^{\frac{q-2}{2}}\right)D^2_{ij}u + \left((p-2)a_{\epsilon}w_{\epsilon}^{\frac{p-2}{2}-1}+(q-2) b_\epsilon w_\epsilon^{\frac{q-2}{2}-1}\right) D_iu \sum_{k=1}^N D_ku D^2_{kj}u
\\
& \quad
+ \frac{1}{2} \left(a_{\epsilon}w_\epsilon^{\frac{p-2}{2}}D_iuD_jp+b_\epsilon w_\epsilon ^{\frac{q-2}{2}} D_iuD_jq\right)\ln w_\epsilon +w_\epsilon^{\frac{p-2}{2}}D_iD_ja  + w_\epsilon^{\frac{q-2}{2}} D_i u D_j b,
\end{split}
\]

\[
\begin{split}
D_i \beta_j & = w_\epsilon^{\frac{r}{2}}\left[\left(a_{\epsilon}w_\epsilon^{\frac{p-2}{2}}+b_\epsilon w_\epsilon^{\frac{q-2}{2}}\right)D^2_{ij}u + \left((p-2)a_{\epsilon}w_{\epsilon}^{\frac{p-2}{2}-1}+(q-2) b_\epsilon w_\epsilon^{\frac{q-2}{2}-1}\right) D_ju \sum_{k=1}^N D_ku D^2_{ki}u\right.
\\
& \quad
\left.+ \frac{1}{2} \left(a_{\epsilon}w_\epsilon^{\frac{p-2}{2}}D_juD_ip+b_\epsilon w_\epsilon ^{\frac{q-2}{2}} D_juD_iq\right)\ln w_\epsilon + w_\epsilon^{\frac{p-2}{2}}D_juD_ia  + w_\epsilon^{\frac{q-2}{2}} D_j u D_i b
\right]
\\
& \quad + \left.r w_{\epsilon}^{\frac{r}{2}-1} \left(a_{\epsilon}w_\epsilon^{\frac{p-2}{2}}+b_\epsilon w_\epsilon^{\frac{q-2}{2}}\right) D_ju\sum_{k=1}^N D_iu D^2_{ki}u D_ju\right)\sum_{k=1}^N \eta_i \eta_k D^2_{kj}u
\end{split}
\]

For every $\epsilon>0$,
the inclusions $\beta_i\in C^2({\Omega})$ hold if $u\in C^3({\Omega})$ and $p,q,a,b \in C^2(\Omega)$. Since
\[
\begin{split}
&
\alpha_\nu=\vec{\alpha}\cdot \vec{\nu}=\vec{\alpha}\cdot \dfrac{\nabla u}{|\nabla u|}= \left(a_{\epsilon}w_\epsilon^{\frac{p(x)-2}{2}} + b_\epsilon w_\epsilon^{\frac{q(x)-2}{2}}\right)|\nabla u|,\qquad \vec \alpha_\tau=\vec \alpha-\alpha_\nu \vec\nu=0,
\\
& \beta_\nu=
\vec{\beta}\cdot \dfrac{\nabla u}{|\nabla u|}= w_\epsilon^{\frac{r}{2}}\alpha_\nu ,\qquad \vec \beta_\tau=\vec \beta-\beta_\nu \vec\nu=0,
\end{split}
\]
the boundary integral in \eqref{eq:by-parts} transforms into

\[
\begin{split}
\int_{\partial \Omega}  & \left(\alpha_\nu \operatorname{div}\vec \beta - ((\vec \alpha\cdot \nabla)\vec \beta)\cdot \vec \nu\right)\,dS
= -\int_{\partial\Omega} w_\epsilon^{\frac{r}{2}}\left(a_{\epsilon}w_\epsilon^{\frac{p-2}{2}} +  b_\epsilon w_\epsilon^{\frac{q-2}{2}} \right)^2|\nabla u|^2\operatorname{trace}\mathcal{B}\,dS
\end{split}
\]
and formula \eqref{eq:by-parts}  becomes

\begin{equation}
\label{eq:double-final-e-start}
\begin{split}
\int_{\Omega} \operatorname{div} & \left(\mathcal{F}_\epsilon(z,w_\epsilon)\nabla u\right)
\operatorname{div}\left( w_\epsilon ^{\frac{r}{2}}\mathcal{F}_{\epsilon}(z,\nabla u)\nabla u\right)\,dx
= -\int_{\partial\Omega} w_\epsilon^{\frac{r}{2}}\left(a_{\epsilon}w_\epsilon^{\frac{p-2}{2}} +  b_\epsilon w_\epsilon^{\frac{q-2}{2}} \right)^2|\nabla u|^2\operatorname{trace}\mathcal{B}\,dS
\\
& \qquad + \sum_{i,j=1}^N\int_{\Omega}D_{i}\left( \left(a_{\epsilon}w_\epsilon^{\frac{p-2}{2}} + b_\epsilon  w_\epsilon^{\frac{q-2}{2}} \right)D_{j}u \right) D_{j}\left(w_\epsilon^{\frac{r}{2}}\left(a_{\epsilon}w_\epsilon^{\frac{p-2}{2}} + b_\epsilon w_\epsilon ^{\frac{q-2}{2}}\right)D_{i}u\right)\,dx.
\end{split}
\end{equation}

Since $\operatorname{trace}\mathcal{B}\leq 0$ for smooth convex domains, for such domain the boundary integral in \eqref{eq:double-final-e-start} is nonpositive. In the general case, the integral over $\partial\Omega$ is estimated by virtue of the following assertion.

\begin{lemma}
\label{le:trace-main}
Let in the conditions of Lemma \ref{le:d-r}, $u\in C^2(\overline{\Omega})\cap C^3(\Omega)$, $u=0$ on $\partial\Omega$. If $r\geq 0$,
\begin{equation}
\label{eq:d-boundary}
\begin{split}
& |\nabla a|, |\nabla b|\in L^d(\Omega)\quad \text{with $\displaystyle d > 1+\frac{N+2}{4(1-\beta)}\left(2(\overline{s}^+-1)+ r\right)$},
\\
& \text{either $d=\infty$ and in the balance condition \eqref{eq:balance-p-q} $\beta=1$, or $d$ is finite and $\beta\in (0,1)$.}
\end{split}
\end{equation}
Then for every $\lambda\in (0,1)$
\begin{equation}
\label{eq:trace-3} \int_{\partial \Omega} w_{\epsilon}^{\frac{r}{2}} \mathcal{F}^2_{\epsilon}(x,\nabla u)\vert \nabla u\vert ^{2}\,dS\leq
\lambda \int_{\Omega} w_{\epsilon}^{\frac{r}{2}}\mathcal{F}^2_{\epsilon}(x,\nabla u) \vert u_{xx}\vert ^2\,dx+ C
\end{equation}
with a constant
$C=C'(\lambda,
\underline{s}^-, \overline{s}^+, N, L_{p,q}, \alpha, r, \|u\|_{2,\Omega}) + C''\left( \|\nabla a\|_{d,\Omega}, \|\nabla b\|_{d,\Omega}\right)
$.
\end{lemma}
\begin{proof}
By \cite[Lemma 1.5.1.9]{Grisvard-2011} there exist a constant $\gamma>0$ and a function $\vec \mu\in C^{\infty}(\overline{\Omega})^N$ such that $\vec \mu\cdot\nu \geq 2\gamma>0$ on $\partial\Omega$. Then
\[
\begin{split}
\gamma\int_{\partial\Omega} & w_\epsilon^{\frac{r}{2}}\mathcal{F}^2_\epsilon {x,\nabla u}|\nabla u|^2\,dS
=\gamma \int_{\partial\Omega} \left(a_\epsilon w_\epsilon^{\frac{p+\frac{r}{2}-2}{2}} + b_\epsilon w_\epsilon^{\frac{q+\frac{r}{2}-2}{2}} \right)^2|\nabla u|^2\,dS
\\
& \leq 2\gamma \int_{\partial\Omega} \left(a^2_\epsilon w_\epsilon^{p+\frac{r}{2}-2} + b^2_\epsilon w_\epsilon^{q+\frac{r}{2}-2} \right)|\nabla u|^2\,dS
\\
& \leq 2\gamma\int_{\partial\Omega} \left(a^2_\epsilon w_\epsilon^{p+\frac{r}{2}-1} + b^2_\epsilon w_\epsilon^{q+\frac{r}{2}-1} \right)\,dS
\\
&
\leq \int_{\Omega}\operatorname{div}\left( \left(a^2_\epsilon w_\epsilon^{p+\frac{r}{2}-1} + b^2_\epsilon w_\epsilon^{q+\frac{r}{2}-1} \right)\vec \mu\right)\,dx
 \\
 &
 = \int_{\Omega} \vec \mu \cdot\nabla \left(a^2_\epsilon w_\epsilon^{p+\frac{r}{2}-1} + b^2_\epsilon w_\epsilon^{q+\frac{r}{2}-1} \right)\,dx
 + \int_{\Omega}(\operatorname{div}\vec \mu ) \left(a^2_\epsilon w_\epsilon^{p+\frac{r}{2}-1} + b^2_\epsilon w_\epsilon^{q+\frac{r}{2}-1} \right)\,dx
\\
& \leq C_1 \int_{\Omega} \left(a^2_\epsilon w_\epsilon^{p+\frac{r}{2}-1} + b^2_\epsilon w_\epsilon^{q+\frac{r}{2}-1} \right)\,dx\\
& \quad + C_2 
\int_{\Omega} \left(a^2_\epsilon w_\epsilon^{p+\frac{r}{2}-2} + b^2_\epsilon w_\epsilon^{q+\frac{r}{2}-2\quad } \right) |\nabla u||u_{xx}| \,dx
\\
&
\quad + C_3 \int_{\Omega} \left(a^2_\epsilon w_\epsilon^{p+\frac{r}{2}-1} |\nabla p| + b^2_\epsilon w_\epsilon^{q+\frac{r}{2}-1} |\nabla q|\right) |\ln w_\epsilon|\,dx
\\
& \quad
+ C_4\int_{\Omega} \left(w_\epsilon^{p+\frac{r}{2}-1} |\nabla a|+ w_\epsilon^{q+\frac{r}{2}-1} |\nabla b|\right)\,dx
\equiv C_1\mathcal{I}_1 +C_2\mathcal{I}_2 + C_3\mathcal{I}_3 + C_4 \mathcal{I}_4.
\end{split}
\]
The integrals $\mathcal{I}_i$ are estimated separately. The estimate on $\mathcal{I}_1$ follows from Corollaries \ref{cor:d-1-cor-mod}, \ref{cor:inter-extra} that

\[
\mathcal{I}_1\leq C\left(1+\int_\Omega w_\epsilon^{\overline{s}+\frac{r}{2}-1}\,dx\right)\leq C'+\delta \int_\Omega w_\epsilon^{\frac{r}{2}}\mathcal{F}_{\epsilon}^2(x,\nabla u)\,dx.
\]
By Young's inequality, for every $\delta>0$
\[
\begin{split}
\mathcal{I}_2 & \leq \int_\Omega \left(a_\epsilon w_\epsilon^{\frac{p+\frac{r}{2}-1}{2}}\right) \left(a_\epsilon w_{\epsilon}^{\frac{p+\frac{r}{2}-2}{2}}|u_{xx}|\right)\,dx +\int_\Omega \left(b_\epsilon w_\epsilon^{\frac{q+\frac{r}{2}-1}{2}}\right) \left(b_\epsilon w_{\epsilon}^{\frac{q+\frac{r}{2}-2}{2}}|u_{xx}|\right)\,dx
\\
& \leq C(a^++b^+)\int_{\Omega}\left(w_{\epsilon}^{p+\frac{r}{2}-1} + w_{\epsilon}^{q+\frac{r}{2}-1}\right) + \delta\int_{\Omega}w_\epsilon^{\frac{r}{2}}\mathcal{F}_{\epsilon}^2(x,\nabla u)|u_{xx}|^2\,dx
\\
& \leq C'\mathcal{I}_1 + \delta\int_{\Omega}w_\epsilon^{\frac{r}{2}}\mathcal{F}_{\epsilon}^2(x,\nabla u)|u_{xx}|^2\,dx.
\end{split}
\]
To estimate $\mathcal{I}_3$ we employ the following elementary inequality: for every $\theta, \gamma>0$ there is a constant $C=C(\theta, \gamma)$ such that

\begin{equation}
\label{eq:elem-ln}
|x^\theta  \ln x|\leq C(1+x^{\theta + \gamma})\quad \text{for $x>0$}.
\end{equation}
Assumption \eqref{eq:balance-p-q} allows one to choose $\gamma>0$ so small that

\[
\overline{s} +\frac{r}{2}-1+\gamma<\underline{s}-1+\frac{r}{2}+\dfrac{2}{N+2}.
\]
Using \eqref{eq:elem-ln} and \eqref{eq:d-ell-cor-mod} we estimate

\[
\begin{split}
\mathcal{I}_3 & \leq C''\left(1+\mathcal{I}_1+\int_\Omega w_\epsilon^{\overline{s} +\frac{r}{2}-1+\gamma}\,dx\right)
\\
&
\leq C\left(1+\mathcal{I}_1\right)+\delta \int_{\Omega}w_\epsilon^{\frac{r}{2}}\mathcal{F}_{\epsilon}^2(x,\nabla u)|u_{xx}|^2\,dx
\end{split}
\]
with an arbitrary $\delta>0$ and $C''$ depending on $\delta$. The same estimate with $\theta=0$ holds for $\mathcal{I}_1$.

Let $d$ be finite and $\beta\in (0,1)$. By Young's inequality

\[
\begin{split}
\mathcal{I}_4 & \leq \int_{\Omega}\left(|\nabla a|^d+|\nabla b|^d\right)\,dx + \int_\Omega \left(w_\epsilon^{\frac{d}{d-1}\left(p+\frac{r}{2}-1\right)} +w_\epsilon^{\frac{d}{d-1}\left(q+\frac{r}{2}-1\right)}\right)\,dx
\\
& \leq C+\|\nabla a\|^d_{d,\Omega}+\|\nabla b\|^d_{d,\Omega}+\int_{\Omega} w_\epsilon^{\frac{d}{d-1}\left(\overline{s}+\frac{r}{2}-1\right)}\,dx
\end{split}
\]
The last term of this inequality is estimated by \eqref{eq:d-ell-cor-mod}, provided that

\[
\frac{d}{d-1}\left(\overline{s}+\frac{r}{2}-1\right) <\underline{s}-1+\frac{r}{2}+\frac{2}{N+2}\qquad \Leftrightarrow \qquad (\overline{s}-\underline{s}) +\dfrac{1}{d-1}\left(\overline{s}+\frac{r}{2}-1\right)<\dfrac{2}{N+2}.
\]
The last inequality is fulfilled because of assumption \eqref{eq:balance-p-q} and condition \eqref{eq:d-boundary}:

\[
\begin{split}
(\overline{s}-\underline{s}) & +\dfrac{1}{d-1}\left(\overline{s}+\frac{r}{2}-1\right)<\dfrac{2\beta}{N+2}
+\dfrac{1}{d-1}\left(\overline{s}+\frac{r}{2}-1\right)<\dfrac{2}{N+2}
\\
& \Leftrightarrow \qquad d>1+\dfrac{N+2}{4(1-\beta)}\left(2(\overline{s}-1)+r\right).
\end{split}
\]
It follows that

\[
\mathcal{I}_4\leq C+\delta \int_\Omega w_{\epsilon}^{\frac{r}{2}}\mathcal{F}_\epsilon^2(x,\nabla u)|u_{xx}|^2\,dx
\]
with any $\delta>0$ and $C=C(\delta)$. Gathering the estimates on $\mathcal{I}_i$ we arrive at inequality \eqref{eq:trace-3}. To complete the proof, it is sufficient to notice that in the case $d=\infty$

\[
\mathcal{I}_4\leq \left(\|\nabla a\|_{\infty,\Omega}+\|\nabla a\|_{\infty,\Omega}\right)\left(1+\mathcal{I}_1\right).
\]
\end{proof}

\section{A priori estimates}
\subsection{The time derivative}
The main estimates on the classical solution $u$ and its derivatives follow from the equality obtained by multiplication of equation \eqref{eq:main-reg} by $\operatorname{div} \left(w_\epsilon^{\frac{r}{2}}\mathcal{F}_\epsilon(z,\nabla u)\nabla u\right)$ and integration over $\Omega$:

\begin{equation}
\label{eq:energy-1}
\begin{split}
\int_{\Omega}u_t \operatorname{div} \left(w_\epsilon^{\frac{r}{2}}\mathcal{F}_\epsilon(z,\nabla u)\nabla u\right)\,dx & - \int_\Omega \operatorname{div} \left(w_\epsilon^{\frac{r}{2}}\mathcal{F}_\epsilon(z,\nabla u)\nabla u\right)\operatorname{div} \left(\mathcal{F}_\epsilon(z,\nabla u)\nabla u\right)\,dx
\\
& = \int_\Omega f \operatorname{div} \left(w_\epsilon^{\frac{r}{2}}\mathcal{F}_\epsilon(z,\nabla u)\nabla u\right)\,dx.
\end{split}
\end{equation}

{\bf (a)} The first term on the left-hand side on \eqref{eq:energy-1} can be written as

\[
\begin{split}
-\int_{\Omega} & u_t\operatorname{div} \left(w_\epsilon^{\frac{r}{2}}\mathcal{F}_\epsilon(z,\nabla u)\nabla u\right)\,dx
 =\int_\Omega \left(a_\epsilon w_\epsilon^{\frac{r+p-2}{2}}+b_\epsilon w_\epsilon^{\frac{r+q-2}{2}}\right)(\nabla u\cdot \nabla u_t)\,dx
\\
&
= \frac{1}{2}\int_\Omega \left(a_\epsilon w_\epsilon^{\frac{r+p-2}{2}}+b_\epsilon w_\epsilon^{\frac{r+q-2}{2}}\right)(w_\epsilon)_t\,dx
\\
&
= \dfrac{1}{2}\dfrac{d}{dt}\int_{\Omega}\left(\int_{0}^{w_\epsilon}\left(a_\epsilon s^{\frac{r+p-2}{2}}+b_\epsilon s^{\frac{r+q-2}{2}}\right)\,ds\right)\,dx - \frac{1}{2}\int_{\Omega}\int_{0}^{w_\epsilon} s^{\frac{r}{2}}\left(\mathcal{F}_{\epsilon}((x,t),s)\right)_t\,ds dx
\\
& =  \dfrac{1}{2}\dfrac{d}{dt}\int_{\Omega}\left(\int_{0}^{w_\epsilon}s^{\frac{r}{2}}\mathcal{F}_{\epsilon}((x,t), s)\,ds\right)\,dx - \frac{1}{2}\int_{\Omega}\int_{0}^{w_\epsilon} s^{\frac{r}{2}}\left(\mathcal{F}_{\epsilon}((x,t),s)\right)_t\,ds dx.
\end{split}
\]
Using the easily verified formulas
\[
\begin{split}
\left(\mathcal{F}_\epsilon((x,t),s)\right)_t & =
a_t s^{\frac{p-2}{2}} + b_ts^{\frac{q-2}{2}}
 + \frac{1}{2}a_\epsilon s^{\frac{p-2}{2}}\ln s p_t+ \frac{1}{2}b_\epsilon s^{\frac{q-2}{2}}\ln sq_t,
\end{split}
\]

\[
\begin{split}
\frac{1}{2}\int_{0}^{w_\epsilon}s^{\frac{r}{2}} &  \mathcal{F}_\epsilon(z,s)\,ds = \frac{1}{2}\int_0^{w_\epsilon} \left(a_\epsilon s^{\frac{p+r-2}{2}}+b_{\epsilon}s^{\frac{q+r-2}{2}}\right)\,ds
= \frac{a_{\epsilon}}{p+r}w_\epsilon^{\frac{p+r}{2}} + \frac{b_{\epsilon}}{q+r}w_\epsilon^{\frac{q+r}{2}},
\end{split}
\]

\[
\begin{split}
\int_{0}^{w_\epsilon} s^{\frac{r}{2}} & \left(\mathcal{F}_{\epsilon}((x,t),s)\right)_t\,ds = \int_{0}^{w_\epsilon}\left(a_t s^{\frac{p+r-2}{2}} + b_ts^{\frac{q+r-2}{2}}
 + \frac{1}{2}a_\epsilon s^{\frac{p+r-2}{2}}\ln s p_t+ \frac{1}{2}b_\epsilon s^{\frac{q+r-2}{2}}\ln sq_t\right) \,ds
 \\
 & = \frac{2a_t}{p+r}w_\epsilon^{\frac{p+r}{2}} + \frac{2b_t}{q+r}w_\epsilon^{\frac{q+r}{2}} + \dfrac{a_\epsilon p_t}{p+r}w_\epsilon^{\frac{p+r}{2}}\ln w_\epsilon - \dfrac{a_\epsilon p_t}{p+r}\int_{0}^{w_\epsilon} s^{\frac{p+r-2}{2}}\,ds
 \\
 & \qquad + \dfrac{b_\epsilon q_t}{q+r}w_\epsilon^{\frac{q+r}{2}}\ln w_\epsilon - \dfrac{b_\epsilon q_t}{q+r}\int_{0}^{w_\epsilon} s^{\frac{q+r-2}{2}}\,ds
\\
 & = \frac{2a_t}{p+r}w_\epsilon^{\frac{p+r}{2}} - \dfrac{2a_\epsilon p_t}{(p+r)^2}w_\epsilon^{\frac{p+r}{2}} +  \dfrac{a_\epsilon p_t}{p+r}w_\epsilon^{\frac{p+r}{2}}\ln w_\epsilon
 +\frac{2b_t}{q+r}w_\epsilon^{\frac{q+r}{2}} - \dfrac{2b_\epsilon q_t}{(q+r)^2}w_\epsilon +  \dfrac{b_\epsilon q_t}{q+r}w_\epsilon^{\frac{q+r}{2}}\ln w_\epsilon,
\end{split}
\]
we obtain

\begin{equation}
\label{eq:by-parts-t}
\notag
\begin{split}
- \int_{\Omega} &  u_t\operatorname{div} \left(w_\epsilon^{\frac{r}{2}}\mathcal{F}_\epsilon(z,\nabla u)\nabla u\right)\,dx = \dfrac{1}{2}\dfrac{d}{dt}\int_{\Omega}\left(\int_{0}^{w_\epsilon}s^{\frac{r}{2}} \mathcal{F}_\epsilon(z,s)\,ds\right)\,dx
\\
& - \int_\Omega \left(\frac{a_t}{p+r}w_\epsilon^{\frac{p+r}{2}} - \dfrac{a_\epsilon p_t}{(p+r)^2}w_\epsilon^{\frac{p+r}{2}} +  \dfrac{a_\epsilon p_t}{2(p+r)}w_\epsilon^{\frac{p+r}{2}}\ln w_\epsilon\right)\,dx
\\
&
 - \int_\Omega\left(\frac{b_t}{q+r}w_\epsilon^{\frac{q+r}{2}} - \dfrac{b_\epsilon q_t}{(q+r)^2}w_\epsilon^{\frac{q+r}{2}} +  \dfrac{b_\epsilon q_t}{2(q+r)}w_\epsilon^{\frac{q+r}{2}}\ln w_\epsilon\right)\,dx
 \\
 & \equiv \dfrac{d}{dt}\left(\int_\Omega\left(\frac{a_{\epsilon}}{p+r}w_\epsilon^{\frac{p+r}{2}} + \frac{b_{\epsilon}}{q+r}w_\epsilon^{\frac{q+r}{2}}\right)\,dx\right) + \sum_{i=1}^6\mathcal{J}_i.
\end{split}
\end{equation}
The terms $\mathcal{J}_i$ are estimated with the help of \eqref{eq:d-ell-cor-mod-1}, Young's inequality, and inequality \eqref{eq:elem-ln}.
Then
\begin{equation}
\notag
\begin{split}
|\mathcal{J}_2| & + |\mathcal{J}_3| + |\mathcal{J}_5| + |\mathcal{J}_6|\leq C\left(1+\int_{\Omega}w_{\epsilon}^{\frac{\overline{s}+r+\theta}{2}}\,dx\right)
\end{split}
\end{equation}
with any $\theta\in \left(0, \frac{r^\sharp}{2} \right)$ and $C$ depending on $\theta$:

\begin{equation}
\label{eq:time-1}
|\mathcal{J}_2| + |\mathcal{J}_3| + |\mathcal{J}_5| + |\mathcal{J}_6|\leq \delta \int_{\Omega} w_\epsilon^{\frac{r}{2}}\mathcal{F}_\epsilon^2(z,\nabla u)|u_{xx}|^2\,dx+C
\end{equation}
with any $\delta>0$ and a constant $C$ depending on $\delta$ and the same magnitudes as the constant in \eqref{eq:d-ell-cor-mod}. It remains to estimate $\mathcal{J}_1$ and $\mathcal{J}_4$. Consider $\mathcal{J}_1$. By Young's inequality

\[
|\mathcal{J}_1| \leq \int_{\Omega}|a_t|^d\,dx + C\int_{\Omega} w_{\epsilon}^{\frac{d}{d-1}\frac{p+r}{2}}\,dx.
\]
The first term is bounded by assumption. The second one is estimated by virtue of \eqref{eq:d-ell-cor-mod-1}, provided that

\[
\frac{d}{d-1}(\overline{s}+r)=\overline{s}+r +\dfrac{1}{d-1}(\overline{s}+r)
<2(\underline{s}-1)+r +\dfrac{4}{N+2}.
\]
The last inequality can be written in the equivalent form

\[
2-\underline{s}(z)+(\overline{s}(z)-\underline{s}(z)) +\dfrac{1}{d-1}(\overline{s}(z)+r)<\dfrac{4}{N+2}.
\]
The validity of this inequality follows from assumption  \eqref{eq:balance-p-q}, \eqref{eq:reg-data} and the lower bound on $p^-,q^-$:

\[
\begin{split}
2-\underline{s}(z)& +(\overline{s}(z)-\underline{s}(z)) +\dfrac{1}{d-1}(\overline{s}(z)+r)
<
2-\underline{s}(z) + \frac{2\beta}{N+2} + \dfrac{1}{d-1}(\overline{s}(z)+r)
\\
&
<2\left(1-\dfrac{N+1}{N+2}\right)+ \frac{2\beta}{N+2} + \dfrac{1}{d-1}(\overline{s}(z)+r)
\\
&
=\frac{2(1+\beta)}{N+2} +\dfrac{1}{d-1}(\overline{s}(z)+r)<\frac{4}{N+2}.
\end{split}
\]
Solving for $d$ we obtain

\begin{equation}
\notag
d>1+\frac{N+2}{2(1-\beta)}\left(\overline{s}^++r\right),
\end{equation}
which is true by assumption \eqref{eq:reg-data}. It follows that for every $\delta>0$

\[
|\mathcal{J}_1| + |\mathcal{J}_4|\leq \delta \int_\Omega w_\epsilon^{\frac{r}{2}}\mathcal{F}_\epsilon^2(z,\nabla u)\,dx + C\left(1+\|a_t\|_{d,\Omega}^d + \|b_t\|_{d,\Omega}^d\right),
\]
provided that $\underline{s}^->\dfrac{2(N+1)}{N+2}$ and assumption \eqref{eq:reg-data} on $d$ is fulfilled.

Gathering the above estimates we find that for every $\delta>0$

\begin{equation}
\label{eq:time-2}
\begin{split}
-\int_{\Omega} u_t\operatorname{div} \left(w_\epsilon^{\frac{r}{2}}\mathcal{F}_\epsilon(z,\nabla u)\nabla u\right)\,dx & \geq \dfrac{d}{dt}\left(\int_\Omega\left(\frac{a_{\epsilon}}{p+r}w_\epsilon^{\frac{p+r}{2}} + \frac{b_{\epsilon}}{q+r}w_\epsilon^{\frac{q+r}{2}}\right)\,dx\right)
\\
&
- \delta \int_{\Omega} w_\epsilon^{\frac{r}{2}}\mathcal{F}_\epsilon^2(z,\nabla u)|u_{xx}|^2\,dx -C
\end{split}
\end{equation}
with a positive constant $C$ depending on $\delta. $

\medskip

\textbf{(b)} Estimates on the right-hand side of \eqref{eq:energy-1}. By the straightforward differentiation
\[
\begin{split}
\int_{\Omega} & f \operatorname{div}\left(w_\epsilon^{\frac{r}{2}}\mathcal{F}_\epsilon(z,\nabla u)\nabla u\right)\,dx
 =\int_\Omega f w_\epsilon^{\frac{r}{2}}\mathcal{F}_\epsilon(z,\nabla u)\Delta u \,dx
 \\
&
 + \int_\Omega w_\epsilon^{\frac{r}{2}}f\left((p-2)a_{\epsilon}w_{\epsilon}^{\frac{p-2}{2}-1}+(q-2) b_\epsilon w_\epsilon^{\frac{q-2}{2}-1}\right) \sum_{i=1}^ND_iu \sum_{k=1}^N D_ku D^2_{ki}u \,dx
\\
&
+ \frac{1}{2} \int_{\Omega}fw_\epsilon^{\frac{r}{2}}\left(a_{\epsilon}w_\epsilon^{\frac{p-2}{2}}\nabla u\nabla p+b_\epsilon w_\epsilon ^{\frac{q-2}{2}} \nabla u\nabla q\right)\ln w_\epsilon \,dx
\\
&
+ \int_{\Omega}f w_\epsilon^{\frac{r}{2}}\left(w_\epsilon^{\frac{p-2}{2}}\nabla u\nabla a  + w_\epsilon^{\frac{q-2}{2}} \nabla u\nabla b\right)\,dx
\\
& + r\int_\Omega f w_\epsilon^{\frac{r}{2}-1}\mathcal{F}_\epsilon(z,\nabla u)\sum_{i=1}^ND_iu\sum_{k=1}^N D_{k}u D_{ki}^2u\,dx \equiv \sum_{i=1}^5 \mathcal{I}_i.
\end{split}
\]
By Young's inequality

\[
\begin{split}
|\mathcal{I}_1| +|\mathcal{I}_2| +|\mathcal{I}_5| & \leq C\int_\Omega \left(w_\epsilon^{\frac{r}{2}}f^2\right)^{\frac{1}{2}} \left(w_\epsilon^{\frac{r}{2}}\mathcal{F}^2_\epsilon(z,\nabla u)|u_{xx}|^2\right)^{\frac{1}{2}}\,dx
\\
& \leq C_\delta \int_{\Omega}f^2w_\epsilon^{\frac{r}{2}}\,dx+\delta \int_{\Omega}w_\epsilon^{\frac{r}{2}}\mathcal{F}_\epsilon^2(z,\nabla u)|u_{xx}|^2\,dx
\end{split}
\]
with any $\delta\in (0,1)$, whereas

\[
\begin{split}
& |\mathcal{I}_3| \leq C\int_\Omega w_{\epsilon}^{\frac{r}{2}}|f|\left(a_\epsilon w_{\epsilon}^{\frac{p-1}{2}}|\ln w_{\epsilon}| + b_\epsilon w_{\epsilon}^{\frac{q-1}{2}}|\ln w_{\epsilon}|\right)\,dx\equiv \mathcal{I}_{3,p}+ \mathcal{I}_{3,q},
\\
& |\mathcal{I}_4| \leq \int_{\Omega} w_{\epsilon}^{\frac{r}{2}}|f|\left(w_{\epsilon}^{\frac{p-1}{2}}|\nabla a| + w_{\epsilon}^{\frac{q-1}{2}}|\nabla b|\right)\,dx\equiv \mathcal{I}_{4,p}+ \mathcal{I}_{4,q},
\end{split}
\]

The estimates on $\mathcal{I}_{3,p}$ and $\mathcal{I}_{3,q}$ are similar and follow from \eqref{eq:d-ell-cor-mod} and \eqref{eq:elem-ln}: for every $\theta\in (0,r^\sharp)$ there is a constant $C_\theta$ such that

\[
\begin{split}
|\mathcal{I}_{3,p}| & \leq \int_{\Omega}f^2w_\epsilon^{\frac{r}{2}}\,dx + \int_\Omega a_\epsilon^2 w_{\epsilon}^{\frac{2(p-1)+r}{2}}\ln^2w_\epsilon\,dx
\\
& \leq \int_{\Omega}f^2w_\epsilon^{\frac{r}{2}}\,dx + C_\theta\int_{\Omega} w_{\epsilon}^{\frac{2(p-1)+\theta+r}{2}}\,dx +C'
\\
& \leq \delta \int_\Omega w_{\epsilon}^{\frac{r}{2}}\mathcal{F}_\epsilon^{2}(z,\nabla u) |u_{xx}|^2 \,dx +\int_{\Omega}f^2w_\epsilon^{\frac{r}{2}}\,dx + C''
\end{split}
\]
with $C''$ depending on the same quantities as the constant in \eqref{eq:d-r-ell}, $a^+$ and $b^+$. To apply \eqref{eq:d-ell-cor-mod} in estimating $\mathcal{I}_{4,p}$, $\mathcal{I}_{4,q}$ we claim

\[
\begin{split}
& \frac{d}{d-2}(2(\overline{s}-1)+r)<2(\underline{s}-1)+r+\frac{4}{N+2}.
\end{split}
\]
We rewrite this inequality in the equivalent form

\[
\frac{1}{d-2}\left(\overline{s}-1+\frac{r}{2}\right) =\overline{s}-1+\frac{r}{2}+\dfrac{2}{d-2}\left(\overline{s}+\frac{r}{2}-1\right)
< \underline{s}-1+\frac{r}{2}+\frac{2}{N+2}
\]
By virtue of \eqref{eq:balance-p-q}, the last inequality is true for all $z\in Q_T$ if

\begin{equation}
\label{eq:d-2}
\begin{split}
\dfrac{2}{d-2}\left(\overline{s}^++\frac{r}{2}-1\right)<\dfrac{2(1-\beta)}{N+2}
\quad \Leftrightarrow \quad d>2+\frac{N+2}{2(1-\beta)}\left(2(\overline{s}^+-1)+r\right).
\end{split}
\end{equation}
Then

\[
\begin{split}
\mathcal{I}_{4,p} & \leq \int_{\Omega}f^2w_{\epsilon}^{\frac{r}{2}}\,dx + \int_{\Omega}w_{\epsilon}^{\frac{2(p-1)+r}{2}} |\nabla a|^2 \,dx
\\
& \leq \int_{\Omega}f^2w_{\epsilon}^{\frac{r}{2}}\,dx + \int_{\Omega}|\nabla a|^d\,dx +\int_{\Omega}w_{\epsilon}^{\frac{d}{d-2}\frac{2(p-1)+r}{2}}\,dx.
\end{split}
\]
The same inequality holds for $\mathcal{I}_{4,q}$. Gathering the obtained estimates we conclude that for every $\delta>0$ there is a constant $C$ such that

\begin{equation}
\label{eq:right-hand-side-1}
\int_{\Omega} f \operatorname{div}\left(w_\epsilon^{\frac{r}{2}}\mathcal{F}_\epsilon(z,\nabla u)\nabla u\right)\,dx\leq \delta \int_{\Omega}w_\epsilon^{\frac{r}{2}} \mathcal{F}_{\epsilon}^2(z,\nabla u)|u_{xx}|^2\,dx+  \int_{\Omega}f^2w_\epsilon^{\frac{r}{2}}\,dx + C.
\end{equation}
The constant $C$ depends on $a^+$, $b^+$, $p^\pm$, $q^\pm$, $N$, $\partial\Omega$, $\|u(t)\|_{2,\Omega}$, $\|\nabla a\|_{d,\Omega}$, $\|\nabla b\|_{d,\Omega}$ but is independent of $\epsilon$. Noting that all assumptions on $d$ are fulfilled by virtue of \eqref{eq:reg-data}, we may summarize the estimates of this section in the following assertion.

\begin{lemma}
\label{le:main-est-prelim}
Let $u$ be a classical solution of problem \eqref{eq:main-reg}. Assume that $p$, $q$ satisfy condition \eqref{eq:balance-p-q}, $a_t,b_t,|\nabla a|,|\nabla b|\in L^d(\Omega)$ for every $t\in (0,T)$, and $d$ satisfies \eqref{eq:reg-data}. Then for any $r\geq 0$ and $\delta>0$ there is a constant $C$, depending only on the $\textbf{data}$, $\delta$, $r$, and $\|u(t)\|_{2,\Omega}$, such that

\begin{equation}
\label{eq:main-prelim-1}
\begin{split}
\dfrac{d}{dt} & \left(\int_\Omega \left(\frac{a_{\epsilon}}{p+r} w_\epsilon^{\frac{p+r}{2}}+\frac{b_{\epsilon}}{q+r} w_\epsilon^{\frac{q+r}{2}}\right)\,dx\right) + \int_{\Omega} \operatorname{div}\left(\mathcal{F}_{\epsilon}(x,\nabla u)\nabla u\right)
\operatorname{div}\left(w_\epsilon^{\frac{r}{2}}\mathcal{F}_{\epsilon}(x,\nabla u)\nabla u\right)\,dx
\\
&
\leq C_\delta \int_{\Omega}f^2w_\epsilon^{\frac{r}{2}}\,dx+\delta \int_{\Omega}w_\epsilon^{\frac{r}{2}}\mathcal{F}_\epsilon^2(z,\nabla u)|u_{xx}|^2\,dx +C.
\end{split}
\end{equation}
\end{lemma}

\subsection{The second-order derivatives}
\subsubsection{Pointwise inequalities}
Given a smooth function $u$, we denote $\vec{\eta}=\dfrac{\nabla u}{  \sqrt{w_\epsilon} }$, $|\vec \eta|<1$. By the straightforward computation
\begin{equation}
\label{eq:calc-1}
\begin{split}
&  D_j\alpha_i D_i\beta_j
\\
&
\quad
= w_\epsilon^{\frac{r}{2}} \left[\left(a_{\epsilon}w_\epsilon^{\frac{p-2}{2}}+b_\epsilon w_\epsilon^{\frac{q-2}{2}}\right)D^2_{ij}u + \left((p-2)a_{\epsilon}w_{\epsilon}^{\frac{p-2}{2}}+(q-2) b_\epsilon w_\epsilon^{\frac{q-2}{2}}\right) \eta_i \sum_{k=1}^N \eta_k D^2_{kj}u\right.
\\
& \quad
+ \left.\frac{1}{2} \left(a_{\epsilon}w_\epsilon^{\frac{p-2}{2}}D_iuD_jp+b_\epsilon w_\epsilon ^{\frac{q-2}{2}} D_iuD_jq\right)\ln w_\epsilon + w_\epsilon^{\frac{p-2}{2}}D_iuD_ja + w_\epsilon^{\frac{q-2}{2}} D_i u D_j b\right]
\\
& \times \left[\left(a_{\epsilon}w_\epsilon^{\frac{p-2}{2}}+b_\epsilon w_\epsilon^{\frac{q-2}{2}}\right)D^2_{ij}u + \left((p-2)a_{\epsilon}w_{\epsilon}^{\frac{p-2}{2}}+(q-2) b_\epsilon w_\epsilon^{\frac{q-2}{2}}\right) \eta_j \sum_{k=1}^N \eta_k D^2_{ki}u\right.
\\
& \quad
+ \left.\frac{1}{2} \left(a_{\epsilon}w_\epsilon^{\frac{p-2}{2}}D_juD_ip+b_\epsilon w_\epsilon ^{\frac{q-2}{2}} D_juD_iq\right)\ln w_\epsilon + w_\epsilon^{\frac{p-2}{2}}D_juD_ia + w_\epsilon^{\frac{q-2}{2}} D_j u D_i b
\right]
\\
& \quad + r w_{\epsilon}^{\frac{r}{2}} \left(a_{\epsilon}w_\epsilon^{\frac{p-2}{2}}+b_\epsilon w_\epsilon^{\frac{q-2}{2}}\right)\eta_i\sum_{k=1}^N \eta_k D^2_{kj}u
\\
& \times
\left[\left(a_{\epsilon}w_\epsilon^{\frac{p-2}{2}}+b_\epsilon w_\epsilon^{\frac{q-2}{2}}\right)D^2_{ij}u + \left((p-2)a_{\epsilon}w_{\epsilon}^{\frac{p-2}{2}}+(q-2) b_\epsilon w_\epsilon^{\frac{q-2}{2}}\right) \eta_j \sum_{k=1}^N \eta_k D^2_{ki}u\right.
\\
& \quad
\left.+ \frac{1}{2} \left(a_{\epsilon}w_\epsilon^{\frac{p-2}{2}}D_juD_ip+b_\epsilon w_\epsilon ^{\frac{q-2}{2}} D_juD_iq\right)\ln w_\epsilon+ w_\epsilon^{\frac{p-2}{2}}D_juD_ia  + w_\epsilon^{\frac{q-2}{2}} D_j u D_i b\right].
\end{split}
\end{equation}
Fix a point $x_0\in \Omega$ and denote by $\mathcal{H}\equiv \mathcal{H}(u(x_0))$ the Hessian matrix of $u$ at $x_0$. The following formulas hold:
\[
\begin{split}
& \sum_{i,j=1}^N\left(D^2_{ij}u\right)^2=\operatorname{trace}\mathcal{H}^2,
\\
& \sum_{i,j=1}^N\eta_jD^2_{ij}u\sum_{k=1}^N\eta_kD^2_{ki}u =\sum_{i=1}^N\left(\sum_{j=1}^N\mathcal{H}_{ij}\eta_j \right)\left(\sum_{k=1}^{N}\mathcal{H}_{ki}\eta_k \right)=\left|\mathcal{H}\cdot \eta\right|^2,
\\
& \sum_{i,j=1}^N\eta_i\eta_j\left(\sum_{k=1}^N \eta_kD^2_{ki}u\right)\left(\sum_{k=1}^N \eta_kD_{kj}^2u\right) =
\left(\sum_{i=1}^N\eta_i \cdot (\mathcal{H}\cdot \eta)_i\right)\left(\sum_{j=1}^{N}\eta_j \cdot (\mathcal{H}\cdot \eta)_j\right)=(\mathcal{H}\cdot \eta,\eta)^2.
\end{split}
\]
Using these formulas we rewrite \eqref{eq:calc-1} in the equivalent form
\[
\begin{split}
w_{\epsilon}^{-\frac{r}{2}}\sum_{ij} & D_i\alpha_j D_j\beta_i = \left(a_{\epsilon}w_\epsilon^{\frac{p-2}{2}}+b_\epsilon w_\epsilon^{\frac{q-2}{2}}\right)^2 \operatorname{trace}\mathcal{H}^2
\\
&
+ 2\left(a_{\epsilon}w_\epsilon^{\frac{p-2}{2}}+b_\epsilon w_\epsilon^{\frac{q-2}{2}}\right)\left(a_{\epsilon}(p-2)w_\epsilon^{\frac{p-2}{2}}+b_\epsilon (q-2) w_\epsilon^{\frac{q-2}{2}}\right)|\mathcal{H}\cdot \eta|^2
\\
& + \left(a_{\epsilon}(p-2)w_\epsilon^{\frac{p-2}{2}}+b_\epsilon (q-2) w_\epsilon^{\frac{q-2}{2}}\right)^2 (\mathcal{H}\cdot \eta,\eta)^2
\\
& + r \left(a_{\epsilon}w_\epsilon^{\frac{p-2}{2}}+b_\epsilon w_\epsilon^{\frac{q-2}{2}}\right)^2 |\mathcal{H}\cdot \eta|^2
\\
& +r  \left(a_{\epsilon}w_\epsilon^{\frac{p-2}{2}}+b_\epsilon w_\epsilon^{\frac{q-2}{2}}\right) \left(a_{\epsilon}(p-2)w_\epsilon^{\frac{p-2}{2}}+b_\epsilon (q-2) w_\epsilon^{\frac{q-2}{2}}\right) (\mathcal{H}\cdot\eta,\eta)^2
+ \mathcal{J}
\\
& \equiv \sum_{k=1}^5\mathbf{I}_k + \mathcal{J} + \mathcal{L},
\end{split}
\]
where $\mathbf{I}_k$ are second-order polynomials of $D_{ij}^2u$, while $\mathcal{J}=\sum_{ij}\mathcal{J}_{ij}$ are first-order and $\mathcal{L}=\sum_{ij}\mathcal{L}_{ij}$ zero-order polynomials of $D_{ij}^2u$. The polynomials $\mathcal{J}_{ij}$ and $\mathcal{L}_{ij}$ have the form
\begin{equation}
\label{eq:res-1}
\begin{split}
\mathcal{J}_{ij} & = \left[\left(a_{\epsilon}w_\epsilon^{\frac{p-2}{2}}+b_\epsilon w_\epsilon^{\frac{q-2}{2}}\right)D^2_{ij}u + \left((p-2)a_{\epsilon}w_{\epsilon}^{\frac{p-2}{2}}+(q-2) b_\epsilon w_\epsilon^{\frac{q-2}{2}}\right) \eta_i \sum_{k=1}^N \eta_k D^2_{kj}u\right]
\\
& \qquad \times \left[\frac{1}{2} \left(a_{\epsilon}w_\epsilon^{\frac{p-2}{2}}D_juD_ip+b_\epsilon w_\epsilon ^{\frac{q-2}{2}} D_juD_iq\right)\ln w_\epsilon +w_\epsilon^{\frac{p-2}{2}}D_iuD_ja + w_\epsilon^{\frac{q-2}{2}} D_j u D_i b
\right]
\\
& + \frac{1}{2} \left[\left(a_{\epsilon}w_\epsilon^{\frac{p-2}{2}}D_iuD_jp+b_\epsilon w_\epsilon ^{\frac{q-2}{2}} D_iuD_jq\right)\ln w_\epsilon + w_{\epsilon}^{\frac{p-2}{2}}D_iuD_ja + w_\epsilon^{\frac{q-2}{2}} D_i u D_j b\right]
\\
& \qquad \times \left[\left(a_{\epsilon}w_\epsilon^{\frac{p-2}{2}}+b_\epsilon w_\epsilon^{\frac{q-2}{2}}\right)D^2_{ij}u + \left((p-2)a_{\epsilon}w_{\epsilon}^{\frac{p-2}{2}}+(q-2) b_\epsilon w_\epsilon^{\frac{q-2}{2}}\right) \eta_j \sum_{k=1}^N \eta_k D^2_{ki}u\right]
\\
& + \frac{r}{2} \left[\left(a_{\epsilon}w_\epsilon^{\frac{p-2}{2}}D_iuD_jp+b_\epsilon w_\epsilon ^{\frac{q-2}{2}} D_iuD_jq\right)\ln w_\epsilon + w_\epsilon^{\frac{p-2}{2}}D_iuD_ja  + w_\epsilon^{\frac{q-2}{2}} D_i u D_j b\right]
\\
& \qquad \times \left(a_{\epsilon}w_\epsilon^{\frac{p-2}{2}}+b_\epsilon w_\epsilon^{\frac{q-2}{2}}\right)\eta_i\sum_{k=1}^N \eta_k D^2_{kj}u
\end{split}
\end{equation}
and
\begin{equation}
\label{eq:res-2}
\begin{split}
\mathcal{L}_{ij} & =\frac{1}{4} \left[\left(a_{\epsilon}w_\epsilon^{\frac{p-2}{2}}D_iuD_jp+b_\epsilon w_\epsilon ^{\frac{q-2}{2}} D_iuD_jq\right)\ln w_\epsilon + w_\epsilon^{\frac{p-2}{2}}D_iuD_ja + w_\epsilon^{\frac{q-2}{2}} D_i u D_j b\right]\\
& \qquad \times \left[\left(a_{\epsilon}w_\epsilon^{\frac{p-2}{2}}D_juD_ip+b_\epsilon w_\epsilon ^{\frac{q-2}{2}} D_juD_iq\right)\ln w_\epsilon + w_\epsilon^{\frac{p-2}{2}}D_juD_ia + w_\epsilon^{\frac{q-2}{2}} D_j u D_i b\right]
\end{split}
\end{equation}
The polynomials $\mathcal{J}_{ij}$ and $\mathcal{K}_{ij}$ are estimated as follows: for $i,j=\overline{1,N}$
\begin{equation}
\label{eq:J-pw}
\begin{split}
|\mathcal{J}_{ij}| & \leq C\left(\mathcal{F}_\epsilon(z,\nabla u)|u_{xx}|\right)\left(w_{\epsilon}^{\frac{p-1}{2}} + w_{\epsilon}^{\frac{q-1}{2}}\right)\left(|\ln w_{\epsilon}| +|\nabla a|+|\nabla b|\right)
\end{split}
\end{equation}
and
\begin{equation}
\label{eq:J-pw-1}
\begin{split}
|\mathcal{L}_{ij}| & \leq C \left[\left(a_{\epsilon}^2 w_\epsilon^{p-2} |\nabla u|^2 + b_\epsilon^2 w_\epsilon ^{q-2} |\nabla u|^2 \right) |\ln w_\epsilon|^2 + \left(w_\epsilon^{p-2} |\nabla u|^2 |\nabla a|^2+ w_\epsilon^{q-2} |\nabla u|^2 |\nabla b|^2\right)\right]
\end{split}
\end{equation}
with a constant $C=C(a^+,b^+,p^+,q^+,r)$. Combining the terms $\mathbf{I}_k$ we obtain

\[
\mathbf{I}_1+\mathbf{I}_4= \left(a^2_{\epsilon}w_\epsilon^{p-2}+2a_\epsilon b_\epsilon w_\epsilon^{\frac{p+q-4}{2}} +b^2_\epsilon w_\epsilon^{q-2}\right)\left(\operatorname{trace}\mathcal{H}^2 + r|\mathcal{H}\cdot \eta|^2\right),
\]

\[
\mathbf{I}_2+\mathbf{I}_5=\left[2(p-2)a_{\epsilon}^2w_{\epsilon}^{p-2}  + 2a_\epsilon b_\epsilon (p+q-4) w_\epsilon^{\frac{p+q-4}{2}} +2b_\epsilon^2(q-2)w_\epsilon^{q-2}\right] \left(|\mathcal{H}\cdot \eta|^2 +r (\mathcal{H}\cdot \eta,\eta)^2\right),
\]

\[
\mathbf{I}_3= \left((p-2)^2a^2_{\epsilon}w_\epsilon^{p-2}+2(p-2)(q-2)a_\epsilon b_\epsilon w_\epsilon^{\frac{p+q-4}{2}} +(q-2)^2b^2_\epsilon w_\epsilon^{q-2}\right)(\mathcal{H}\cdot \eta,\eta)^2,
\]
Regrouping these expressions we obtain

\[
\begin{split}
\sum_{i=1}^{5}\mathbf{I}_i =
& a_{\epsilon}^2 w_{\epsilon}^{p-2}\left(\operatorname{trace}\mathcal{H}^2 + r|\mathcal{H}\cdot \eta|^2 + 2(p-2)\left[|\mathcal{H}\cdot \eta|^2 +r(\mathcal{H}\cdot \eta,\eta)^2\right]+(p-2)^2(\mathcal{H}\cdot \eta,\eta)^2 \right)
\\
+ & b_{\epsilon}^2 w_{\epsilon}^{q-2}\left(\operatorname{trace}\mathcal{H}^2 + r|\mathcal{H}\cdot \eta|^2 + 2(q-2)\left[|\mathcal{H}\cdot \eta|^2 +r(\mathcal{H}\cdot \eta,\eta)^2\right]+(q-2)^2(\mathcal{H}\cdot \eta,\eta)^2 \right)
\\
+ & 2a_\epsilon b_\epsilon w_{\epsilon}^{\frac{p+q-4}{2}}\left(\operatorname{trace}\mathcal{H}^2 + r|\mathcal{H}\cdot \eta|^2 + (p+q-4)\left(|\mathcal{H}\cdot \eta|^2 +r(\mathcal{H}\cdot \eta,\eta)^2 \right) \right.
\\
&
\left.+(p-2)(q-2)(\mathcal{H}\cdot \eta,\eta)^2 \right).
\end{split}
\]
Pass to the coordinate system with the origin $x_0$ where the matrix $\mathcal{H}$ becomes diagonal. Denote by $d_i$, $i=\overline{1,N}$,  the diagonal elements of the transformed matri,x and by $\zeta$ the vector $\eta$ in the new coordinate system. Applying the formulas

\[
\begin{split}
& \operatorname{trace}\mathcal{H}^2=\sum_{i}d_i^2,\qquad |\mathcal{H}\cdot \eta|^2 = (\mathcal{H}\cdot \eta,\mathcal{H}\cdot \eta)= \sum_id_i^2\zeta_i^2,\qquad
(\mathcal{H}\cdot\eta,\eta)^2 = \left(\sum_id_i\zeta_i^2\right)^2
\end{split}
\]
we represent

\[
\sum_{i=1}^{5}\mathbf{I}_i = a_{\epsilon}^2 w_{\epsilon}^{p-2}\mathcal{G}_1(\zeta) + b_{\epsilon}^2 w_{\epsilon}^{q-2}\mathcal{G}_2(\zeta) + 2a_\epsilon b_\epsilon w_{\epsilon}^{\frac{p+q-4}{2}}\mathcal{G}_3(\zeta),
\]
where

\[
\begin{split}
\mathcal{G}_1(\zeta) & = \sum_id_i^2 + (2p+r-4)\sum_id^2_i\zeta_i^2 + (p-2)(2r+p-2)\left(\sum_id_i\zeta_i^2\right)^2,
\\
\mathcal{G}_2(\zeta) & = \sum_id_i^2 + (2q+r-4)\sum_id^2_i\zeta_i^2 + (q-2)(2r+q-2)\left(\sum_id_i\zeta_i^2\right)^2,
\\
\mathcal{G}_3(\zeta) & = \sum_id_i^2 +(p+q+r-4)\sum_id_i^2\zeta_i^2 + [r(p+q-4)+(p-2)(q-2)]\left(\sum_id_i\zeta_i^2\right)^2.
\end{split}
\]
Set $\lambda=|\zeta|<1$, define the vector $\xi=\lambda^{-1}\zeta$, $|\xi|=1$, and rewrite

\[
\begin{split}
\mathcal{G}_1(\zeta) & = \sum_id_i^2 + \lambda^2(2p+r-4)\sum_id^2_i\xi_i^2 + \lambda^4(p-2)(2r+p-2)\left(\sum_id_i\xi_i^2\right)^2,
\\
\mathcal{G}_2(\zeta) & = \sum_id_i^2 + \lambda^2(2q+r-4)\sum_id^2_i\xi_i^2 + \lambda^4(q-2)(2r+q-2)\left(\sum_id_i\xi_i^2\right)^2,
\\
\mathcal{G}_3(\zeta) & = \sum_id_i^2 +\lambda^2(p+q+r-4)\sum_id_i^2\xi_i^2 + \lambda^4[r(p+q-4)+(p-2)(q-2)]\left(\sum_id_i\xi_i^2\right)^2.
\end{split}
\]
Notice that

\begin{equation}
\label{eq:aux-elem-1}
\begin{split}
& \lambda^4\left(\sum_id_i\xi_i^2\right)^2=\lambda^4\left(\sum_i(d_i\xi_i)\xi_i\right)^2\leq \lambda^4\left(\sum_id_i^2\xi_i^2\right)^2|\xi|^2=\lambda^4\sum_id_i^2\xi_i^2\leq \lambda^2\sum_id_i^2\xi_i^2,
\\
& \sum_id_i^2\xi_i^2\leq \sum_id_i^2.
\end{split}
\end{equation}

\begin{lemma}
\label{le:pointwise}
If $p>\dfrac{3}{2}$, $q>\dfrac{3}{2}$, and $r\geq 0$, then there is a finite constant $\theta\in (0,1)$ such that

\[
\mathcal{G}_i(\zeta)\geq (1-\theta)\operatorname{trace}\mathcal{H}^2.
\]
\end{lemma}

\begin{proof}
We consider all possible combinations of the parameters $p$, $q$.

\begin{enumerate}
\item $p\geq 2$, $q\geq 2$, $r>0$. Then

\[
\mathcal{G}_i(\zeta)\geq \sum_{i}d_i^2=\operatorname{trace}\mathcal{H}^2,\quad i=1,2,3.
\]

\item $p\geq 2$, $\dfrac{3}{2}<q<2$, $r>0$. For every $\theta\in [0,1)$

\[
\begin{split}
\mathcal{G}_1(\zeta) & \geq \sum_id_i^2,
\\
\mathcal{G}_2(\zeta) & \geq (1-\theta)\sum_{i}d_i^2+ \lambda^2\left(\theta +(2q+r-4)\right)\sum_{i}d_i^2 + \lambda^4(q-2)(2r+q-2)\sum_id_i^2\xi_i^2.
\end{split}
\]
If $2r+q-2\leq 0$, the last term is nonnegative and can be omitted. Since $2q-4>-1$ by assumption, there is $\theta\in (0,1)$ such that

\[
\mathcal{G}_2(\zeta)\geq (1-\theta)\sum_{i}d_i^2.
\]
If $2r+q-2> 0$, the last term of $\mathcal{G}_2(\zeta)$ is negative. Since $\lambda<1$, we substitute $\lambda^4$ by $\lambda^2$ to obtain

\[
\mathcal{G}_2(\zeta) \geq (1-\theta)\sum_{i}d_i^2+ \lambda^2\left(\theta +(2q+r-4)+(q-2)(2r+q-2)\right)\sum_id_i^2\xi_i^2.
\]
For $q>\dfrac{3}{2}$ the last term on the right-hand side is bounded from below:

\[
\begin{split}
(2q+r-4) & +(q-2)(2r+q-2)=2(q-2)+r+2r(q-2)+(q-2)^2
\\
& = (q-1)^2-1+2r\left(q-\frac{3}{2}\right)>\dfrac{1}{4}-1=-\dfrac{3}{4},
\end{split}
\]
whence $\mathcal{G}_2(\zeta)\geq (1-\theta) \sum_id_i^2$ with $\theta>\dfrac{3}{4}$. To estimate $\mathcal{G}_3(\zeta)$ we notice that for $\theta\in (1/2,1)$

\[
\theta+p+q+r-4\geq \theta + q-2> \theta +\frac{3}{2}-2=\theta -\frac{1}{2}>0,
\]
therefore

\[
\begin{split}
\mathcal{G}_3(\zeta)\geq (1-\theta)\sum_id_i^2 & +\lambda^2(\theta+p+q+r-4)\sum_{i}d_i^2\xi_i^2
\\
&
+ \lambda^4[r(p+q-4)+(p-2)(q-2)]\left(\sum_id_i\xi_i^2\right)^2.
\end{split}
\]
If $r(p+q-4)+(p-2)(q-2)\geq 0$, then the last term can be omitted, and the required inequality follows: for $\theta\in [1/2,1)$, $r>0$, $p\geq 2$, $q\in (3/2,2)$

\[
\begin{split}
\mathcal{G}_3(\zeta) & \geq (1-\theta)\sum_id_i^2+\lambda^2(\theta+p+q+r-4)\sum_{i}d_i^2\xi_i^2
\\
& > (1-\theta)\sum_id_i^2+\lambda^2(\theta+q-2)\sum_{i}d_i^2\xi_i^2 > (1-\theta)\sum_id_i^2.
\end{split}
\]
Let $r(p+q-4)+(p-2)(q-2)< 0$. By virtue of \eqref{eq:aux-elem-1}

\[
\begin{split}
\mathcal{G}_3(\zeta) & \geq (1-\theta)\sum_id_i^2+ \lambda^2(\theta+p+q+r-4+r(p+q-4)+(p-2)(q-2))\sum_{i} d_i^2\xi_i^2.
\end{split}
\]
For $p\geq 2$, $q\in (3/2,2)$, and $\theta \in (1/2,1)$ the coefficient of the last term is estimated from below,

\[
\begin{split}
\theta & +p+q+r-4+r(p+q-4)+(p-2)(q-2)
\\
&
= \theta + r + (p-2)(q-1)+q-2+r((q-2)+(p-2))
\\
& \geq \theta +(p-2)(q-1)+ r(q-1)+(q-2)\geq \theta +(q-2)>\theta-\dfrac{1}{2}>0,
\end{split}
\]
whence

\[
\mathcal{G}_3(\zeta) \geq (1-\theta)\sum_id_i^2.
\]
The same estimates hold if $q\geq 2$, $p\in (3/2,2)$, $r>0$.

\item $p,q\in (3/2,2)$, $r>0$. In this case, the estimates on $\mathcal{G}_1(\zeta)$, $\mathcal{G}_2(\zeta)$ follow as in item (2). To estimate $\mathcal{G}_3$ we use \eqref{eq:aux-elem-1}:

    \[
    \begin{split}
    \mathcal{G}_3(\zeta) & \geq (1-\theta)\sum_id_i^2 +\lambda^2(\theta+p+q+r-4 + r(p+q-4))\sum_id_i^2\xi_i^2
        \end{split}
    \]
Because $p+q-4>-1$, we may choose $\theta$ so close to one that

\[
\theta+p+q+r-4 + r(p+q-4)= (\theta+p+q-4)+r +r(p+q-4)> \theta+p+q-4\geq 0.
\]
It follows that

\[
\mathcal{G}_3\geq (1-\theta)\sum_{i}d_i^2\qquad \text{with $\theta\in(0,1)$}.
\]
\end{enumerate}
\end{proof}

\begin{lemma}
\label{le:pointwise-1}
Let $u\in C^2(\Omega)$ and $p^-,q^->\dfrac{3}{2}$. For every $r\geq 0$ there is a constant $\theta\equiv \theta(p^-,q^-,r)\in (0,1)$ such that

\begin{equation}
\label{eq:pointwise-1}
\begin{split}
\sum_{i,j=1}^N  D_i \left(\mathcal{F}_{\epsilon}(x,\nabla u) D_j u\right)  & D_j\left(w_\epsilon^{\frac{r}{2}}\mathcal{F}_{\epsilon}(x,\nabla u)D_iu \right)
\\
& \geq (1-\theta)w_\epsilon^{\frac{r}{2}}\mathcal{F}_{\epsilon}^2(x,\nabla u)|u_{xx}|^2 + w_\epsilon^{\frac{r}{2}}\sum_{i,j=1}^N\mathcal{J}_{ij}+ w_\epsilon^{\frac{r}{2}}\sum_{i,j=1}^N\mathcal{L}_{ij}
\end{split}
\end{equation}
with $\mathcal{J}_{ij}$ and $\mathcal{L}_{ij}$ are defined in \eqref{eq:res-1} and \eqref{eq:res-2}.
\end{lemma}

\subsubsection{Integral inequalities}
Combining \eqref{eq:pointwise-1} with \eqref{eq:double-final-e-start} we arrive at the following integral inequality.

\begin{lemma}
\label{le:principal-e} Let $\partial\Omega\in C^2$, $u\in C^3(\Omega)\cap
C^{2}(\overline{\Omega})$, $p, q \in C^{0,1}(\overline{\Omega})$, $p^-, q^-> \dfrac{3}{2}$ and $r\geq 0$. There is a constant $\theta\in (0,1)$ such that

\begin{equation}
\label{eq:p-est-1}
\begin{split}
\int_{\Omega} & \operatorname{div}\left(\mathcal{F}_{\epsilon}(x,\nabla u)\nabla u\right)
\operatorname{div}\left(w_\epsilon^{\frac{r}{2}}\mathcal{F}_{\epsilon}(x,\nabla u)\nabla u\right)\,dx \geq \theta \int_{\Omega} w_{\epsilon}^{\frac{r}{2}}\mathcal{F}_\epsilon^{2}(x,\nabla u) |u_{xx}|^2 \,dx
\\
&
- \int_{\partial\Omega} w_\epsilon^{\frac{r}{2}}\mathcal{F}^2_{\epsilon}(x,\nabla u)|\nabla u|^2\operatorname{trace}\mathcal{B}\,dS
 + \sum_{i,j=1}^N\int_{\Omega}w_\epsilon^{\frac{r}{2}}\mathcal{J}_{ij}\,dx  + w_\epsilon^{\frac{r}{2}}\sum_{i,j=1}^N\mathcal{L}_{ij}
\end{split}
\end{equation}
where $\mathcal{B}$ is the second fundamental form of the surface $\partial\Omega$, and $\mathcal{J}_{ij}$ and $\mathcal{L}_{ij}$ are defined in \eqref{eq:res-1} and \eqref{eq:res-2}.
\end{lemma}

Using Lemma \ref{le:trace-main} to estimate the boundary integral in \eqref{eq:p-est-1} we refine Lemma \ref{le:principal-e} as follows.
\begin{lemma}
\label{le.second-der-est-1}
Let the conditions of Lemmas 
\ref{le:principal-e} be fulfilled and \eqref{eq:d-boundary} holds true.  Then

\begin{equation}
  \label{eq:p-est-2}
  \begin{split}
\int_{\Omega} \operatorname{div}\left(\mathcal{F}_{\epsilon}(x,\nabla u)\nabla u\right)
& \operatorname{div}\left(w_\epsilon^{\frac{r}{2}}\mathcal{F}_{\epsilon}(x,\nabla u)\nabla u\right)\,dx
\\
&
\geq \theta \int_{\Omega} w_{\epsilon}^{\frac{r}{2}}\mathcal{F}_\epsilon^{2}(x,\nabla u) |u_{xx}|^2 \,dx
  + \sum_{i,j=1}^N\int_{\Omega}w_\epsilon^{\frac{r}{2}}\mathcal{J}_{ij}\,dx + \sum_{i,j=1}^N  \int_{\Omega} w_\epsilon^{\frac{r}{2}} \mathcal{L}_{ij} ~dx +C
 \end{split}
\end{equation}
with a constant $\theta\in (0,1)$ and a constant $C$ depending only on the data.
\end{lemma}

\section{The main estimate}
Multiplication of equation \eqref{eq:main-reg} by the solution $u$ and integration by parts in $Q_t=\Omega\times (0,t)$, $t\in (0,T)$, leads to the equality

\begin{equation}
\label{eq:a-priori-1}
\frac{1}{2}\|u(t)\|_{2,\Omega}^2+\int_{Q_t} \mathcal{F}_\epsilon(z,\nabla u)|\nabla u|^2
\,dz=\int_{Q_t}fu\,dz + \frac{1}{2}\|u_0\|_{2,\Omega}^2.
\end{equation}
Dropping the second term on the left-hand side, from \eqref{eq:a-priori-1} we obtain the differential inequality

\begin{equation}
\label{eq:Gronwall}
X'(t)\leq X(t)+\|f\|^2_{2,Q_t} + \|u_0\|_{2,\Omega}^2, \qquad X(t) =\int_0^t\|u(\tau)\|_{2,\Omega}^2\,d\tau.
\end{equation}
By Gr\"onwall's inequality $X(t)\leq C(T,\|f\|_{2,Q_T},\|u_0\|_{2,\Omega})$. Plugging this inequality into \eqref{eq:Gronwall} we obtain the uniform estimate

\begin{equation}
\label{eq:unif-1}
\sup_{(0,T)}\|u(t)\|^2_{2,\Omega}\leq C,\quad C=C(T,\|f\|_{2,Q_T},\|u_0\|_{2,\Omega}).
\end{equation}
Gathering \eqref{eq:unif-1} and \eqref{eq:a-priori-1} we obtain the uniform in $\epsilon$ estimate

\begin{equation}
\label{eq:unif-2}
\frac{1}{2}\sup_{(0,T)}\|u(t)\|_{2,\Omega}^2+\int_{Q_T} \mathcal{F}_\epsilon(z,\nabla u)|\nabla u|^2
\,dz\leq C\|f\|_{2,Q_T}^2 + \frac{1}{2}\|u_0\|_{2,\Omega}^2.
\end{equation}
By \eqref{eq:main-prelim-1} and \eqref{eq:p-est-2}

\begin{equation}
\label{eq:est-1}
\begin{split}
\dfrac{1}{2}\dfrac{d}{dt} \int_{\Omega}\left(\int_{0}^{w_\epsilon}s^{\frac{r}{2}} \mathcal{F}_\epsilon(z,s)\,ds\right)\,dx
 & + C_0\int_\Omega w_\epsilon^{\frac{r}{2}}\mathcal{F}_{\epsilon}^2(x,\nabla u)|u_{xx}|^2\,dx
\\
&
\leq C_1 +\sum_{i,j=1}^N \int_\Omega w_\epsilon^{\frac{r}{2}}\mathcal{J}_{ij}\,dx + \sum_{i,j=1}^N  \int_{\Omega} w_\epsilon^{\frac{r}{2}} \mathcal{L}_{ij} ~dx
+C_2\int_{\Omega}f^2w_\epsilon^{\frac{r}{2}}\,dx.
\end{split}
\end{equation}
Using \eqref{eq:J-pw} and Young's inequality, we find that for every $\delta>0$

\[
\begin{split}
\int_\Omega w_\epsilon^{\frac{r}{2}}\mathcal{J}_{ij}\,dx &  \leq C\int_\Omega \left(w_\epsilon^{\frac{r}{4}}\mathcal{F}_\epsilon(z,\nabla u)|u_{xx}|\right)\left(w_{\epsilon}^{\frac{p-1}{2}+\frac{r}{4}} + w_{\epsilon}^{\frac{q-1}{2}+\frac{r}{4}}\right)\left(|\ln w_{\epsilon}| +|\nabla a|+|\nabla b|\right)\,dx
\\
& \leq \delta \int_{\Omega}w^{\frac{r}{2}}_\epsilon \mathcal{F}_\epsilon^2(x,\nabla u)|u_{xx}|^2\,dx
\\
&
+C \int_{\Omega}\left(w_{\epsilon}^{\frac{2(p-1)+r}{2}} + w_{\epsilon}^{\frac{2(q-1)+r}{2}}\right)\left(|\ln w_{\epsilon}| ^2 +|\nabla a|^2+|\nabla b|^2\right)\,dx
\\
&
\equiv \delta \int_{\Omega}w^{\frac{r}{2}}_\epsilon \mathcal{F}_\epsilon^2(x,\nabla u)|u_{xx}|^2\,dx+\mathcal{K}.
\end{split}
\]
Assume first that in \eqref{eq:reg-data} $d$ is finite. Due to \eqref{eq:elem-ln}

\begin{equation}
\label{eq:est-d}
\begin{split}
\mathcal{K} & \leq C_1\int_{\Omega} \left(w_{\epsilon}^{\frac{2(p-1)+r}{2}} + w_{\epsilon}^{\frac{2(q-1)+r}{2}}\right)\,dx + C_2\int_{\Omega}\left(w_{\epsilon}^{\frac{2(p-1)+r+\theta}{2}} + w_{\epsilon}^{\frac{2(q-1)+r+\theta}{2}}\right)\,dx
\\
&
+ C_3\int_{\Omega}\left(|\nabla a|^d+|\nabla b|^d\right)\,dx + C_4 \int_{\Omega} \left(w_{\epsilon}^{\frac{2(p-1)+r}{2}\frac{d}{d-2}} + w_{\epsilon}^{\frac{2(q-1)+r}{2}\frac{d}{d-2}}\right)\,dx + C_5
\end{split}
\end{equation}
with finite constants $C_i$ and an arbitrary $\theta>0$. When $\theta \in (0,r^\sharp)$, the first two terms are estimated by \eqref{eq:d-ell-cor-mod} and \eqref{eq:unif-1}. The third term is bounded by assumption. To estimate the last term by \eqref{eq:d-ell-cor-mod}, we claim

\[
(2(\overline{s}-1)+r)\frac{d}{d-2}=2(\overline{s}-1)+r+(2(\overline{s}-1)+r)\frac{2}{d-2} <2(\underline{s}(z)-1)+\frac{4}{N+2}+r,
\]
which is equivalent to

\[
(\overline{s}(z)-\underline{s}(z))+\dfrac{1}{d-2}(2(\overline{s}-1)+r)<\frac{2}{N+2}
\]
In view of \eqref{eq:balance-p-q}, this is true if we claim that $d$ satisfies \eqref{eq:d-2}:
\[
d>2+\dfrac{N+2}{2(1-\beta)}\left(2(\overline{s}^+-1)+r\right)\quad \Rightarrow \quad \dfrac{1}{d-2}(2(\overline{s}(z)-1)+r)<\dfrac{2(1-\beta)}{N+2}.
\]
If $d=\infty$, then

\[
\mathcal{K}\leq C \int_{\Omega}\left(w_{\epsilon}^{\frac{2(p-1)+r}{2}} + w_{\epsilon}^{\frac{2(q-1)+r}{2}}\right)\left(|\ln w_{\epsilon}| ^2 +\|\nabla a\|_{\infty,\Omega}^2+\|\nabla b\|_{\infty,\Omega}^2\right)\,dx,
\]
and the estimate follows from \eqref{eq:d-ell-cor-mod} as in the case of finite $d$.

Using \eqref{eq:J-pw-1}, \eqref{eq:elem-ln} and Young's inequality, we find that
\[
\begin{split}
    \int_{\Omega} w_\epsilon^{\frac{r}{2}} \mathcal{L}_{ij} ~dx
    & \leq C_0 \int_{\Omega}  w_\epsilon^{\frac{r}{2}} \left(a_{\epsilon}^2 w_\epsilon^{p-2} |\nabla u|^2 + b_\epsilon^2 w_\epsilon ^{q-2} |\nabla u|^2 \right) |\ln w_\epsilon|^2 ~dx \\
    & \qquad \qquad \qquad + \int_{\Omega}  w_\epsilon^{\frac{r}{2}} \left(w_\epsilon^{p-2} |\nabla u|^2 |\nabla a|^2+ w_\epsilon^{q-2} |\nabla u|^2 |\nabla b|^2\right) ~dx\\
    & \leq C_1 \int_{\Omega}\left(w_{\epsilon}^{\frac{2(p-1)+r+\theta}{2}} + w_{\epsilon}^{\frac{2(q-1)+r+\theta}{2}}\right)\,dx + C_2\int_{\Omega}\left(|\nabla a|^d+|\nabla b|^d\right)\,dx \\
    & \qquad + C_3 \int_{\Omega} \left(w_{\epsilon}^{\frac{2(p-1)+r}{2}\frac{d}{d-2}} + w_{\epsilon}^{\frac{2(q-1)+r}{2}\frac{d}{d-2}}\right)\,dx + C_4
\end{split}
\]
with finite constants $C_i$ and an arbitrary $\theta>0$. When $\theta \in (0, \frac{r^\sharp}{2})$, the first term is estimated by \eqref{eq:d-ell-cor-mod} and \eqref{eq:unif-1}. The second and third can be estimated as in \eqref{eq:est-d}.

It remains to estimate
\begin{equation}
\label{eq:integral-I}
\mathcal{I} :=\int_{Q_T}f^2  w_\epsilon^{\frac{r}{2}} \,dz.
\end{equation}
\subsection{Case I: $f \in L^\sigma(\Omega), \sigma \geq N+2$}
By H\"older inequality
\[
\mathcal{I}\leq \|f\|_{N+2,Q_T}^2 \left(\int_{Q_T} w_\epsilon^{\frac{r(N+2)}{2N}}\,dz\right)^{\frac{N}{N+2}}.
\]
Observe that
\[
\frac{r(N+2)}{2N} = r\left(\frac{1}{2}+\frac{1}{N}\right) <\frac{2 (\ell-1)+r+\frac{2r}{N}}{2}, \quad \ell \in \{p, q\}.
\]
By using $a(z) + b(z) \geq \alpha>0$, $a, b \in L^\infty(Q_T)$, and by applying Young inequality, for every $\mu>0$
\[
\begin{split}
    w_\epsilon^{\frac{r(N+2)}{2N}} \leq C(\mu) + \mu \left(a^2_\epsilon (z) w_\epsilon^{\frac{2p+r-2+\frac{2r}{N}}{2}} + b^2_\epsilon (z) w_\epsilon^{\frac{2q+r-2+\frac{2r}{N}}{2}}\right)
\end{split}
\]
where $C$ depends upon $\|a\|_{L^\infty}$, $\|b\|_{L^\infty}$, $\mu$ and $N$ but is independent of $\epsilon$. Using the above inequality, we have for every $\mu>0$

\begin{equation}
  \label{eq:I-1}
  \mathcal{I}\leq \|f\|_{N+2,Q_T}^2 \left[C+ \mu \int_{Q_T}\left(a^2_\epsilon (z) w_\epsilon^{\frac{2p+r-2+\frac{2r}{N}}{2}} + b^2_\epsilon (z) w_\epsilon^{\frac{2q+r-2+\frac{2r}{N}}{2}}\right) \,dz\right]^{\frac{N}{N+2}},\qquad C=C(\mu).
  \end{equation}
By the Sobolev embedding $W^{1,\frac{2N}{N+2}}(\Omega)
\subset L^2(\Omega)$. Hence for every $v\in W^{1,\frac{2N}{N+2}}(\Omega)$

\begin{equation}
\label{eq:Sob-1}
\|v\|_{2,\Omega}^2\leq C_{s}\left(\|\nabla v\|_{\frac{2N}{N+2},\Omega}^2 + \|v\|^2_{\frac{2N}{N+2},\Omega}\right)
\end{equation}
with an independent of $v$ constant $C_s$. Denote
\[
v_\epsilon(z) := a^2_\epsilon (z) w_\epsilon^{\frac{2p+r-2+\frac{2r}{N}}{2}} + b^2_\epsilon (z) w_\epsilon^{\frac{2q+r-2+\frac{2r}{N}}{2}}.
\]
Notice that

\begin{equation}\label{lowerest-1}
    v_\epsilon(z) \geq  \begin{cases}
(a^2(z)+b^2(z)) w_\epsilon^{\frac{2\overline{s}+r-2+\frac{2r}{N}}{2}} \geq \frac{\alpha^2}{2} w_\epsilon^{\frac{2\overline{s}+r-2+\frac{2r}{N}}{2}} & \text{if $|w_\epsilon| < 1$},
\\
(a^2(z)+b^2(z))  w_\epsilon^{\frac{2\underline{s}+r-2+\frac{2r}{N}}{2}} \geq \frac{\alpha^2}{2} w_\epsilon^{\frac{2 \underline{s}+r-2+\frac{2r}{N}}{2}} & \text{if $|w_\epsilon|\geq 1$}.
\end{cases}
\end{equation}
Applying \eqref{eq:Sob-1} to $v_\epsilon^\frac{1}{2}$ we arrive at the inequality
\[
\int_{\Omega} v_\epsilon ~dx \leq C \left( \left(\int_{\Omega} \left(\frac{\left|\nabla v_\epsilon\right|}{\sqrt{v_\epsilon}}\right)^\frac{2N}{N+2} ~dx\right)^\frac{N+2}{N} + \left(\int_{\Omega} v_\epsilon^\frac{N}{N+2} ~dx\right)^\frac{N+2}{N}\right)
\]
with $C$ depending on $C_s$ and $N$. We start with estimating the first integral on the right-hand side of the above inequality. By the straightforward computation
\begin{equation}\label{eqest-1}
\begin{split}
& D_i \left(a^2_\epsilon  w_\epsilon^{\frac{2p+r-2+\frac{2r}{N}}{2}}\right) = \frac{ a^2_\epsilon}{2}  w_\epsilon^{\frac{2p+r-2+\frac{2r}{N}}{2}}\ln w_\epsilon D_ip \\
&
\qquad + \left(2p+r-2+\frac{2r}{N}\right) a^2_\epsilon  w_{\epsilon}^{\frac{2p+r-2+\frac{2r}{N}}{2}-1}\sum_{j=1}^N
D_{j}uD_{ij}^2u + w_\epsilon^{\frac{2p+r-2+\frac{2r}{N}}{2}} 2 a_\epsilon D_i a.
\end{split}
\end{equation}
As in estimating of the first term in \eqref{eq:est-1}, we distinguish between the cases $d<\infty$ and $d=\infty$. Assume first the $d$ is finite. The assumptions
\begin{equation}\label{eq:bounds}
    \min\{p(z), q(z)\} >\frac{2(N+1)}{N+2} \quad \text{for all} \ z \in \overline{Q}_T, \quad N \geq 2, \quad \text{and} \ \sup_{ z\in Q_T} |p(z)-q(z)| \leq \frac{2\beta}{N+2},
\end{equation}
imply
\begin{equation}\label{lowerest-2}
    2(2p - \max\{p,q\})-2 \geq 0.
\end{equation}
Now, by using \eqref{lowerest-1}, \eqref{eqest-1} and \eqref{lowerest-2} and applying Young's inequality, we have
\begin{equation}\label{upperest-1}
\begin{split}
    \dfrac{1}{\sqrt{v_{\epsilon}}} & \left|\nabla \left(a^2_\epsilon (z) w_\epsilon^{\frac{2p+r-2+\frac{2r}{N}}{2}}\right)\right|\leq C_1 a_\epsilon (z)  w_\epsilon^{\frac{2p+r-2+\frac{2r}{N}}{4}} |\ln w_\epsilon|
    \\
    &
    \quad +  C_2 r a_\epsilon (z)  w_\epsilon^{\frac{2p+r-4+\frac{2r}{N}}{4}} |u_{xx}| + C_3 |\nabla a| \begin{cases}
 w_\epsilon^{\frac{2(2p- \overline{s}) +r-2+\frac{2r}{N}}{4}}  & \text{if $|w_\epsilon| < 1$},
\\
w_\epsilon^{\frac{2(2p- \underline{s}) +r-2+\frac{2r}{N}}{4}}  & \text{if $|w_\epsilon|\geq 1$}
\end{cases}\\
& \leq  C_1 a_\epsilon (z)  w_\epsilon^{\frac{2p+r-2+\frac{2r}{N}}{4}} |\ln w_\epsilon| +  C_2 r a_\epsilon (z)  w_\epsilon^{\frac{2p+r-4+\frac{2r}{N}}{4}} |u_{xx}|
\\
&
\quad +  C_3 |\nabla a| \left(1+w_\epsilon^{\frac{\left(2(\overline{s}-1)+r+\frac{2r}{N} + 2 \lambda \right)}{4}} \right)
+ C_4,
\end{split}
\end{equation}
where $C_i$ are independent of $\epsilon$ and $r$, and
\[
\lambda:= \max_{z \in \overline{Q}_T} |p(z)-q(z)|
\]
Similarly, we have
\begin{equation}
\label{eq:upperest-2}
    \begin{split}
    \dfrac{1}{\sqrt{v_{\epsilon}}}  \left|\nabla \left(b^2_\epsilon (z) w_\epsilon^{\frac{2q+r-2+\frac{2r}{N}}{2}}\right)\right|  & \leq  C_1' b_\epsilon (z)  w_\epsilon^{\frac{2q+r-2+\frac{2r}{N}}{4}} |\ln w_\epsilon| +  C_2' r b_\epsilon (z)  w_\epsilon^{\frac{2q+r-4+\frac{2r}{N}}{4}} |u_{xx}|\\
    & \quad + C_3' |\nabla b|\left(1+ w_\epsilon^{\frac{\left(2(\overline{s}-1)+r+\frac{2r}{N} + 2 \lambda \right)}{4}}\right)  + C_4
\end{split}
\end{equation}
Combining \eqref{upperest-1} and \eqref{eq:upperest-2}, we obtain
\[
\begin{split}
& \left(\int_{\Omega} \left(\frac{\left|\nabla v_\epsilon\right|}{\sqrt{v_\epsilon}}\right)^\frac{2N}{N+2} ~dx\right)^\frac{N+2}{N}
\\
& \leq C_1'' r^2 \left(\int_\Omega \left(\left(a_\epsilon (z)  w_\epsilon^{\frac{2p+r-4+\frac{2r}{N}}{4}} + b_\epsilon (z)  w_\epsilon^{\frac{2q+r-4+\frac{2r}{N}}{4}}\right)|u_{xx}|\right)^{\frac{2N}{N+2}}\,dx\right)^{\frac{N+2}{N}}
\\
& \quad + C_2'' \left(\int_\Omega \left( \left(a_\epsilon (z)  w_\epsilon^{\frac{2p+r-2+\frac{2r}{N}}{4}} + b_\epsilon (z)  w_\epsilon^{\frac{2q+r-2+\frac{2r}{N}}{4}}\right) |\ln w_\epsilon| \right)^{\frac{2N}{N+2}}\,dx\right)^{\frac{N+2}{N}}
\\
& \quad + C_3''\left(\int_\Omega \left(w_\epsilon^{\frac{d\left(2( \overline{s}-1)+r+\frac{2r}{N} +2 \lambda \right)}{4(d-2)}} \right)^{\frac{2N}{N+2}}\,dx\right)^{\frac{N+2}{N}} + C_4'' \int_{\Omega} \left(|\nabla a|^d + |\nabla b|^d \right) + C_5''
\\ &
\equiv C'' \left(r^2I_1+I_2+I_3 + I_4 + 1\right)
\end{split}
\]
with a constant $C''$ depending on $C_i$, $C_i'$, $p^\pm$, $q^\pm$, $L_{p,q}$, $N$, $\Omega$, but independent of $r$ and $\epsilon$. The integrals $I_k$ are estimated separately. By H\"older's inequality with the conjugate exponents $\frac{N+2}{N}$ and $\frac{N+2}{2}$
\[
\begin{split}
I_1^{\frac{N}{N+2}} &  \leq C \int_{\Omega}\left(\left(a_\epsilon (z)  w_{\epsilon}^{\frac{2p+r-4}{4}} + b_\epsilon (z)  w_{\epsilon}^{\frac{2q+r-4}{4}}\right)|u_{xx}|\right)^{\frac{2N}{N+2}} \left(w_{\epsilon}^{\frac{r}{N}} \right)^\frac{N}{N+2}\,dx \\
& \leq C \left(\int_{\Omega} \left(a^2_\epsilon (z) w_{\epsilon}^{\frac{2p+r-4}{2}} + b^2_\epsilon (z) w_{\epsilon}^{\frac{2q+r-4}{2}}\right) |u_{xx}|^2\,dx\right)^{\frac{N}{N+2}} \left(\int_\Omega w_\epsilon^{\frac{r}{2}}\,dx\right)^{\frac{2}{N+2}},
\end{split}
\]
which further implies
\[
\begin{split}
I_1 & \leq C_1  \left(\int_{\Omega} \left(a^2_\epsilon (z) w_{\epsilon}^{\frac{2p+r-4}{2}} + b^2_\epsilon (z) w_{\epsilon}^{\frac{2q+r-4}{2}}\right) |u_{xx}|^2\,dx \right) \\
& \quad + C_2 \left(\int_{\Omega} \left(a^2_\epsilon (z) w_{\epsilon}^{\frac{2p+r-4}{2}} + b^2_\epsilon (z) w_{\epsilon}^{\frac{2q+r-4}{2}}\right) |u_{xx}|^2\,dx \right) \left(\int_\Omega w_\epsilon^{\frac{r+\underline{s}}{2}}\,dx\right)^{\frac{2}{N}}
\end{split}
\]
Proceeding in the same way, we estimate
\[
\begin{split}
I_2 & \equiv \left(\int_\Omega \left(\left(a_\epsilon (z)  w_\epsilon^{\frac{2p+r-2}{4}} + b_\epsilon (z)  w_\epsilon^{\frac{2q+r-2}{4}}\right)|\ln w_\epsilon|\right)^{\frac{2N}{N+2}}\left(w_\epsilon^{\frac{r}{N+2}}\right)\,dx\right)^{\frac{N+2}{N}}
\\
&
\leq C_3 \left(\int_\Omega  \left(a^2_\epsilon (z) w_{\epsilon}^{\frac{2p+r-2}{2}} + b^2_\epsilon (z) w_{\epsilon}^{\frac{2q+r-2}{2}}\right)  \ln^2 w_\epsilon\,dx\right) \\
& \qquad + C_4 \left(\int_\Omega  \left(a^2_\epsilon (z) w_{\epsilon}^{\frac{2p+r-2}{2}} + b^2_\epsilon (z) w_{\epsilon}^{\frac{2q+r-2}{2}}\right)  \ln^2 w_\epsilon\,dx\right) \left(\int_\Omega w_\epsilon^{\frac{r+\underline{s}}{2}}\,dx\right)^{\frac{2}{N}}.
\end{split}
\]
To estimate $I_3$, we use the higher integrability estimate of Corollary \ref{cor:higher-int-parab}.
By applying H\"older's inequality with the conjugate exponents $\frac{N+2}{N}$ and $\frac{N+2}{2}$, we have the following estimate:
\[
\begin{split}
I_3 & \equiv \left(\int_\Omega \left( w_\epsilon^{\frac{d \left(2(\overline{s}-1)+r+\frac{2r}{N} +2\lambda\right)}{4(d-2)}} \right)^{\frac{2N}{N+2}}\,dx\right)^{\frac{N+2}{N}}
 = \left(\int_\Omega \left( w_\epsilon^{\frac{d \left(2(\overline{s}-1)+r\right)}{4(d-2)}} \right)^{\frac{2N}{N+2}} w_\epsilon^\frac{d(r+\lambda N)}{(d-2)(N+2)} \,dx\right)^{\frac{N+2}{N}}
\\
& \leq C \left(\int_\Omega  w_\epsilon^{\frac{d \left(2(\overline{s}-1)+r\right)}{2(d-2)}} \,dx\right) \left(\int_{\Omega} w_\epsilon^\frac{d(r+\lambda N)}{2(d-2)} \,dx\right)^\frac{2}{N}.
\end{split}
\]
The first integral is estimated as in \eqref{eq:est-d}. To estimate the second one, we observe that by virtue of assumptions \eqref{eq:balance-p-q} on $\underline{s}^-$ and \eqref{eq:reg-data} on $d$
\[
d>2 + \frac{N+2}{2(1-\beta)} \left(\max\{\overline{s}^+, 2(\overline{s}^+-1)\} +r\right)\geq 2+\dfrac{N+2}{2(1-\beta)}(\underline{s}^-+r)>2+ \dfrac{2(\underline{s}^-+r)}{\underline{s}^-(1-\beta)}.
\]
Since

\[
r + \beta \underline{s}^- + \frac{2(r+\beta \underline{s}^-)}{d-2} < r+\underline{s}^-\quad \Leftrightarrow \quad 2(r+\beta \underline{s}^-)<(1-\beta)\underline{s}^-(d-2) \quad \Leftrightarrow \quad d>2+ \dfrac{2(r+\beta \underline{s}^-)}{\underline{s}^-(1-\beta)},
\]
we find that

\[
\begin{split}
& \frac{d(r+\lambda N)}{d-2} = r+\lambda N + \frac{2(r+\lambda N)}{d-2}< r + \frac{2 \beta N}{N+2} + \frac{2(r+\frac{2 \beta N}{N+2})}{d-2} < r + \beta \underline{s}^- + \frac{2(r+\beta \underline{s}^-)}{d-2} < r+\underline{s}^-.
\end{split}
\]
It follows that
\[
\int_{\Omega} w_\epsilon^\frac{d(r+\lambda N)}{2(d-2)} \,dx \leq C\left(1+\int_\Omega w_\epsilon^{\frac{\underline{s}+r}{2}}\,dx\right).
\]

Since by Young's inequality

\begin{equation}
\label{eq:flux-1}
\alpha w_\epsilon^{\frac{\underline{s}+r}{2}}\leq (a_\epsilon+b_\epsilon)w_\epsilon^{\frac{\underline{s}+r}{2}}\leq C+ w_\epsilon^{\frac{r}{2}}\mathcal{F}_{\epsilon}(z,\nabla u)|\nabla u|^2
\end{equation}
with an independent of $\epsilon$ constant $C$, the previous inequality implies

\[
\int_{\Omega} w_\epsilon^\frac{d(r+\lambda N)}{2(d-2)} \,dx \leq \left( C+ \int_{\Omega} w_\epsilon^{\frac{r}{2}}\mathcal{F}_{\epsilon}(z,\nabla u)|\nabla u|^2\,dx\right).
\]
Again, by applying Young's inequality and Corollary \ref{cor:d-1-cor-mod}, we obtain
\[
\begin{split}
I_3 & \leq C_5 + C_6 \left(\int_{\Omega} w_\epsilon^\frac{r+\underline{s}}{2} \,dx\right)^\frac{2}{N} + C_7 \left(\int_{\Omega} \left(a^2_\epsilon (z) w_{\epsilon}^{\frac{2p+r-4}{2}} + b^2_\epsilon (z) w_{\epsilon}^{\frac{2q+r-4}{2}}\right) |u_{xx}|^2\,dx \right) \\
& + C_8 \left(\int_{\Omega} \left(a^2_\epsilon (z) w_{\epsilon}^{\frac{2p+r-4}{2}} + b^2_\epsilon (z) w_{\epsilon}^{\frac{2q+r-4}{2}}\right) |u_{xx}|^2\,dx \right)\left(C+ \int_{\Omega} w_\epsilon^{\frac{r}{2}}\mathcal{F}_{\epsilon}(z,\nabla u)|\nabla u|^2\,dx\right)^\frac{2}{N}
\end{split}
\]
Gathering the estimates on $I_k$, we obtain the inequality
\begin{equation}
\label{eq:new-1}
\left(\int_{\Omega} \left(\frac{\left|\nabla v_\epsilon\right|}{\sqrt{v_\epsilon}}\right)^\frac{2N}{N+2} ~dx\right)^\frac{N+2}{N} \leq \widetilde{C} \Pi_1(t) \Pi_2^{\frac{2}{N}}(t) + \widetilde{C}',
\end{equation}
where $\widetilde C$, $\widetilde{C}'$ are constants independent of $r$ and $\epsilon$, and
\[
\begin{split}
& \Pi_1(t) = 1+ r^2\int_\Omega \left(a^2_\epsilon (z) w_{\epsilon}^{\frac{2p+r-4}{2}} + b^2_\epsilon (z) w_{\epsilon}^{\frac{2q+r-4}{2}}\right) |u_{xx}|^2\,dx \\
& \qquad \qquad \qquad \qquad  + \int_{\Omega} \left(a^2_\epsilon (z) w_{\epsilon}^{\frac{2p+r-2}{2}} + b^2_\epsilon (z) w_{\epsilon}^{\frac{2q+r-2}{2}}\right) \left(1+\ln^2w_\epsilon\right)\,dx,
\\
& \Pi_2(t)= 1+ \int_{\Omega} w_\epsilon^{\frac{r}{2}}\mathcal{F}_{\epsilon}(z,\nabla u)|\nabla u|^2\,dx.
\end{split}
\]
Applying the H\"older inequality with the conjugate exponents $\frac{N+2}{N}$ and $\frac{N+2}{2}$, we have
\[
\begin{split}
\left(\int_{\Omega} v_\epsilon^\frac{N}{N+2} ~dx\right)^\frac{N+2}{N} & \leq C \left(\int_\Omega  \left(a^2_\epsilon (z) w_{\epsilon}^{\frac{2p+r-2}{2}} + b^2_\epsilon (z) w_{\epsilon}^{\frac{2q+r-2}{2}}\right)\,dx\right) \left(\int_\Omega w_\epsilon^{\frac{r}{2}}\,dx\right)^{\frac{2}{N}}
\leq C \Pi_1(t) \Pi_2^{\frac{2}{N}}(t).
\end{split}
\]
It follows that
\begin{equation}
\label{eq:new-2}
\begin{split}
\int_{Q_T} & \left(a^2_\epsilon (z) w_\epsilon^{\frac{2p+r-2+\frac{2r}{N}}{2}} + b^2_\epsilon (z) w_\epsilon^{\frac{2q+r-2+\frac{2r}{N}}{2}}\right) \,dz \leq \widetilde{C} \left(\sup_{(0,T)}\Pi_2(t)\right)^{\frac{2}{N}}\int_{Q_T}\Pi_1(t)\,dt.
\end{split}
\end{equation}
Now we plug the obtained inequalities into \eqref{eq:I-1}:
\[
\begin{split}
\mathcal{I} & \leq \|f\|^2_{N+2,Q_T}\left(C(\mu)+\mu C'\left(\sup_{(0,T)}\Pi_2(t)\right)^{\frac{2}{N}}\int_{0}^T\Pi_1(t)\,dt \right)^{\frac{N}{N+2}}
\\
& \leq C''\|f\|^2_{N+2,Q_T}\left(C(\mu)+\mu C' \left(\sup_{(0,T)}\Pi_2(t)\right)^{\frac{2}{N+2}} \left(\int_{0}^T\Pi_1(t)\,dt\right)^{\frac{N}{N+2}} \right)
\end{split}
\]
with an arbitrary $\mu>0$ and constants $C(\mu)$, $C'$, $C''$ independent of $w_\epsilon$ and $u$. By Young's inequality with the conjugate exponents $\frac{N+2}{N}$ and $\frac{N+2}{2}$ we continue the last inequality as follows:

\[
\begin{split}
\mathcal{I} & \leq C''\|f\|^2_{N+2,Q_T}\left(C(\mu)+ C'\left(\mu^{\frac{N+2}{4}} \sup_{(0,T)}\Pi_2(t)\right)^{\frac{2}{N+2}} \left(\mu^{\frac{N+2}{2N}}\int_{0}^T\Pi_1(t)\,dt\right)^{\frac{N}{N+2}}
\right)
\\
& \leq  C''\|f\|^2_{N+2,Q_T}\left(C(\mu)+ C'\mu^{\frac{N+2}{4}} \sup_{(0,T)}\Pi_2(t)+ C'  \mu^{\frac{N+2}{2N}}\int_{0}^T\Pi_1(t)\,dt\right)
\\
& \equiv C''\|f\|^2_{N+2,Q_T}\left(C(\mu)+ \mu^{\frac{N+2}{4}} \sup_{(0,T)} \left(1+ \int_{\Omega} w_\epsilon^{\frac{r}{2}}\mathcal{F}_{\epsilon}(z,\nabla u)|\nabla u|^2\,dx \right)
\right.
\\
&
\qquad \qquad  + C'\mu^{\frac{N+2}{2N}}\left[r^2\int_{Q_T} \left(a^2_\epsilon (z) w_{\epsilon}^{\frac{2p+r-4}{2}} + b^2_\epsilon (z) w_{\epsilon}^{\frac{2q+r-4}{2}}\right) |u_{xx}|^2\,dz \right.
\\
&
\qquad \qquad \left.\left. + \int_{Q_T} \left(a^2_\epsilon (z) w_{\epsilon}^{\frac{2p+r-2}{2}} + b^2_\epsilon (z) w_{\epsilon}^{\frac{2q+r-2}{2}}\right)\left(1+\ln^2w_\epsilon\right)\,dz
\right]\right).
\end{split}
\]
By Corollary \ref{cor:d-1-cor-mod}, the second integral in the square brackets is bounded by

\[
\nu \int_{Q_T}\left(a^2_\epsilon (z) w_{\epsilon}^{\frac{2p+r-4}{2}} + b^2_\epsilon (z) w_{\epsilon}^{\frac{2q+r-4}{2}}\right)|u_{xx}|^2\,dz + \widehat C
\]
with an arbitrary $\nu>0$ and a constant $\widehat C=\widehat C(r,N,p^\pm, q^\pm, r,L_{p,q},\nu,\|u\|_{2,\Omega})$. Thus,
\begin{equation}
\label{eq:I-2}
\begin{split}
\mathcal{I}&\leq C_1 \|f\|_{N+2,Q_T}^2\left( \mu^{\frac{N+2}{4}} \sup_{(0,T)} \left( 1+ \int_{\Omega} w_\epsilon^{\frac{r}{2}}\mathcal{F}_{\epsilon}(z,\nabla u)|\nabla u|^2\,dx \right) \right.
\\
&
\left. \qquad + \mu^{\frac{N+2}{2N}} \left(r^2+\nu\right) \int_{Q_T}\left(a^2_\epsilon (z) w_{\epsilon}^{\frac{2p+r-4}{2}} + b^2_\epsilon (z) w_{\epsilon}^{\frac{2q+r-4}{2}}\right) |u_{xx}|^2\,dz + C_2\right)
\end{split}
\end{equation}
with a constant $C_1$ depending only on $\textbf{data}$, and $C_2$ depending on $\textbf{data}$ and $\mu$, $\nu$. Substituting \eqref{eq:I-2} into \eqref{eq:est-1}
and choosing $\mu$ and $\nu$ sufficiently small we transform \eqref{eq:est-1} into

\begin{equation}
\label{eq:final}
\begin{split}
\sup_{(0,T)} \int_{\Omega} w_\epsilon^{\frac{r}{2}} \mathcal{F}_\epsilon(z,\nabla u) |\nabla u|^2\,dx  & + C  \int_{Q_T} \left(a^2_\epsilon (z) w_\epsilon^{\frac{2p+r-4}{2}} + b^2_\epsilon (z) w_\epsilon^{\frac{2q+r-4}{2}} \right)|u_{xx}|^2\,dz
 \\
& \leq C' + C''\int_{\Omega} w_\epsilon^{\frac{r}{2}} \mathcal{F}_\epsilon((x,0),\nabla u_0)|\nabla u_0|^2\,dx
\end{split}
\end{equation}
with finite positive constants $C,C',C''$ depending only on $\textbf{data}$, and independent of $\epsilon$.

If $d=\infty$, the terms with $|\nabla a|, |\nabla b|$ do not need any special estimating and the needed estimate follows as in the proof of Lemma \ref{le:trace-main}.

\subsection{Case II: $\sigma \in (2,N+2)$}
\label{subsec:est-2}
We refine the estimates on the integral \eqref{eq:integral-I} and extend them to the case of low integrability of $f$.
Let us take $\sigma>2$, $r\geq 0$, and assume the following inequality holds true:
\begin{equation}\label{imp:ineq:modi}
    \frac{r\sigma}{2(\sigma-2)} \leq \frac{2(\underline{s}-1)+r+\frac{2(r+ r^\ast)}{N}}{2},
    \qquad r^\ast:= \frac{2}{N+2}.
\end{equation}
By the Young inequality (cf. with \eqref{eq:I-1})
\begin{equation}
  \label{eq:I-1-modi}
  \begin{split}
  \mathcal{I} & \leq \|f\|_{\sigma,Q_T}^2 \left(\int_{Q_T} w_\epsilon^{\frac{r \sigma}{2(\sigma-2)}}\,dz\right)^{\frac{\sigma-2}{\sigma}} \\
  & \leq \|f\|_{\sigma,Q_T}^2 \left(C + \int_{Q_T} \left(a^2_\epsilon (z) w_\epsilon^\frac{2(p-1)+r+\frac{2(r+ r^\ast)}{N}}{2} + b^2_\epsilon (z) w_\epsilon^\frac{2(q-1)+r+\frac{2(r+ r^\ast)}{N}}{2} \right)\,dz\right)^{\frac{\sigma-2}{\sigma}}.
  \end{split}
  \end{equation}
Fix
\[
\alpha_\sharp = \frac{2N}{N+2}, \quad  \quad \alpha_\sharp^\ast = 2.
\]
Since $\alpha_\sharp^\ast$ is the Sobolev conjugate exponent, by the Sobolev embedding theorem $W^{1,\alpha_\sharp}(\Omega)
\subset L^2(\Omega)$, and for every $v\in W^{1, \alpha_\sharp}(\Omega)$
\begin{equation}
\label{eq:Sob-1-modi}
\|v\|_{2,\Omega}^{2} \leq C_{s}\left(\|\nabla v\|_{\alpha_\sharp,\Omega}^{2} + \|v\|^{2}_{\alpha_\sharp,\Omega}\right)
\end{equation}
with an independent of $v$ constant $C_s$. Denote
\[
s_\epsilon(z) := a^2_\epsilon (z) w_\epsilon^{\frac{2p+r-2+\frac{2(r+ r^\ast)}{N}}{2}} + b^2_\epsilon (z) w_\epsilon^{\frac{2q+r-2+\frac{2(r+ r^\ast)}{N}}{2}}.
\]
Applying \eqref{eq:Sob-1-modi} to $s_\epsilon^\frac{1}{2}$ we arrive at the inequality
\[
\begin{split}
\int_{\Omega} s_\epsilon \,dx \leq C_s'\left(\int_\Omega \left|\frac{\nabla s_\epsilon}{\sqrt{s_\epsilon}}\right|^{\frac{2N}{N+2}}\,dz \right)^{\frac{N+2}{N}} + C_s' \left(\int_\Omega s_\epsilon^{\frac{N}{N+2}}\,dz \right)^{\frac{N+2}{N}},
\end{split}
\]
with $C'_s$ depending on $C_s$ and $\sigma$. Using \eqref{lowerest-2}, we have
\begin{equation}\label{upperest-3}
\begin{split}
    \frac{1}{\sqrt{s_\epsilon}}  \left|\nabla \left(a^2_\epsilon (z) w_\epsilon^{\frac{2p+r-2+\frac{2(r+ r^\ast)}{N}}{2}}\right)\right|
& \leq  C_1  a_\epsilon(z) w_\epsilon^{\frac{2p+r-2+\frac{2(r+ r^\ast)}{N}}{4}} |\ln w_\epsilon|
\\
&
+  C_2 r a_\epsilon (z)  w_\epsilon^{\frac{2p+r-4+\frac{2(r+ r^\ast)}{N}}{4}} |u_{xx}|  + C_3 |\nabla a| w_\epsilon^{\frac{2p + r-2+\frac{2(r+ r^\ast)}{N} + 2\lambda}{4}} + C_4
\end{split}
\end{equation}
and
\begin{equation}\label{upperest-4}
    \begin{split}
    \frac{1}{\sqrt{s_\epsilon}}  \left|\nabla \left(b^2_\epsilon (z) w_\epsilon^{\frac{2q+r-2+\frac{2(r+ r^\ast)}{N}}{2}}\right)\right|
& \leq  C_1' b_\epsilon (z)  w_\epsilon^{\frac{2q+r-2+\frac{2(r+ r^\ast)}{N}}{4}} |\ln w_\epsilon|
\\
&
+  C_2' r b_\epsilon (z)  w_\epsilon^{\frac{2q+r-4+\frac{2(r+ r^\ast)}{N}}{4}} |u_{xx}| + C_3' w_\epsilon^{\frac{2q+ r-2+\frac{2(r+ r^\ast)}{N} + 2\lambda}{4}} + C_4'.
\end{split}
\end{equation}
Combining \eqref{upperest-3} and \eqref{upperest-4}, we obtain
\[
\begin{split}
& \left(\int_{\Omega} \left(\frac{\left|\nabla s_\epsilon\right|}{\sqrt{s_\epsilon}}\right)^{\frac{2N}{N+2}} ~dx\right)^\frac{N+2}{N}
\\
\qquad &
\leq C_1'' \left(r^2\int_\Omega \left(\left(a_\epsilon (z)  w_\epsilon^{\frac{2p+r-4+\frac{2(r+ r^\ast)}{N}}{4}} + b_\epsilon (z)  w_\epsilon^{\frac{2q+r-4+\frac{2(r+ r^\ast)}{N}}{4}}\right)|u_{xx}|\right)^{\frac{2N}{N+2}}\,dx\right)^{\frac{N+2}{N}}
\\
\qquad & \quad + C_2'' \left(\int_\Omega \left( \left(a_\epsilon (z)  w_\epsilon^{\frac{2p+r-2+\frac{2r}{N}}{4}} + b_\epsilon (z)  w_\epsilon^{\frac{2q+r-2+\frac{2(r+ r^\ast)}{N}}{4}}\right) |\ln w_\epsilon| \right)^{\frac{2N}{N+2}}\,dx\right)^{\frac{N+2}{N}}
\\
\qquad & \quad + C_3''\left(\int_\Omega \left( w_\epsilon^{\frac{d\left(2 \overline{s}+r-2+\frac{2(r+ r^\ast)}{N} +2\lambda\right)}{4(d-2)}} \right)^{\frac{2N}{N+2}}\,dx\right)^{\frac{N+2}{N}} + C_4'' \int_{\Omega} \left(|\nabla a|^d + |\nabla b|^d\right) ~dx + C_4''
\\
\qquad & \equiv C'' \left(r^2I_1+I_2+I_3 + I_4 + 1\right)
\end{split}
\]
with a constant $C_i''$ depending on $C_i$, $C_i'$, $p^\pm$, $q^\pm$, $L_{p,q}$, $\eta$, but independent of $r$ and $\epsilon$. The integrals $I_k$ are estimated separately. By H\"older's inequality with the conjugate exponents $\frac{2}{\alpha_\sharp} =\frac{N+2}{N} >1$ and $ \frac{N+2}{2}>1$

\[
\begin{split}
I_1^{\frac{N}{N+2}} &  \leq C \int_{\Omega} \left(\left( (a^2_\epsilon (z)  w_{\epsilon}^{\frac{2p+r-4}{2}} + b^2_\epsilon (z) w_{\epsilon}^{\frac{2q+r-4}{2}}\right) |u_{xx}|^2\right)^{\frac{N}{N+2}} \left(w_{\epsilon}^{\frac{(r+ r^\ast)}{2}}\right)^{\frac{2}{N+2}}\,dx \\
& \leq \left(\int_{\Omega} \left(a^2_\epsilon (z)  w_{\epsilon}^{\frac{2p+r-4}{2}} + b^2_\epsilon (z) w_{\epsilon}^{\frac{2q+r-4}{2}}\right)|u_{xx}|^2\,dx\right)^{\frac{N}{N+2}} \left(\int_\Omega w_\epsilon^{\frac{(r+ r^\ast)}{2}} \,dx\right)^{\frac{2}{N+2}},
\end{split}
\]
which further implies
\[
\begin{split}
I_1 & \leq \left(\int_{\Omega} \left(a^2_\epsilon (z) w_{\epsilon}^{\frac{2p+r-4}{2}} + b^2_\epsilon (z)  w_{\epsilon}^{\frac{2q+r-4}{2}}\right)|u_{xx}|^2\,dx\right) \left(\int_\Omega w_\epsilon^{\frac{(r+ r^\ast)}{2}}\,dx\right)^{\frac{2}{N}}\\
& \leq C_1 \int_{\Omega} \left(a^2_\epsilon (z) w_{\epsilon}^{\frac{2p+r-4}{2}} + b^2_\epsilon (z)  w_{\epsilon}^{\frac{2q+r-4}{2}}\right)|u_{xx}|^2\,dx  \\
& \qquad + C_2 \left(\int_{\Omega} \left(a^2_\epsilon (z) w_{\epsilon}^{\frac{2p+r-4}{2}} + b^2_\epsilon (z)  w_{\epsilon}^{\frac{2q+r-4}{2}}\right)|u_{xx}|^2\,dx\right) \left(\int_\Omega w_\epsilon^{\frac{r+\underline{s}}{2}}\,dx\right)^{\frac{2}{N}}
\end{split}
\]
while inequality \eqref{imp:ineq:modi} takes on the following form:

\begin{equation}
\label{eq:r-modi}
0 \leq r \leq (\sigma - 2) \frac{N}{N+2-\sigma} \left(\min\{p^-, q^-\}-1 + \frac{r^\ast}{N}\right).
\end{equation}
Proceeding in the same way, we estimate
\[
\begin{split}
I_2 & \equiv \left(\int_\Omega \left( \left(a^2_\epsilon (z) w_{\epsilon}^{\frac{2p+r-2}{2}} + b^2_\epsilon (z) w_{\epsilon}^{\frac{2q+r-2}{2}}\right) \ln^2 w_\epsilon\right)^{\frac{N}{N+2}}\left(w_{\epsilon}^{\frac{(r+ r^\ast)}{2}} \right)^\frac{2}{N+2}\,dx\right)^{{\frac{N+2}{N}}}
\\
&
\leq \left(\int_\Omega  \left(a^2_\epsilon (z) w_{\epsilon}^{\frac{2p+r-2}{2}} + b^2_\epsilon (z) w_{\epsilon}^{\frac{2q+r-2}{2}}\right)  \ln^{2} w_\epsilon\,dx \right)\left(\int_\Omega w_\epsilon^{\frac{(r+ r^\ast)}{2}}\,dx\right)^{\frac{2}{N}}\\
& \leq C_1 \int_\Omega  \left(a^2_\epsilon (z) w_{\epsilon}^{\frac{2p+r-2}{2}} + b^2_\epsilon (z) w_{\epsilon}^{\frac{2q+r-2}{2}}\right)  \ln^{2} w_\epsilon\,dx \\
& \quad + C_2 \left(\int_\Omega  \left(a^2_\epsilon (z) w_{\epsilon}^{\frac{2p+r-2}{2}} + b^2_\epsilon (z) w_{\epsilon}^{\frac{2q+r-2}{2}}\right)  \ln^{2} w_\epsilon\,dx \right)\left(\int_\Omega w_\epsilon^{\frac{r+\underline{s}}{2}}\,dx\right)^{\frac{2}{N}}\\
\end{split}
\]
By H\"older's inequality, we obtain
\[
\begin{split}
I_3 & \equiv \left(\int_\Omega \left( w_\epsilon^{\frac{d\left(2\overline{s}+r-2+\frac{2(r+ r^\ast)}{N} +2\lambda\right)}{4(d-2)}} \right)^{\frac{2N}{N+2}}\,dx\right)^{\frac{N+2}{N}} = \left(\int_\Omega \left( w_\epsilon^{\frac{d\left(2\overline{s}+r-2\right)}{2(d-2)}} \right)^{\frac{N}{N+2}} \left(w_\epsilon^\frac{d(r + r^\ast + N \lambda)}{2(d-2)}\right)^\frac{2}{N+2}\,dx\right)^{\frac{N+2}{N}}\\
&\leq \left(\int_\Omega w_\epsilon^{\frac{d\left(2\overline{s}+r-2\right)}{2(d-2)}} \,dx \right)\left(\int_\Omega w_\epsilon^{\frac{d(r+ r^\ast+ N\lambda)}{2(d-2)}}\,dx\right)^{\frac{2}{N}}.
\end{split}
\]
The first integral is estimated as in \eqref{eq:est-d}. To estimate the second intergal, if we use assumptions \eqref{eq:bounds}. We have:

\[
\begin{split}
d \geq 2 + & \frac{(N+2)}{2(1-\beta)} (\overline{s}^+ +r) \geq 2 + \frac{r(N+2) + 2 (N+1)}{2 (1- \beta)} = 2+ \frac{\left(r+ r^\ast + \frac{2 N }{N+2}\right)}{\frac{2 (1- \beta)}{N+2}}\\
& \geq 2 + \frac{\left(r+ r^\ast + \frac{2 N \beta }{N+2} \right)}{\frac{N}{N+2}(1-\beta)}
= 2+ \frac{2 \left(r+ r^\ast+ \frac{2 N \beta }{N+2}\right)}{\frac{2(N+1)}{N+2} - \frac{2}{N+2} -\frac{2 N \beta }{N+2}} \geq 2+ \frac{2 \left(r+ r^\ast+ \frac{2 N \beta }{N+2}\right)}{\underline{s}^- - r^\ast -\frac{2 N \beta }{N+2} }.
\end{split}
\]
The above yields the chain of inequalities
\[
\begin{split}
& d \geq 2+ \frac{2 \left(r+ r^\ast+ \frac{2 N \beta }{N+2} \right)}{\underline{s}^- - r^\ast -\frac{2 N \beta }{N+2}}
\quad \Leftrightarrow \quad \frac{2}{d-2} \left(r + r^\ast+ \frac{2 N \beta }{N+2} \right) \leq \underline{s}^- - \left(r^\ast + \frac{2 N \beta }{N+2}\right)
 \\
& \Rightarrow \quad \left(r+ r^\ast + N \lambda \right) + \frac{2}{d-2} \left(r + r^\ast+ N \lambda \right) \leq r + \underline{s}^- \quad \Leftrightarrow \quad \frac{d\left(r+ r^\ast + N \lambda \right)}{2(d-2)} \leq \frac{r+\underline{s}^-}{2}.
\end{split}
\]
Again, by applying Young's inequality and Corollary \ref{cor:d-1-cor-mod}, we obtain
\[
\begin{split}
I_3 & \leq C_3 + C_4 \left(\int_{\Omega} w_\epsilon^\frac{r+\underline{s}}{2} \,dx\right)^\frac{2}{N} + C_5 \left(\int_{\Omega} \left(a^2_\epsilon (z) w_{\epsilon}^{\frac{2p+r-4}{2}} + b^2_\epsilon (z) w_{\epsilon}^{\frac{2q+r-4}{2}}\right) |u_{xx}|^2\,dx \right) \\
& + C_6 \left(\int_{\Omega} \left(a^2_\epsilon (z) w_{\epsilon}^{\frac{2p+r-4}{2}} + b^2_\epsilon (z) w_{\epsilon}^{\frac{2q+r-4}{2}}\right) |u_{xx}|^2\,dx \right)\left(\int_{\Omega} w_\epsilon^\frac{r+\underline{s}}{2} \,dx\right)^\frac{2}{N}
\end{split}
\]
Using \eqref{eq:flux-1} and gathering the estimates on $I_k$, $k=1,2,3$, we obtain the inequality
\begin{equation}
\label{eq:new-1-modi}
\left(\int_{\Omega} \left(\frac{\left|\nabla s_\epsilon\right|}{\sqrt{s_\epsilon}}\right)^{\frac{2N}{N+2}} ~dx\right)^\frac{N+2}{N}
 \leq \widetilde{C} \Pi_1(t) \Pi_2^{\frac{2}{N}}(t),
\end{equation}
where

\[
\begin{split}
& \Pi_1(t) = 1+ r^2\int_\Omega \left(a^2_\epsilon (z) w_{\epsilon}^{\frac{2p+r-4}{2}} + b^2_\epsilon (z) w_{\epsilon}^{\frac{2q+r-4}{2}}\right) |u_{xx}|^2\,dx \\
& \qquad \qquad \qquad \qquad  + \int_{\Omega} \left(a^2_\epsilon (z) w_{\epsilon}^{\frac{2p+r-2}{2}} + b^2_\epsilon (z) w_{\epsilon}^{\frac{2q+r-2}{2}}\right) \left(1+\ln^2w_\epsilon\right)\,dx,
\\
& \Pi_2(t)= 1+ \int_{\Omega} w_\epsilon^{\frac{r}{2}}\mathcal{F}_{\epsilon}(z,\nabla u)|\nabla u|^2\,dx. 
\end{split}
\]

An imitation of the proof given in \textbf{Case 1}, leads to the estimate
\begin{equation}
\label{eq:I-2-new}
\begin{split}
\mathcal{I} & \leq C_1 \|f\|_{\sigma,Q_T}^2 \left(\gamma^\frac{N\sigma}{2(\sigma-2)}  \sup_{(0,T)}\left( 1+ \int_{\Omega} w_\epsilon^{\frac{r}{2}}\mathcal{F}_{\epsilon}(z,\nabla u)|\nabla u|^2\,dx \right)\right.
\\
& \qquad \left.
 + \gamma^{\frac{\sigma}{\sigma-2}}\left(r^2+\nu\right) \int_{Q_T} \left(a^2_\epsilon(z) w_{\epsilon}^{\frac{2p+r-2}{2}} + b^2_\epsilon(z) w_{\epsilon}^{\frac{2q+r-2}{2}}\right) |u_{xx}|^2\,dz + C_2\right)
 \end{split}
\end{equation}
with a constant $C_1$ depending only on the same quantities as $\widehat C$, and $C_2$ depending also on $\nu$. Substituting \eqref{eq:I-2-new} into \eqref{eq:est-1}, using \eqref{eq:d-r-ell} and choosing $\gamma$ and $\nu$ sufficiently small we transform \eqref{eq:est-1} into \eqref{eq:final} with finite positive constants $\alpha$, $\beta$, $\gamma$ independent of $\epsilon$ but depending on the same quantities as $\widehat C$, and on $\|f\|_{\sigma,Q_T}$.

\subsection{Case III: $\left(\underline{s}^--1+\dfrac{2}{N+2}\right)(\sigma -2)> r$}
\label{subsec:est-3}
For $f\in L^\sigma(\Omega)$ the integral $\mathcal{I}$ in \eqref{eq:integral-I} is estimated by 

\[
\mathcal{I}\leq \int_{\Omega}|f|^{\sigma}\,dx + \int_{\Omega}w_\epsilon^{\frac{\sigma}{\sigma-2}\frac{r}{2}}\,dx.
\]
The first term is bounded by assumption. The second one is estimated by \eqref{eq:d-ell-cor-mod}, provided that

\[
\frac{\sigma r}{\sigma-2}=r+\dfrac{2r}{\sigma-2}\leq 2(\underline{s}(z)-1)+r+\frac{4}{N+2}.
\]
By the assumption on $\underline{s}^-$, this is true if we claim
\begin{equation}\label{est:small-r}
    \frac{r}{\sigma-2} \leq \underline{s}^--1+r^\ast \quad \Leftrightarrow \quad \sigma \geq 2 + \frac{r}{\underline{s}^--1+r^\ast}.
\end{equation}
\begin{remark}
In view of \eqref{eq:r-modi} and \eqref{est:small-r}, \eqref{eq:final} holds true for any $r \geq 0$ and $f \in L^\sigma(\Omega)$ with
\begin{equation}
\label{eq:combined-condition}
\sigma \geq 2 + r \min\left\{\frac{1}{\underline{s}^--1+r^\ast}, \frac{1}{\underline{s}^-- 1 + \frac{r + r^\ast}{N}}\right\}.
\end{equation}
Note that the above inequality holds true under the assumption \eqref{bound-r-sigma}.
\end{remark}

\begin{lemma}
\label{le:main-ODI}
Let $\underline{s}^->\dfrac{2(N+1)}{N+2}$, $\partial\Omega\in C^{2+\gamma}$, the balance condition \eqref{eq:balance-p-q} be fulfilled with some $\beta\in [0,1]$, and $f(\cdot,t)\in L^{\sigma}(\Omega)$ with $\sigma>2$ for every $t\in (0,T)$.
Assume that $d$ satisfies \eqref{eq:reg-data} and
\[
\begin{split}
& \text{either $\sigma\geq N+2$ and $r\geq 0$,}
\\
& \text{or $\sigma\in (2,N+2)$ and $r\geq 0$ satisfies inequality \eqref{bound-r-sigma}}.
\end{split}
\]
Then the classical solution of problem \eqref{eq:main-reg} satisfies the inequality

\begin{equation}
\label{eq:main-ODI}
\dfrac{d}{dt} \int_{\Omega}\left(\int_{0}^{w_\epsilon}s^{\frac{r}{2}} \mathcal{F}_\epsilon(z,s)\,ds\right)\,dx
+ C_0\int_\Omega w_\epsilon^{\frac{r}{2}}\mathcal{F}_{\epsilon}^2(x,\nabla u)|u_{xx}|^2\,dx\leq C_1+\|f\|_{\sigma,\Omega}^\sigma
\end{equation}
with constants $C_0$, $C_1$ depending only on $\textbf{data}$ and $\sup_{(0,T)}\|u(t)\|_{2,\Omega}^2$.
\end{lemma}

Multiplying equation \eqref{eq:main-reg} by $u_t$, integrating, and applying \eqref{eq:time-2} and \eqref{eq:main-ODI} with $r=0$, we find that

\begin{equation}
\label{eq:time-3}
\begin{split}
\|u_t\|_{2,\Omega}^2  & + \dfrac{d}{dt}\left(\int_\Omega\left(\frac{a_{\epsilon}}{p}w_\epsilon^{\frac{p}{2}} + \frac{b_{\epsilon}}{q}w_\epsilon^{\frac{q}{2}}\right)\,dx\right)
\leq C'+ C''\|f\|_{\sigma,\Omega}^\sigma.
\end{split}
\end{equation}

\begin{theorem}
\label{th:main-estimate}
If the data of problem \eqref{eq:main-reg} satisfy the conditions of Lemma \ref{le:main-ODI}, then the classical solution satisfies the estimate
\begin{equation}
\label{eq:main-est-cor}
\begin{split}
\|u_t\|_{2,Q_T}^2 & + \sup_{(0,T)}\|u(t)\|_{2,\Omega}^2 + \frac{1}{r+\overline{s}^+}\sup_{(0,T)} \int_{\Omega}|\nabla u|^{r+2}\mathcal{F}(z,\nabla u)
\,dx
+ \int_{Q_T}|\nabla u|^{2(\underline{s}-1)+r+s}\,dz
\\
&
\leq C''_1+C''_2\int_{\Omega}|\nabla u_0|^{r+2}\mathcal{F}((x,0),\nabla u_0)\,dx+\|f\|_{\sigma,Q_T}^\sigma
\\
& \quad +\epsilon\, C''_3\int_{\Omega}\left(|\nabla u_0|^{p(x,0)+r}+|\nabla u_0|^{q(x,0)+r}\right)\,dx + C''_4\epsilon^{\underline{s}^-}, \quad \epsilon\in (0,1),
\end{split}
\end{equation}
with any $s \in (0,r^\sharp)$ and independent of $\epsilon$ constants $C_i'=C_i'\left(r,\theta,\textbf{data}\right)$.
\end{theorem}

\begin{proof}
By \eqref{eq:unif-1}, \eqref{eq:time-2}, \eqref{eq:time-3}, \eqref{eq:int-2-parab}

\begin{equation}
\label{eq:est-main-parab}
\begin{split}
\|u_t\|_{2,Q_T}^2 & + \sup_{(0,T)}\|u(t)\|_{2,\Omega}^2 + \frac{1}{r+\overline{s}^+}\sup_{(0,T)} \int_{\Omega}w_\epsilon^{\frac{r+2}{2}} \mathcal{F}_\epsilon(z,\nabla u)\,dx
\\
& \quad
+ C'_0\int_{Q_T} w_\epsilon^{\frac{r}{2}}\mathcal{F}_{\epsilon}^2(z,\nabla u)|u_{xx}|^2\,dz + \int_{Q_T}w_\epsilon^{\frac{2(\underline{s}-1)+r+s}{2}}\,dz
\\
&
\leq C'_1+C_2'\int_{\Omega}(\epsilon^2+|\nabla u_0|^2)^{\frac{r+2}{2}} \mathcal{F}_\epsilon((x,0),\nabla u_0)\,dx+\|f\|_{\sigma,Q_T}^\sigma
\end{split}
\end{equation}
Inequality \eqref{eq:main-est-cor} follows from \eqref{eq:est-main-parab} because the flux $\mathcal{F}$ and the regularized flux $\mathcal{F}_\epsilon$ are connected by the following inequalities: for $\epsilon\in (0,1)$

\[
\begin{split}
|\nabla u|^{r+2}\mathcal{F}(z,\nabla u) & =a|\nabla u|^{p+r} +b|\nabla u|^{q+r}\leq  (a+\epsilon)w_\epsilon^{\frac{p+r}{2}}+(b+\epsilon)w_\epsilon^{\frac{q+r}{2}}
\\
& \leq \begin{cases}
2(a+\epsilon)w_{\epsilon}^{\frac{p+r-2}{2}}|\nabla u|^2 + 2(b+\epsilon)w_{\epsilon}^{\frac{q+r-2}{2}}|\nabla u|^2& \text{if $|\nabla u|\geq \epsilon$},
\\
C(a+b+2\epsilon) & \text{if $|\nabla u|<\epsilon$}
\end{cases}
\\
& \leq C'+C''w_{\epsilon}^{\frac{r}{2}}\mathcal{F}_{\epsilon}(z,\nabla u)|\nabla u|^2,
\end{split}
\]

\[
\begin{split}
w_{\epsilon}^{\frac{r}{2}} \mathcal{F}_{\epsilon}(z,\nabla u)|\nabla u|^2 & \leq (a+\epsilon)w_{\epsilon}^{\frac{p+r}{2}} + (b+\epsilon)w_{\epsilon}^{\frac{q+r}{2}}
\\
& \leq C|\nabla u|^{r+2}\mathcal{F}(z,\nabla u) + \epsilon C'\left(|\nabla u|^{p+r}+|\nabla u|^{q+r}\right)+C''\epsilon^{\underline{s}^-+r}.
\end{split}
\]
\end{proof}

\section{Degenerate problem with smooth data}
Let $\{u_\epsilon\}$ be the family of classical solutions of problem \eqref{eq:main-reg} with smooth data. Assumption \eqref{eq:balance-p-q} yields the inequality

\begin{equation}
\label{eq:balance-add}
\overline{s}(z)<2(\underline{s}(z)-1)+ r^\sharp.
\end{equation}
By Theorem \ref{th:main-estimate} with $r=0$ and $\theta\in (0,r^\sharp)$ (sufficiently close to $r^\sharp$) there exist functions $u$, $A$, $B$ and a sequence $\{u_{\epsilon_k}\}$

\begin{equation}
\label{eq:conv-1}
\begin{split}
& \text{$u_{\epsilon_k} \to u$ $\star$-weak in $L^\infty(0,T;L^2(\Omega))$, \quad $u_{\epsilon t}\rightharpoonup u_t$ in $L^2(Q_T)$},
\\
& \text{$\nabla u_{\epsilon_k} \rightharpoonup \nabla u$ in $L^{\overline{s}(\cdot)}(Q_T)^N$}
\\
& \text{$w_{\epsilon_k}^{ \frac{p-2}{2}}\nabla u_{\epsilon_k}\rightharpoonup A$ in $L^{p'}(Q_T)^N$}, \quad \text{$w_{\epsilon_k}^{ \frac{q-2}{2}}\nabla u_{\epsilon_k}\rightharpoonup B$ in $L^{q'}(Q_T)^N$}
\\
& \text{$\nabla u_{\epsilon_k} \rightharpoonup \nabla u$ in $L^{2(\underline{s}(\cdot)-1)+\theta}(Q_T)$}.
\end{split}
\end{equation}
With certain abuse of notation, we will denote by $\{u_\epsilon\}$ the chosen sequence $\{u_{\epsilon_k}\}$.

Since $u_\epsilon\in \mathbb{W}_{\overline{s}(\cdot,\cdot)}(Q_T)$, the functions $u_\epsilon$ satisfy the identity

\begin{equation}
\label{eq:int-iden-reg}
\begin{split}
\int_{Q_T} & \left(u_{\epsilon t}\phi+aw_{\epsilon}^{\frac{p-2}{2}}\nabla u_\epsilon\cdot\nabla \phi +bw_{\epsilon}^{\frac{q-2}{2}}\nabla u_\epsilon\cdot\nabla \phi-f\phi\right) \,dz
=-\epsilon \int_{Q_T}\left(w_{\epsilon}^{\frac{p-2}{2}} +w_{\epsilon}^{\frac{q-2}{2}}\right)\nabla u_{\epsilon}\cdot \nabla \phi\,dz
\end{split}
\end{equation}
with any test-function $\phi\in \mathbb{W}_{\overline{s}(\cdot,\cdot)}(Q_T)$. Sending $\epsilon\to 0$ and using \eqref{eq:conv-1} we arrive at the identity

\[
\int_{Q_T}\left(u_{ t}\phi+aA\cdot \nabla \phi+bB\cdot\nabla \phi-f\phi\right) \,dz=0\qquad \forall \phi\in \mathbb{W}_{\overline{s}(\cdot,\cdot)}(Q_T).
\]
To identify $A$ and $B$ we apply the Vitali convergence theorem. It is sufficient to show that
\begin{itemize}
\item[(a)] \qquad $|\nabla u_\epsilon|^{p}, |\nabla u_\epsilon|^{q}\in L^{1+\delta}(Q_T)$ with some $\delta>0$,

    \item[(b)] \qquad $\nabla u_\epsilon\to\nabla u$ a.e. in $Q_T$.
\end{itemize}
Claim (a) follows from the uniform estimate \eqref{eq:main-est-cor} and \eqref{eq:balance-add}. Let us prove (b). Introduce the functions

\[
\begin{split}
& \mathcal{S}_\epsilon(z,\nabla w)= a(\epsilon^2+|\nabla w|^2)^{\frac{p-2}{2}}+b(\epsilon^2+|\nabla w|^2)^{\frac{q-2}{2}},
\\
& \mathcal{R}_\epsilon(z,\nabla w)=(\epsilon^2+|\nabla w|^2)^{\frac{p-2}{2}}+(\epsilon^2+|\nabla w|^2)^{\frac{q-2}{2}},
\end{split}
\]
and represent $\mathcal{F}_\epsilon(z,\nabla w)=\mathcal{S}_\epsilon(z,\nabla w)+ \epsilon\mathcal{R}_\epsilon(z,\nabla w)$. Combining identities \eqref{eq:int-iden-reg} for $u_\epsilon$, $u_\mu$ with $0<\mu<\epsilon$, and the test-function $v=u_\epsilon-u_\mu$, we rewrite the result in the form

\begin{equation}
\label{eq:inter-conv}
\begin{split}
\frac{1}{2}\|v\|_{2,\Omega}^2(T) & + \int_{Q_T} \left(\mathcal{S}_{\epsilon}(z,\nabla u_\epsilon)\nabla u_\epsilon-\mathcal{S}_{\epsilon}(z,\nabla u_\mu)\nabla u_\mu\right)\cdot \nabla v\,dz
\\
&
= \int_{Q_T}f v\,dz - \epsilon\int_{Q_T}\mathcal{R}_\epsilon(z,\nabla u_\epsilon)\nabla u_\epsilon\cdot \nabla v\,dz + \mu\int_{Q_T}\mathcal{R}_\mu(z,\nabla u_\mu)\nabla u_\mu\cdot \nabla v\,dz
\\
& \quad + \int_{Q_T} \left(\mathcal{S}_{\mu}(z,\nabla u_\mu)-\mathcal{S}_{\epsilon}(z,\nabla u_\mu)\right)\nabla u_\mu\cdot \nabla v\,dz\equiv \sum_{i=1}^4\mathcal{I}_i.
\end{split}
\end{equation}
It follows from \eqref{eq:main-est-cor} with $\delta=\{\epsilon,\mu\}$ that

\[
\delta \int_{Q_T}\mathcal{R}_\delta(z,\nabla u_\delta)\nabla u_\delta\cdot \nabla v\,dz\leq \delta C\left(\| w_\delta^{\frac{p-1}{2}}\|_{p'(\cdot),Q_T}+\|\nabla w_\delta^{\frac{q-1}{2}}\|_{q'(\cdot),Q_T}\right)\left(\|\nabla v\|_{p,Q_T}+\|\nabla v\|_{q,Q_T}\right),
\]
whence $\mathcal{I}_{2},\mathcal{I}_3\to 0$ as $\epsilon\to 0$. By \eqref{eq:conv-1} $\mathcal{I}_1=(f,v)_{2,Q_T}\to 0$ as $\epsilon\to 0$. To show that $\mathcal{I}_4\to 0$ as $\epsilon\to 0$ we use the mean value theorem and consider independently the terms of $\mathcal{S}_\mu-\mathcal{S}_\epsilon$  corresponding to the powers $p$ and $q$. For every $\xi\in \mathbb{R}^N$ and $0<\mu<\epsilon<1$

\[
\begin{split}
& \left|(\epsilon^2+|\xi|^2)^{\frac{p-2}{2}}\xi-
(\mu^2+|\xi|^2)^{\frac{p-2}{2}}\xi\right| = \left|\int_0^1\dfrac{d}{ds}((\epsilon s+(1-s)\mu)^2+|\xi|^2)^{\frac{p-2}{2}} \xi\,ds\right|
\\
& \qquad \leq (\epsilon-\mu)\frac{|p-2|}{2}\int_0^1 2\left[(\epsilon s+(1-s)\mu)|\xi| \right]((\epsilon s+(1-s)\mu)^2+|\xi|^2)^{\frac{p-4}{2}} \,ds
\\
& \qquad \leq (\epsilon-\mu)\frac{|p-2|}{2}\int_{0}^1 ((\epsilon s+(1-s)\mu)^2+|\xi|^2)^{\frac{p-2}{2}}\,ds
\\
& \leq(\epsilon-\mu)\frac{|p-2|}{2}\begin{cases}
((2\epsilon)^2+|\xi|^2)^{\frac{p-2}{2}} & \text{if $p\geq 2$},
\\
\displaystyle (\epsilon-\mu)^{p-2}\int_0^1\dfrac{ds}{s^{2-p}} & \text{if $1<p<2$}
\end{cases}
\leq C\left((\epsilon-\mu)+(\epsilon-\mu)^{p-1}\right)\left(1+|\xi|^{p-1}\right).
\end{split}
\]
The same estimate holds for the term with $p$ substituted by $q$. It follows that

\[
|\mathcal{I}_4|\leq C \left((\epsilon-\mu)+(\epsilon-\mu)^{\underline{s}^--1}\right)\left(1+\|\nabla u_\mu\|_{\overline{s}(\cdot),Q_T}\right)\left(1+\|\nabla v\|_{\overline{s}(\cdot),Q_T}\right)\to 0 \quad \text{as $\epsilon\to 0$}.
\]
Returning to \eqref{eq:inter-conv} we find that
\begin{equation}
\label{eq:conv-2}
\int_{Q_T} \left(\mathcal{S}_{\epsilon}(z,\nabla u_\epsilon)\nabla u_\epsilon-\mathcal{S}_{\epsilon}(z,\nabla u_\mu)\nabla u_\mu\right)\cdot \nabla v\,dz\to 0\quad \text{as $\epsilon\to 0$}.
\end{equation}
By \cite[Proposition 3.2, Lemma 3.1]{Ar-Shm-RACSAM-2023} relation \eqref{eq:conv-2} yields

\begin{equation}
\label{eq:grad-a-e-1}
\text{$\|\nabla (u_\epsilon-u)\|_{\underline{s}(\cdot),Q_T}\to 0$ as $\epsilon\to 0$, whence $\nabla u_\epsilon\to \nabla u$ a.e. in $Q_T$.}
\end{equation}

By \eqref{eq:alpha} and \eqref{eq:main-est-cor} with $r=0$ the set $\{u_\epsilon\}$ is uniformly bounded in $L^{\infty}(0,T;W_0^{1,\underline{s}^-}(\Omega))$, and the set $\{u_{\epsilon t}\}$ is uniformly bounded in $L^2(0,T;L^2(\Omega))$. Since the embedding $W^{1,\underline{s}^-}_0(\Omega)\subset L^2(\Omega)$ is compact, it follows from \cite[Corollary 4]{Simon-1987} that the family $\{u_\epsilon\}$ is relatively compact in $C([0,T];L^2(\Omega))$.

Let $r>0$. By virtue of the uniform estimate \eqref{eq:main-est-cor} the sequence $\{u_{\epsilon}\}$ can be chosen so that

\[
\text{$\nabla u_{\epsilon}\rightharpoonup \nabla u$ in $L^{2(\underline{s}(\cdot)-1)+r+\theta}(Q_T)$ with any $\theta\in (0,r^\sharp)$},
\]
whence

\[
\|\nabla u\|_{2(\underline{s}(\cdot)-1)+r+\theta,Q_T}\leq C\quad \text{for any $\theta\in (0,r^\sharp)$ and $C=C(\theta,\textbf{data})$}.
\]
To pass to the limit in the third term on the right-hand side of \eqref{eq:main-est-cor} we use the pointwise convergence \eqref{eq:grad-a-e-1}   and the Fatou Lemma for the modulars, see \cite[Lemma 3.1.4]{HH-2019}: for a.e. $t\in (0,T)$

\[
\int_{\Omega}|\nabla u|^{r+2}\mathcal{F}(z,\nabla u)\,dx\leq \liminf_{\epsilon\to 0}\int_{\Omega}|\nabla u_{\epsilon}|^{r+2}\mathcal{F}(z,\nabla u_{\epsilon})\,dx.
\]
\begin{theorem}
\label{th:reg-smooth-exist}
Let the conditions of Theorem \ref{th:main-estimate} be fulfilled. The family of solutions to the regularized problem \eqref{eq:main-reg} $\{u_\epsilon\}$ contains a sequence that converges to a strong solution $u(z)$ of problem \eqref{eq:main}. The solution satisfies the estimate

\begin{equation}
\label{eq:est-limit-parab}
\begin{split}
\operatorname{ess}\sup_{(0,T)} & \|u(t)\|_{2,\Omega}^2 + \|u_{t}\|_{2,Q_T}^2+\frac{1}{r+\overline{s}^+}\operatorname{ess}\sup_{(0,T)}\int_{\Omega}|\nabla u|^{r+2} \mathcal{F}(z,\nabla u)\,dx
\\
&
+
\int_{Q_T}|\nabla u|^{2(\underline{s}(z)-1)+r+s}\,dz
\leq C+C'\int_{\Omega}|\nabla u_0|^{r+2} \mathcal{F}((x,0),\nabla u_0)\,dx+\|f\|_{\sigma,Q_T}^\sigma:=M
\end{split}
\end{equation}
with any $s\in (0,r^\sharp)$ and constants $C$, $C'$ depending only on $s$ and the \textbf{data}. The sequence $\{\nabla u_\epsilon\}$ converges to $\nabla u$ in $L^{\underline{s}(\cdot)}(Q_T)$ and a.e. in $Q_T$.
\end{theorem}

\begin{theorem}
\label{th:second-order-reg-1}
Let the conditions of Theorem \ref{th:main-estimate} be fulfilled, and $u(z)$ be the strong solution of problem \eqref{eq:main} with smooth data. Then

\begin{equation}
\label{eq:reg-smooth-data}
\begin{split}
{\rm (i)} \quad & a|\nabla u|^{p(z)-1+\frac{r}{2}}, \,b|\nabla u|^{q(z)-1+\frac{r}{2}}\in L^2(0,T;W^{1,2}(\Omega)),
\\
{\rm (ii)} \quad & |\nabla u|^{\frac{r}{2}}\mathcal{F}(z,\nabla u)\nabla u\in L^{2}(0,T;W^{1,2}(\Omega))^N,
\end{split}
\end{equation}
and the corresponding norms are bounded by a constant depending only on the \textbf{data}.
\end{theorem}

\begin{proof} By the Cauchy-Schwarz inequality

\begin{equation}
\label{eq:F-1}
\begin{split}
w_\epsilon|u_{xx}|^2 & \geq |\nabla u|^2|u_{xx}|^2 =\left(\sum_{k=1}^N\left(D_ku\right)^2\right)\left(\sum_{i,j=1}^N\left(D^2_{ij}u\right)^2\right)
\geq \sum_{i=1}^N\left(\sum_{j=1}^N D^2_{ij}uD_ju\right)^2
\\
& =\frac{1}{4}\sum_{i=1}^N\left(D_i\left(\epsilon^2+\sum_{j=1}^N \left(D_j u\right)^2\right)\right)^2 \equiv \frac{1}{4}|\nabla w_\epsilon|^2.
\end{split}
\end{equation}
By the straightforward computation

\begin{equation}
\label{eq:second-order-new}
\begin{split}
\left|\nabla \left(w_\epsilon^{\gamma}\right)\right|^2 & = (\gamma w_\epsilon^{\gamma-1}\nabla w_\epsilon + w_\epsilon^\gamma \ln w_\epsilon \nabla \gamma,\gamma w_\epsilon^{\gamma-1}\nabla w_\epsilon+ w_\epsilon^\gamma \ln w_\epsilon \nabla \gamma)
\\
& \leq 2\gamma^2w_\epsilon^{2(\gamma-1)}|\nabla w_\epsilon|^2 +2w_{\epsilon}^{2\gamma}\ln^2w_{\epsilon}|\nabla \gamma|^2.
\end{split}
\end{equation}
Setting

\[
2\gamma-1=\dfrac{r}{2}+\dfrac{2(p-2)}{2}\qquad \Leftrightarrow \qquad \gamma=\dfrac{2(p-1)+r}{4}
\]
and using \eqref{eq:elem-ln}, \eqref{eq:F-1}, \eqref{eq:second-order-new}, \eqref{eq:est-main-parab} we find that

\[
\begin{split}
\left|\nabla \left(w_\epsilon^{\frac{2(p-1)+r}{4}}\right)\right|^2 & \leq C_1w_{\epsilon}^{\frac{2(p-2)+r}{2}}|u_{xx}|^2 +C_2w_{\epsilon}^{\frac{2(p-1)+r}{2}}(1+w_{\epsilon}^\theta)
\end{split}
\]
with any $\theta>0$ and constants $C_1$, $C_2$ depending on $\textbf{data}$ and $\theta$, but independent of $\epsilon$. Take $\theta\in (0,r^\sharp)$, and integrate the previous inequality over $Q_T$. By Corollary \ref{cor:higher-int-parab} and Theorem \ref{th:main-estimate}

\[
\int_{Q_T}a_\epsilon^2 \left|\nabla \left(w_\epsilon^{\frac{2(p-1)+r}{4}}\right)\right|^2\,dz\leq C,\qquad C=C(\textbf{data}).
\]
Repeating these arguments with $p$, $a_\epsilon$ substituted by $q$, $b_\epsilon$ and gathering the results we arrive at the inequality

\begin{equation}
\label{eq:grad-flux-eps}
\int_{Q_T}\left(a_\epsilon^2\left|\nabla\left( w_\epsilon^{\frac{2(p-1)+r}{4}}\right)\right|^2 +b_\epsilon^2\left|\nabla\left(w_\epsilon^{\frac{2(q-1)+r}{4}}\right)\right|^2\right)\,dz\leq C,\quad C=C(\textbf{data}).
\end{equation}
Accept the notation
\[
V_\epsilon=a_\epsilon w_\epsilon^{\frac{2(p-1)+r}{4}},\qquad W_\epsilon=b_\epsilon w_\epsilon^{\frac{2(q-1)+r}{4}}.
\]
By \eqref{eq:grad-flux-eps} and \eqref{eq:est-d}, for every $i=\overline{1,N}$

\[
\begin{split}
\int_{Q_T}|D_iV_\epsilon|^2\,dz & \leq C(\textbf{data}) +C'\int_{Q_T}|\nabla a|^{2}w_{\epsilon}^{\frac{2(p-1)+r}{2}}\,dz
\\
& \leq C(\textbf{data})+C'\|\nabla a\|_{d,Q_T}^d +C'\int_{Q_T}w_\epsilon^{\frac{d}{d-2}\frac{2(p-1)+r}{2}}\,dz\leq C''(\textbf{data}).
\end{split}
\]
The same estimate is true for $D_{i}W_{\epsilon}$. It follows that there is a sequence $\{\epsilon_k\}$, $\epsilon_k\to 0^+$, and $\eta_i,\zeta_i\in L^2(Q_T)$ such that

\[
D_iV_{\epsilon_k}\rightharpoonup \eta_i, \quad D_iW_{\epsilon_k}\rightharpoonup \zeta_i\quad \text{as $\epsilon_k\to 0$.}
\]
Take and arbitrary $\phi\in C_0^\infty(Q_T)$. On the one hand,

\[
\left\langle D_iV_{\epsilon_k},\phi\right\rangle_{2,Q_T}\to \langle \eta_i,\phi\rangle_{2,Q_T},\qquad \left\langle D_iW_{\epsilon_k},\phi\right\rangle_{2,Q_T}\to \langle \zeta_i,\phi\rangle_{2,Q_T}\quad \text{as $\epsilon_k\to 0$},
\]
on the other hand,

\[
\left\langle D_iV_{\epsilon_k},\phi\right\rangle_{2,Q_T}=- \left\langle V_{\epsilon_k},D_i \phi\right\rangle_{2,Q_T} \to -\langle a|\nabla u|^{p-1+\frac{r}{2}},D_i \phi\rangle_{2,Q_T}
\]
because of the a.e. convergence $w_\epsilon\to |\nabla u|^2$, the uniform convergence $a_\epsilon\to a$, and the uniform boundedness of $D_iV_{\epsilon_k}$, $D_iW_{\epsilon_k}$ in $L^2(Q_T)$. The inclusion \eqref{eq:reg-smooth-data} (i) is proven.

The inclusion \eqref{eq:reg-smooth-data} (ii) follows by similar arguments. For the solution $u_\epsilon$ of problem \eqref{eq:main-reg}, every $i,j=\overline{1,N}$, $\beta\in C^{1}(Q_T)$, and $\theta\in (0,r^\sharp)$

\[
\begin{split}
& \left(D_j\left(a_\epsilon w_{\epsilon}^{\beta}D_iu_{\epsilon}\right)\right)^2
\\
& \qquad = \left(2a_\epsilon \beta w_{\epsilon}^{\beta-1}D_iu_{\epsilon}\sum_{k=1}^N D_k u_{\epsilon} D^2_{kj}u_{\epsilon} + a_\epsilon w_\epsilon^\beta D^2_{ij}u_{\epsilon} + a_\epsilon w_\epsilon^\beta D_iu_{\epsilon} \ln w_\epsilon D_j\beta +w_\epsilon^\beta D_iu_{\epsilon}D_ja\right)^2
\\
&  \qquad \leq C\left(a_\epsilon^2 w_\epsilon^{2\beta} |(u_{\epsilon})_{xx}|^2 + a_\epsilon^2 w_\epsilon^{2\beta+1+\theta}+ |D_ja|^2w_{\epsilon}^{2\beta+1} +1 \right)
\\
&  \qquad \leq C\left(a_\epsilon^2 w_\epsilon^{2\beta} |(u_{\epsilon})_{xx}|^2 + a_\epsilon^2 w_\epsilon^{2\beta+1+\theta}+ |D_ja|^d+w_{\epsilon}^{\frac{d}{d-2}(2\beta+1)} +1 \right)
\end{split}
\]
Take $2\beta=\frac{2(p-2)+r}{2}$. As in the derivation of \eqref{eq:reg-smooth-data} (i) we find that

\[
\left(D_j\left(a_\epsilon w_{\epsilon}^{\frac{2(p-2)+r}{4}}D_iu_\epsilon\right)\right)^2\leq C\left(a_{\epsilon}^2w_\epsilon^{\frac{2(p-2)+r}{2}}|(u_{\epsilon})_{xx}|^2 + w_{\epsilon}^{\frac{2(p-1)+r+\theta}{2}}+1\right)\leq C',\quad C'=C'(\textbf{data}).
\]
Repeating this estimate with $q$, $b$ instead of $p$, $a$, gathering the results and integrating over $Q_T$ we obtain the uniform estimate

\[
\left\|w_{\epsilon}^{\frac{r}{4}}\mathcal{F}_{\epsilon}(z,\nabla u_\epsilon)\nabla u_\epsilon\right\|_{L^2(0,T;W^{1,2}(\Omega))}\leq C(\textbf{data}),
\]
which yields the existence of $\eta_{ij}\in L^2(0,T;W^{1,2}(\Omega))$ such that

\[
D_j\left(w_{\epsilon}^{\frac{r}{4}}\mathcal{F}_{\epsilon}(z,\nabla u_\epsilon)D_ju_\epsilon\right)\rightharpoonup \eta_{ij}\quad \text{in $L^2(0,T;W^{1,2}(\Omega))$}.
\]
To identify $\eta_{ij}$ we use the a.e. convergence $w_\epsilon\to |\nabla u|^2$.
\end{proof}

\section{Degenerate problem with nonsmooth data. Proofs of the main results}
\subsection{Approximation of the data}
Consider problem \eqref{eq:main} with the data chosen in a special way. Assume that $p,q$, $a,b$ satisfy \eqref{eq:coeff}, \eqref{eq:balance-p-q}, \eqref{eq:reg-data}, $f\in L^\sigma(Q_T)$ with $\sigma> 2$, and $u_0\in \mathcal{W}_r(\Omega)$.
By \cite[Subsec.~3.2]{Arora-Shmarev-JGA-2026} there exist approximation sequences $\{f_m\}$, $\{p_m\}$, $\{q_m\}$ with the following properties:

\begin{equation}
\label{eq:approximation-1}
\begin{split}
& f_m\in C_0^\infty(Q_T),\qquad \text{$f_m\to f$ in $L^\sigma(Q_T)$},
\\
& p_m(z),\,q_m(z)\in C^\infty(Q_T),\qquad \text{$p_m\nearrow p$, $q_m\nearrow q$ in $C^{0,1}(Q_T)$}.
\end{split}
\end{equation}
To obtain the approximating sequences $\{a_m\}$, $\{b_m\}$ we extend the coefficients $a$, $b$ outside $Q_T$, introduce the functions

\[
\widetilde a_m(z)=\left(a(z)-\frac{\alpha}{2^{m+1}}\right)_+, \qquad \widetilde b_m(z)=\left(b(z)-\frac{\alpha}{2^{m+1}}\right)_+,\qquad (\,\cdot\,)_+=\max\{0,\cdot\},
\]
and mollify them with the standard kernel: $a_m=\rho_{1/m}\star \widetilde a_m$, $b_m=\rho_{1/m}\star \widetilde b_m$. The sequences $\{a_m\}$, $\{b_m\}$ have the following properties:

\begin{equation}
\label{eq:approximation-2}
\begin{split}
&
a_m,b_m\in C^{\infty}(Q_T),\qquad \text{$a_m\nearrow a$, $b_m\nearrow b$ uniformly in $\overline Q_T$},
\\
& \text{$a_m\leq a$, $b_m\leq b$ in $Q_T$},\qquad \dfrac{\alpha}{2}\leq a_m(z)+b_m(z)\quad \text{in $Q_T$},
\\
& \text{$\|\nabla a_m\|_{d,Q_T}+\|\nabla b_m\|_{L^d(Q_T)}+\|a_{m t}\|_{d,Q_T}+\|b_{m t}\|_{d,Q_T}\leq M$ uniformly in $m$}.
\end{split}
\end{equation}
By the definition of the space $\mathcal{W}_0(\Omega)$ there is a sequence 

\begin{equation}
\label{eq:initial-approx}
u_{0m}\in C_0^\infty(\Omega), \quad \text{$u_{0m}\to u_0$ in $\mathcal{W}_0(\Omega)$}\quad \Leftrightarrow \quad  \int_\Omega \mathcal{F}((x,0),\nabla(u_{0m}-u_0))|\nabla (u_{0 m}-u_0)|^2\,dx\to 0.
\end{equation}
By \cite[Lemma 3.1]{RACSAM-2023}, \eqref{eq:initial-approx} implies

\[
\int_{\Omega}|\nabla (u_{0 m}-u_0)|^{\underline{s}(x,0)}\,dx\to 0\quad \text{as $m\to \infty$},
\]
whence $|\nabla (u_{0 m}-u_0)|\rightharpoonup 0$ in $L^1(\Omega)$.
Introduce the functions
\begin{equation}
\label{eq:flux-m}
\mathcal{F}^{(m)}(z,\nabla u_m)=a_m|\nabla u_m|^{p_m-2}+b_m|\nabla u_m|^{q_m-2}.
\end{equation}
By \eqref{eq:approximation-1}, \eqref{eq:approximation-2} and Young's inequality

\begin{equation}
\label{eq:approximation-3}
\int_{\Omega}\mathcal{F}^{(m)}((x,0),\nabla u_{0 m})|\nabla u_{0m}|^2\,dx\leq
 C+\int_\Omega \mathcal{F}((x,0),\nabla u_{0 m})|\nabla u_{0 m}|^2\,dx\leq M
\end{equation}
with an independent of $m$ constant $M$.

\begin{proposition}[Lemma 3.1, \cite{Chipot-Oliveira-2019}]
\label{pro:Chipot-Gildo-2019}
Assume that the sequences $\{v_m\}$, $\{p_m\}$ satisfy the conditions

\[
\begin{split}
& 1<p^-\leq p_m(x)\leq p^+<\infty\quad \text{for a.e. $x\in \Omega$, $p^\pm=const.$},
\\
& \text{$p_m\to p$ a.e. in $\Omega$}, \quad \text{$v_m\rightharpoonup v$ in $L^1(\Omega)$},
\quad\displaystyle \int_{\Omega}|v_m|^{p_m}\,dx\leq M.
\end{split}
\]
Then $v\in L^{p(\cdot)}(\Omega)$ and

\[
\int_\Omega |v|^{p(x)}\,dx\leq \liminf_{m\to \infty} \int_\Omega | v_m|^{p_m(x)}\,dx \leq M.
\]
\end{proposition}
It follows from Proposition \ref{pro:Chipot-Gildo-2019} applied to the functions $a^{\frac{1}{p_m}}_m|\nabla (u_{0 m}-u_0)|$, $b^{\frac{1}{q_m}}_m|\nabla (u_{0 m}-u_0)|$ with the exponents $p_m$, $q_m$ that

\begin{equation}
\label{eq:init-m-conv}
\int_\Omega \mathcal{F}^{(m)}((x,0),\nabla (u_{0 m}-u_0))|\nabla (u_{0 m}-u_0)|^2\,dx\to 0.
\end{equation}

\subsection{Existence of a unique strong solution}
Denote by $\{u_m\}$ the sequence of strong solutions to the problem

\begin{equation}
\label{eq:main-non-smooth}
\begin{split}
& u_{m t}-\operatorname{div}\mathcal{F}^{(m)}(z,\nabla u_m)\nabla u_m=f_m\quad \text{in $Q_T$},
\\
& \text{$u_m=0$ on $\partial\Omega\times (0,T)$},\quad \text{$u_m(x,0)=u_{0m}(x)$ in $\Omega$}
\end{split}
\end{equation}
with the flux function $\mathcal{F}^{(m)}$ defined in \eqref{eq:flux-m}.

The data of problem \eqref{eq:main-non-smooth}, $\textbf{data}_m=\{a_m,b_m,p_m,q_m,u_{0m},f_m\}$, are smooth approximations of $\textbf{data}$. By Theorem \ref{th:reg-smooth-exist} the functions $u_m$ satisfy estimate \eqref{eq:est-limit-parab}, which is uniform with respect to $m$. Fix some $s\in (0,r^\sharp)$. By \eqref{eq:unif-2}, \eqref{eq:est-limit-parab} with $r=0$, there exists a function $u$ such that, up to a subsequence,

\begin{equation}
\label{eq:conv-m-1}
\begin{split}
& \text{$u_m\rightharpoonup u$ $\star$-weak in $L^{\infty}(0,T;L^2(\Omega))$},\qquad \text{$u_{m t}\rightharpoonup u_t$ in $L^2(Q_T)$},
\\
& \text{$\nabla u_m\rightharpoonup \nabla u$ in $L^{2(\underline{s}_m(\cdot)-1)+s}(Q_T)\subset L^{\overline{s}(\cdot)}(Q_T)$},\quad \underline{s}_m(z)=\min\{p_m(z),q_m(z)\}.
\end{split}
\end{equation}
By virtue of \eqref{eq:est-limit-parab} with $r=0$ the sequence $\{u_{m t}\}$ is uniformly bounded in $L^2(Q_T)$, while $\{\nabla u_m\}$ is uniformly bounded in $L^\infty(0,T;L^{\underline{s}^-}(\Omega))$ because of \eqref{eq:est-limit-parab} and \eqref{eq:alpha}: for a.e. $t\in (0,T)$

\[
\alpha \int_{\Omega}|\nabla u_m|^{\underline{s}^-}\,dx\leq \int_{\Omega}|\nabla u_m|^2\mathcal{F}^{(m)}(z,\nabla u_m)\,dx +C \leq M
\]
with a constant $M$ not depending on $m$. Since $\underline{s}^->\frac{2(N+1)}{N+2}$, the embedding $W_0^{1,\underline{s}^-}(\Omega)\subset L^2(\Omega)$ is compact. By \cite[Colollary 4]{Simon-1987}

\begin{equation}
\label{eq:rel-comp}
\text{the sequence $\{u_m\}$ is relatively compact in $C([0,T];L^2(\Omega)$}.
\end{equation}
Because of \eqref{eq:balance-p-q}, for all $s$ sufficiently close to $r^\sharp$, and the sufficiently large $m$

\begin{equation}
\label{eq:exp-lev}
\begin{split}
\overline{s}_m(z)-\underline{s}_m(z) & < 
\frac{2}{N+2}=\frac{2(N+1)}{N+2}-2+\frac{4}{N+2}<\underline{s}_m-2+r^\sharp.
\\
& \Rightarrow\qquad
\overline{s}_m(z)<2(\underline{s}_m(z)-1)+s.
\end{split}
\end{equation}
It follows that there exist $\delta>0$ and $m_\delta$ such that for all $m \geq m_\delta$

\begin{equation}
\label{eq:big-m}
\overline{s}(z)= (\overline{s}(z)-\overline{s}_m(z))+\overline{s}_m(z)<\delta+\overline{s}_m(z) < 2(\underline{s}_m(z)-1)+s,
\end{equation}
and there exists $\theta>0$ and $m_0$ such that the sequence $\{\nabla u_m\}_{m\geq m_0}$ is uniformly bounded in $L^{\overline{s}(\cdot)+\theta}(Q_T)^N$.
Since $\overline{s}(z)=\max\{p(z),q(z)\}$,
there exist $\eta\in L^{p'(\cdot)}(Q_T)^N$, $\zeta\in L^{q'(\cdot)}(Q_T)^N$ such that, up to a subsequence,

\[
\text{$a|\nabla u_m|^{p-2}\nabla u_m\rightharpoonup \eta$ in  $L^{p'(\cdot)}(Q_T)^N$},\qquad \text{$b|\nabla u_m|^{q-2}\nabla u_m\rightharpoonup \zeta$ in $L^{q'(\cdot)}(Q_T)^N$}.
\]
By construction, for every $m$ the solution of problem \eqref{eq:main-non-smooth} satisfies the identity

\begin{equation}
\label{eq:ident-1}
\int_{Q_T}\left(u_{m t}\phi +\mathcal{F}^{(m)}(z,\nabla u_m)\nabla u_m\cdot \nabla \phi -f_m\phi\right)\,dz=0
\end{equation}
with any test function $\phi\in \mathbb{W}_{\overline{s}(\cdot)}(Q_T)$. We rewrite \eqref{eq:ident-1} in the form

\begin{equation}
\label{eq:identity-prim}
\int_{Q_T}\left(u_{m t}\phi+\mathcal{F}(z,\nabla u_m)\cdot \nabla \phi-f_m\phi\right)\,dz=\int_{Q_T}\left(\mathcal{F}(z,\nabla u_m) -\mathcal{F}^{(m)}(z,\nabla u_m)\right)\nabla u_m\cdot \nabla \phi\,dz\equiv \mathcal{K}_m.
\end{equation}
Passing in \eqref{eq:identity-prim} to the limit as $m\to \infty$ we obtain

\begin{equation}
\label{eq:ident-2}
\int_{Q_T}\left(u_{t}\phi +(\eta+\zeta)\cdot \nabla \phi -f\phi\right)\,dz=\lim_{m\to \infty}\mathcal{K}_m.
\end{equation}
The proof of existence will be completed if we identify the limits $\eta$ and $\zeta$, and prove that $|\mathcal{K}_m|\to 0$ as $m\to \infty$.

\begin{lemma}
\label{le:strong-grad-1}
The sequence $\{\nabla u_m\}$ converges to $\nabla u$ strongly in $L^{\underline{s}(\cdot)}(Q_T)$ and a.e. in $Q_T$.
\end{lemma}

\begin{proof}
Fix $m,n\in \mathbb{N}$, take $u_m-u_n$ for the test function, and combine identities \eqref{eq:ident-1} for $u_m$, $u_n$:

\begin{equation}
\label{eq:aux-RACSAM}
\begin{split}
\frac{1}{2}\|u_m-u_n\|^2_{2,\Omega}(T) & +\int_{Q_T}  \left(\mathcal{F}(z,\nabla u_m)\nabla u_m- \mathcal{F}(z,\nabla u_n)\nabla u_n\right)\nabla(u_m-u_n)\,dz
\\
& =\int_{Q_T} (f_m-f_n) (u_m-u_n)\,dz +\frac{1}{2}\|u_{0m}-u_{0n}\|^2_{2,\Omega}
\\
& \qquad + \int_{Q_T} \left(\mathcal{F}(z,\nabla u_m)\nabla u_m- \mathcal{F}^{(m)}(z,\nabla u_m)\nabla u_m\right)\nabla(u_m-u_n)\,dz
\\
& \qquad + \int_{Q_T} \left(\mathcal{F}^{(n)}(z,\nabla u_n)\nabla u_n- \mathcal{F}(z,\nabla u_n)\nabla u_n\right)\nabla(u_m-u_n)\,dz
\equiv \sum_{i=1}^4 \mathcal{I}_i.
\end{split}
\end{equation}
By \eqref{eq:rel-comp} and the choice of the sequence $\{f_m\}$ $\mathcal{I}_1\to 0$ as $m,n\to \infty$, while $\mathcal{I}_2\to 0$ by the choice of the sequence $\{u_{0m}\}$. The integrals $\mathcal{I}_{3,4}$ are estimated similarly, therefore  we provide the detailed estimating only for the first term of $\mathcal{I}_3$. By the mean value theorem

\[
\begin{split}
(a|\nabla u_m|^{p-2} & -a_m|\nabla u_m|^{p_m-2})|\nabla u_m|= a\left(|\nabla u_m|^{p-1}- |\nabla u_m|^{p_m-1}\right) + (a-a_m)|\nabla u_m|^{p_m-1}
\\
& = a(p-p_m)|\nabla u_m|^{\theta p+(1-\theta)p_m-1}\ln |\nabla u_m| + (a-a_m)|\nabla u_m|^{p_m-1}
\end{split}
\]
with some $\theta\in (0,1)$. Since $p_m\leq p$ in $Q_T$, the strict inequality \eqref{eq:big-m} allows one to choose $\delta$ so small, and $s$ so close to $r^\sharp$, that

\[
\theta p+(1-\theta)p_m-1+\frac{\delta}{2}=p_m+\theta(p-p_m)-1+\frac{\delta}{2} \leq p-1 +\delta \leq \overline{s}-1+\delta \leq 2(\underline{s}-1)+ s-1.
\]

By \eqref{eq:elem-ln} and Young's inequality

\[
|\nabla u_m|^{\theta p+(1-\theta)p_m-1}|\ln |\nabla u_m||\leq C+ |\nabla u_m|^{\overline{s}-1+\delta}.
\]
By \eqref{eq:est-limit-parab}

\[
\begin{split}
\int_{Q_T} & \left|a|\nabla u_m|^{p-2} -a_m|\nabla u_m|^{p_m-2})|\nabla u_m|\right||\nabla(u_m-u_n)|\,dz \\
&
\leq a^+\sup_{Q_T}|p_m-p|\int_{Q_T}\left(C+|\nabla u_m|^{\overline{s}-1+\delta}\right)|\nabla (u_m-u_n)|\,dz
\\
& \qquad + \sup_{Q_T}|a-a_m|\int_{Q_T}|\nabla u_m|^{p_m-1}|\nabla (u_m-u_n)|\,dz\equiv \mathcal{J}_1 +\mathcal{J}_2.
\end{split}
\]
Using \eqref{eq:est-limit-parab} and \eqref{eq:big-m} with a new $\delta>0$, chosen as small as is needed, we may estimate

\[
\begin{split}
\mathcal{J}_1 & \leq a^+\sup_{Q_T}|p-p_m|\int_{Q_T}\left(C+|\nabla (u_m-u_n)|^{\overline{s}} + |\nabla u_m|^{\overline{s}+\delta \overline{s}'}\right)\,dz\leq C \sup_{Q_T}|p-p_m|,
\\
\mathcal{J}_2 & \leq \sup_{Q_T}|a-a_m|\left(\int_{Q_T}|\nabla (u_m-u_n)|^{\overline{s}}\,dz + \int_{Q_T}|\nabla u_m|^{\frac{p_m-1}{\overline{s}-1}\overline{s}}\,dz\right)\leq C \sup_{Q_T}|a-a_m|
\end{split}
\]
with an independent of $m,n$ constant $C$, whence $\mathcal{J}_1, \mathcal{J}_2\to 0$ as $m,n\to \infty$. Dropping the first term on the left-hand side of \eqref{eq:aux-RACSAM} we find that

\[
\int_{Q_T}  \left(\mathcal{F}(z,\nabla u_m)\nabla u_m- \mathcal{F}(z,\nabla u_n)\nabla u_n\right)\nabla(u_m-u_n)\,dz \to 0\quad \text{as $m,n\to \infty$}.
\]
By \cite[Proposition~3.2, Lemma~3.1]{RACSAM-2023} this relation yields
\[
\int_{Q_T}|\nabla (u_m-u_n)|^{\underline{s}(z)}\,dz\to 0 \qquad  \text{as $m,n\to \infty$},
\]
and the assertion follows from \cite[Lemma 3.2]{RACSAM-2023}.
\end{proof}

The integrals $\mathcal{K}_m$ on the right-hand side of \eqref{eq:identity-prim} are estimated exactly as $\mathcal{I}_{3,4}$ in \eqref{eq:aux-RACSAM}: $|\mathcal{K}_m|\to 0$ as $m\to \infty$. Passing in \eqref{eq:identity-prim} to the limit as $m\to \infty$ we arrive at identity \eqref{eq:ident-2} with the right-hand side zero. The sequence $\{\nabla u_m\}$ is uniformly bounded in $L^{\overline{s}(\cdot)+\theta}(Q_T)$, and by Lemma \ref{le:strong-grad-1} $\nabla u_m\to \nabla u$ a.e. in $Q_T$. Applying the Vitali convergence theorem to both terms of $\mathcal{F}(z,\nabla u_m)\nabla u_m$ we conclude that $\eta=a|\nabla u|^{p-2}\nabla u$, $\zeta=b|\nabla u|^{q-2}\nabla u$.

\subsection{Higher integrability and second-order regularity}
Let $0<r<\sigma-2$. Inequality \eqref{eq:est-limit-parab} allows one to choose a subsequence of $\{u_m\}$ in such a way that, in addition to \eqref{eq:conv-m-1},

\begin{equation}
\label{eq:conv-m-2}
\text{$\nabla u_m\rightharpoonup \nabla u$ in $L^{\overline{s}(\cdot)+r+s}(Q_T)$ with $s\in (0,r^\sharp)$}.
\end{equation}
The constant $M$ in \eqref{eq:est-limit-parab} depends on the structural constants, $\|u_{0m}\|_{2,\Omega}$, and

\[
\int_{\Omega}|\nabla u_{0m}|^{r+2} \mathcal{F}^{(m)}((x,0),\nabla u_{0m})\,dx+\|f_m\|_{\sigma,Q_T}^\sigma.
\]
Because of the choice of the sequences $\{f_m\}$, $\{u_{0m}\}$ and due to \eqref{eq:init-m-conv}, \eqref{eq:big-m}, inequality \eqref{eq:est-limit-parab} can be continued as follows: given $s\in (0,r^\sharp)$, for all sufficiently large $m$

\begin{equation}
\label{eq:est-m}
\begin{split}
\operatorname{ess}\sup_{(0,T)} & \|u_m(t)\|_{2,\Omega}^2 + \|u_{m t}\|_{2,Q_T}^2+\frac{1}{r+\overline{s}_m^+}\operatorname{ess}\sup_{(0,T)} \int_{\Omega}|\nabla u_m|^{r+2} \mathcal{F}^{(m)}(z,\nabla u_m)\,dx
\\
&
+
\int_{Q_T}|\nabla u_m|^{\overline{s}(z)+r+s}\,dz
\leq C+C'\int_{\Omega}|\nabla u_{0}|^{r+2} \mathcal{F}((x,0),\nabla u_{0})\,dx+\|f\|_{\sigma,Q_T}^\sigma:=N
\end{split}
\end{equation}
with independent of $m$ constants $C$, $C'$. The first, second and fourth terms on the left-hand side of \eqref{eq:est-m} have limits as $m\to \infty$, while the right-hand side does not depend on $m$ .

\begin{lemma}
\label{le:flux-conv}
For a.e. $t\in (0,T)$

\begin{equation}
\label{eq:conv-mod-degenerate}
|\nabla u_m|^{r+2}\mathcal{F}^{(m)}(z,\nabla u_m)\to |\nabla u|^{r+2}\mathcal{F}(z,\nabla u)\quad \text{a.e. in $\Omega$}
\end{equation}
and

\begin{equation}
\label{eq:bound-1}
\operatorname{ess}\sup_{(0,T)}\int_{\Omega}|\nabla u|^{r+2}\mathcal{F}(z,\nabla u)\,dx\leq N
\end{equation}
with the constant $N$ from \eqref{eq:est-m}.
\end{lemma}

\begin{proof}
The pointwise convergence \eqref{eq:conv-mod-degenerate} follows from the a.e. in $Q_T$ convergence $\nabla u^{(m)}\to \nabla u$ (Lemma \ref{le:strong-grad-1}) and the uniform convergence $a_m\to a$, $b_m\to b$, $p_m\to p$, $q_m\to q$. Consider the sequences

\[
v_m=a_m^{\frac{1}{p_m+r}}|\nabla u_m|\in L^{p_m+r}(Q_T),\qquad w_m=b_m^{\frac{1}{q_m+r}}|\nabla u_m|L^{q_m+r}(Q_T).
\]
Set $v=a^{\frac{1}{p+r}}|\nabla u|$. The continuous embedding $L^{\underline{s}(\cdot)}(Q_T)\subseteq L^{\underline{s}^-}(Q_T)$ and Lemma \ref{le:strong-grad-1} imply that $v_m(\cdot,t)\to v(\cdot,t)$ in $L^{\underline{s}^-}(\Omega)$ for a.e. $t\in (0,T)$ and, thus, $v_m\rightharpoonup v$ in $L^1(\Omega)$ for a.e. $t\in (0,T)$. For the same reason $w_m\rightharpoonup w\equiv b^{\frac{1}{q+r}}|\nabla u|$ in $L^1(\Omega)$ for a.e. $t\in (0,T)$.
Proposition \ref{pro:Chipot-Gildo-2019} applied to $\{v_m\}$ and $\{w_m\}$ with the exponents $p_m+r$ and $q_m+r$ yields $\mathcal{F}(z,\nabla u)|\nabla u|^{r+2}\in L^1(\Omega)$ for a.e. $t\in (0,T)$, whence \eqref{eq:bound-1}.
\end{proof}

Estimate \eqref{eq:th-2-est} follows now by passing  in \eqref{eq:est-m} to the limit as $m\to \infty$.

By Theorem \ref{th:second-order-reg-1} for every $m\in \mathbb{N}$

\begin{equation}
\label{eq:reg-m}
\begin{split}
{\rm (i)} \quad & a_m|\nabla u_m|^{p_m(z)-1+\frac{r}{2}}, \,b_m|\nabla u_m|^{q_m(z)-1+\frac{r}{2}}\in L^2(0,T;W^{1,2}(\Omega)),
\\
{\rm (ii)} \quad & |\nabla u_m|^{\frac{r}{2}}\mathcal{F}^{(m)}(z,\nabla u_m)\nabla u_m\in L^{2}(0,T;W^{1,2}(\Omega))^N,
\end{split}
\end{equation}
with the norms bounded by a constant independent of $m$. Denote

\[
V_m^{(i)}=|\nabla u_m|^{\frac{r}{2}}\mathcal{F}^{(m)}(z,\nabla u_m)D_i u_m,\qquad V^{(i)}=|\nabla u|^{\frac{r}{2}}\mathcal{F}(z,\nabla u)D_i u,
\]
where $u$ is the strong solution of problem \eqref{eq:main}. By virtue of \eqref{eq:reg-m} there exist a subsequence (which we assume coinciding with $\{u_m\}$) and functions $\eta_i,\xi_i,\zeta_i\in L^{2}(0,T;W^{1,2}(\Omega))$ such that

\[
\begin{split}
& a_m|\nabla u_m|^{p_m(z)-1+\frac{r}{2}}\rightharpoonup \xi_i, \; b_m|\nabla u_m|^{q_m(z)-1+\frac{r}{2}}\rightharpoonup \zeta_i\quad \text{in $L^{2}(0,T;W^{1,2}(\Omega))$},
\\
&
D_jV_m^{(i)}\rightharpoonup \eta_{ij}\quad \text{in $L^{2}(Q_T)$}.
\end{split}
\]
By virtue of the balance condition \eqref{eq:balance-p-q} and the choice of the sequences $p_m$, $q_m$

\[
2(\overline{s}_m(z)-1)+r\leq 2(\underline{s}_m(z)-1)+r+r^\sharp\quad \text{for all sufficiently large $m$}.
\]
It follows then from inequalities \eqref{eq:est-limit-parab} and \eqref{eq:est-m} that

\[
\|V_m^{(i)}\|_{2,Q_T}^2\leq C\left(1+\int_{Q_T}|\nabla u_m|^{2(p_m-1)+r}\,dz +|\nabla u_m|^{2(q_m-1)+r}\,dz\right)\leq C
\]
with an independent of $m$ constant $C$. Hence, $V_m^{(i)}\rightharpoonup \sigma_i$ in $L^2(Q_T)$ for some $\sigma_i$. By virtue of Lemma \ref{le:strong-grad-1} and the uniform convergence $a_m\to a$, $b_m\to b$, $p_m\to p$, $q_m\to q$ it is necessary that $\sigma_i=V^{(i)}$. For every $\phi\in C_0^\infty(Q_T)$

\[
\begin{split}
& \langle D_jV_m^{(i)},\phi\rangle_{2,Q_T} \to \langle \eta_{ij},\phi\rangle_{2,Q_T},
\\
& \langle D_jV_m^{(i)},\phi\rangle_{2,Q_T} = - \langle V_m^{(i)}, D_j\phi\rangle_{2,Q_T}\to - \langle V^{(i)}, D_j\phi\rangle_{2,Q_T},
\end{split}
\]
which means that

\[
\eta_{ij}=D_j\left(|\nabla u|^{\frac{r}{2}}\mathcal{F}(z,\nabla u)D_i u\right),
\]
whence inclusions \eqref{eq:reg-nonsmooth-data}.

\bibliographystyle{elsarticle-num}
\bibliography{Arora-Shmarev-2026-1}

\begin{thebibliography}{10}
\expandafter\ifx\csname url\endcsname\relax
  \def\url#1{\texttt{#1}}\fi
\expandafter\ifx\csname urlprefix\endcsname\relax\def\urlprefix{URL }\fi
\expandafter\ifx\csname href\endcsname\relax
  \def\href#1#2{#2} \def\path#1{#1}\fi

\bibitem{Arora-Shmarev-JGA-2026}
R.~Arora, S.~Shmarev, \href{https://link.springer.com/article/10.1007/s12220-026-02421-0}{Irregular double-phase evolution problem: Existence and global regularity}, Journal of Geometric Analysis 36~(6) (2026).
\newblock \href {https://doi.org/10.1007/s12220-026-02421-0} {\path{doi:10.1007/s12220-026-02421-0}}.
\newline\urlprefix\url{https://link.springer.com/article/10.1007/s12220-026-02421-0}

\bibitem{Zhikov-1986}
V.~V. Zhikov, Averaging of functionals of the calculus of variations and elasticity theory, Izv. Akad. Nauk SSSR Ser. Mat. 50~(4) (1986) 675--710, 877.

\bibitem{Zhikov-1995}
V.~V. Zhikov, On {L}avrentiev's phenomenon, Russian J. Math. Phys. 3~(2) (1995) 249--269.

\bibitem{DeFilippis-2020}
C.~De~Filippis, \href{https://doi.org/10.1007/s00526-020-01822-5}{Gradient bounds for solutions to irregular parabolic equations with {$(p, q)$}-growth}, Calc. Var. Partial Differential Equations 59~(5) (2020) Paper No. 171, 32.
\newblock \href {https://doi.org/10.1007/s00526-020-01822-5} {\path{doi:10.1007/s00526-020-01822-5}}.
\newline\urlprefix\url{https://doi.org/10.1007/s00526-020-01822-5}

\bibitem{RACSAM-2023}
R.~Arora, S.~Shmarev, \href{https://doi.org/10.1007/s13398-022-01346-x}{Existence and regularity results for a class of parabolic problems with double phase flux of variable growth}, Rev. R. Acad. Cienc. Exactas F\'is. Nat. Ser. A Mat. RACSAM 117~(1) (2023) Paper No. 34, 48.
\newblock \href {https://doi.org/10.1007/s13398-022-01346-x} {\path{doi:10.1007/s13398-022-01346-x}}.
\newline\urlprefix\url{https://doi.org/10.1007/s13398-022-01346-x}

\bibitem{Wontae-Kim-JMAA-2025}
W.~Kim, \href{https://doi.org/10.1016/j.jmaa.2025.129593}{Calder\'on-{Z}ygmund type estimate for the singular parabolic double-phase system}, J. Math. Anal. Appl. 551~(1) (2025) Paper No. 129593.
\newblock \href {https://doi.org/10.1016/j.jmaa.2025.129593} {\path{doi:10.1016/j.jmaa.2025.129593}}.
\newline\urlprefix\url{https://doi.org/10.1016/j.jmaa.2025.129593}

\bibitem{Wontae-Kim-NoDeA-2024}
W.~Kim, L.~S\"{a}rki\"{o}, \href{https://doi.org/10.1007/s00030-024-00928-5}{Gradient higher integrability for singular parabolic double-phase systems}, NoDEA Nonlinear Differential Equations Appl. 31~(3) (2024) Paper No. 40, 38.
\newblock \href {https://doi.org/10.1007/s00030-024-00928-5} {\path{doi:10.1007/s00030-024-00928-5}}.
\newline\urlprefix\url{https://doi.org/10.1007/s00030-024-00928-5}

\bibitem{Mingione-Radulescu-2021}
G.~Mingione, V.~Radulescu, \href{https://doi.org/10.1016/j.jmaa.2021.125197}{Recent developments in problems with nonstandard growth and nonuniform ellipticity}, J. Math. Anal. Appl. 501~(1) (2021) Paper No. 125197, 41.
\newblock \href {https://doi.org/10.1016/j.jmaa.2021.125197} {\path{doi:10.1016/j.jmaa.2021.125197}}.
\newline\urlprefix\url{https://doi.org/10.1016/j.jmaa.2021.125197}

\bibitem{Chl-Gw-Gw-Wr-2021}
I.~Chlebicka, P.~Gwiazda, A.~\'Swierczewska-Gwiazda, A.~Wr\'oblewska-Kami\'nska, \href{https://doi.org/10.1007/978-3-030-88856-5}{Partial differential equations in anisotropic {M}usielak-{O}rlicz spaces}, Springer Monographs in Mathematics, Springer, Cham, [2021] \copyright 2021.
\newblock \href {https://doi.org/10.1007/978-3-030-88856-5} {\path{doi:10.1007/978-3-030-88856-5}}.
\newline\urlprefix\url{https://doi.org/10.1007/978-3-030-88856-5}

\bibitem{Iwaniec-1983}
T.~Iwaniec, \href{https://doi.org/10.4064/sm-75-3-293-312}{Projections onto gradient fields and {$L\sp{p}$}-estimates for degenerated elliptic operators}, Studia Math. 75~(3) (1983) 293--312.
\newblock \href {https://doi.org/10.4064/sm-75-3-293-312} {\path{doi:10.4064/sm-75-3-293-312}}.
\newline\urlprefix\url{https://doi.org/10.4064/sm-75-3-293-312}

\bibitem{DiBenedetto-Manfredi-AMJ-1993}
E.~DiBenedetto, J.~Manfredi, \href{https://doi.org/10.2307/2375066}{On the higher integrability of the gradient of weak solutions of certain degenerate elliptic systems}, Amer. J. Math. 115~(5) (1993) 1107--1134.
\newblock \href {https://doi.org/10.2307/2375066} {\path{doi:10.2307/2375066}}.
\newline\urlprefix\url{https://doi.org/10.2307/2375066}

\bibitem{Byun-Kim-2022}
S.-S. Byun, W.~Kim, \href{https://doi.org/10.1007/s00208-020-02089-z}{Global {C}alder\'on-{Z}ygmund estimate for {$p$}-{L}aplacian parabolic system}, Math. Ann. 383~(1-2) (2022) 77--118.
\newblock \href {https://doi.org/10.1007/s00208-020-02089-z} {\path{doi:10.1007/s00208-020-02089-z}}.
\newline\urlprefix\url{https://doi.org/10.1007/s00208-020-02089-z}

\bibitem{Andrade-Bogelein-Duzaar-Moring-2026}
P.~Andrade, V.~B\"{o}gelein, F.~Duzaar, K.~Moring, {C}alder\'on-{Z}ygmund estimates for parabolic p-laplacian systems with non-divergence form right-hand sides, Ricerche mat (2026).
\newblock \href {https://doi.org/https://doi.org/10.1007/s11587-026-01137-1} {\path{doi:https://doi.org/10.1007/s11587-026-01137-1}}.

\bibitem{Kinnunen-Zhou-2001}
J.~Kinnunen, S.~Zhou, A boundary estimate for nonlinear equations with discontinuous coefficients, Differential Integral Equations 14~(4) (2001) 475--492.

\bibitem{Byun-Ok-Ryu-2013}
S.-S. Byun, J.~Ok, S.~Ryu, \href{https://doi.org/10.1016/j.jde.2013.03.004}{Global gradient estimates for general nonlinear parabolic equations in nonsmooth domains}, J. Differential Equations 254~(11) (2013) 4290--4326.
\newblock \href {https://doi.org/10.1016/j.jde.2013.03.004} {\path{doi:10.1016/j.jde.2013.03.004}}.
\newline\urlprefix\url{https://doi.org/10.1016/j.jde.2013.03.004}

\bibitem{Colombo-Mingione-2016}
M.~Colombo, G.~Mingione, \href{https://doi.org/10.1016/j.jfa.2015.06.022}{Calder\'on-{Z}ygmund estimates and non-uniformly elliptic operators}, J. Funct. Anal. 270~(4) (2016) 1416--1478.
\newblock \href {https://doi.org/10.1016/j.jfa.2015.06.022} {\path{doi:10.1016/j.jfa.2015.06.022}}.
\newline\urlprefix\url{https://doi.org/10.1016/j.jfa.2015.06.022}

\bibitem{De-Filippis-Mingione-2020}
C.~De~Filippis, G.~Mingione, \href{https://doi.org/10.1090/spmj/1608}{A borderline case of {C}alder\'on-{Z}ygmund estimates for nonuniformly elliptic problems}, Algebra i Analiz 31~(3) (2019) 82--115.
\newblock \href {https://doi.org/10.1090/spmj/1608} {\path{doi:10.1090/spmj/1608}}.
\newline\urlprefix\url{https://doi.org/10.1090/spmj/1608}

\bibitem{Acerbi-Mingione-2007}
E.~Acerbi, G.~Mingione, \href{https://doi.org/10.1215/S0012-7094-07-13623-8}{Gradient estimates for a class of parabolic systems}, Duke Math. J. 136~(2) (2007) 285--320.
\newblock \href {https://doi.org/10.1215/S0012-7094-07-13623-8} {\path{doi:10.1215/S0012-7094-07-13623-8}}.
\newline\urlprefix\url{https://doi.org/10.1215/S0012-7094-07-13623-8}

\bibitem{Seregin-Acerbi-Mignione-2004}
E.~Acerbi, G.~Mingione, G.~A. Seregin, \href{https://doi-org.uniovi.idm.oclc.org/10.1016/S0294-1449(03)00031-3}{Regularity results for parabolic systems related to a class of non-{N}ewtonian fluids}, Ann. Inst. H. Poincar\'{e} C Anal. Non Lin\'{e}aire 21~(1) (2004) 25--60.
\newblock \href {https://doi.org/10.1016/S0294-1449(03)00031-3} {\path{doi:10.1016/S0294-1449(03)00031-3}}.
\newline\urlprefix\url{https://doi-org.uniovi.idm.oclc.org/10.1016/S0294-1449(03)00031-3}

\bibitem{Crispo-Maremonti-Ruzicka-2019}
F.~Crispo, P.~Maremonti, M.~Rucicka, \href{https://projecteuclid.org/euclid.ade/1556762454}{Global {$L^r$}-estimates and regularizing effect for solutions to the {$p(t,x)$}-{L}aplacian systems}, Adv. Differential Equations 24~(7-8) (2019) 407--434.
\newline\urlprefix\url{https://projecteuclid.org/euclid.ade/1556762454}

\bibitem{Baroni-Bogelin-2014}
P.~Baroni, V.~B\"ogelein, \href{https://doi.org/10.4171/RMI/817}{Calder\'on-{Z}ygmund estimates for parabolic {$p(x,t)$}-{L}aplacian systems}, Rev. Mat. Iberoam. 30~(4) (2014) 1355--1386.
\newblock \href {https://doi.org/10.4171/RMI/817} {\path{doi:10.4171/RMI/817}}.
\newline\urlprefix\url{https://doi.org/10.4171/RMI/817}

\bibitem{Byun-Oh-Wang-2015}
S.-S. Byun, J.~Oh, L.~Wang, \href{https://doi.org/10.1093/imrn/rnu203}{Global {C}alder\'on-{Z}ygmund theory for asymptotically regular nonlinear elliptic and parabolic equations}, Int. Math. Res. Not. IMRN~(17) (2015) 8289--8308.
\newblock \href {https://doi.org/10.1093/imrn/rnu203} {\path{doi:10.1093/imrn/rnu203}}.
\newline\urlprefix\url{https://doi.org/10.1093/imrn/rnu203}

\bibitem{Kim-2025}
W.~Kim, Calder\'on-zygmund type estimate for the parabolic double-phase system, to appear in Ann. Sc. Norm. Super. Pisa Cl. Sci. (2025).

\bibitem{Ant-Zhikov-2005}
S.~Antontsev, V.~Zhikov, Higher integrability for parabolic equations of {$p(x,t)$}-{L}aplacian type, Adv. Differential Equations 10~(9) (2005) 1053--1080.

\bibitem{Zhikov-Past-2010}
V.~V. Zhikov, S.~E. Pastukhova, \href{https://doi.org/10.1134/S0001434610010256}{On the property of higher integrability for parabolic systems of variable order of nonlinearity}, Mat. Zametki 87~(2) (2010) 179--200.
\newblock \href {https://doi.org/10.1134/S0001434610010256} {\path{doi:10.1134/S0001434610010256}}.
\newline\urlprefix\url{https://doi.org/10.1134/S0001434610010256}

\bibitem{Hasto-OK-2021}
P.~H\"{a}st\"{o}, J.~Ok, \href{https://doi.org/10.1016/j.jde.2021.08.012}{Higher integrability for parabolic systems with {O}rlicz growth}, J. Differential Equations 300 (2021) 925--948.
\newblock \href {https://doi.org/10.1016/j.jde.2021.08.012} {\path{doi:10.1016/j.jde.2021.08.012}}.
\newline\urlprefix\url{https://doi.org/10.1016/j.jde.2021.08.012}

\bibitem{Sen-2025}
A.~Sen, \href{https://doi.org/10.1007/s12220-025-01950-4}{Gradient higher integrability for degenerate/singular parabolic multi-phase problems}, J. Geom. Anal. 35~(6) (2025) Paper No. 170, 95.
\newblock \href {https://doi.org/10.1007/s12220-025-01950-4} {\path{doi:10.1007/s12220-025-01950-4}}.
\newline\urlprefix\url{https://doi.org/10.1007/s12220-025-01950-4}

\bibitem{Ar-Sh-2021}
R.~Arora, S.~Shmarev, \href{https://doi.org/10.1016/j.jmaa.2020.124506}{Strong solutions of evolution equations with {$p(x,t)$}-{L}aplacian: existence, global higher integrability of the gradients and second-order regularity}, J. Math. Anal. Appl. 493~(1) (2021) Paper No. 124506, 31.
\newblock \href {https://doi.org/10.1016/j.jmaa.2020.124506} {\path{doi:10.1016/j.jmaa.2020.124506}}.
\newline\urlprefix\url{https://doi.org/10.1016/j.jmaa.2020.124506}

\bibitem{Ar-Sh-2024}
R.~Arora, S.~Shmarev, \href{https://doi.org/10.1515/anona-2024-0016}{Optimal global second-order regularity and improved integrability for parabolic equations with variable growth}, Advances in Nonlinear Analysis 13~(1) (2024) 20240016 [cited 2024-07-07].
\newblock \href {https://doi.org/doi:10.1515/anona-2024-0016} {\path{doi:doi:10.1515/anona-2024-0016}}.
\newline\urlprefix\url{https://doi.org/10.1515/anona-2024-0016}

\bibitem{Arora-Shmarev-JMAA-2025}
R.~Arora, S.~Shmarev, \href{https://doi.org/10.1016/j.jmaa.2025.129582}{Global gradient estimates for solutions of parabolic equations with nonstandard growth}, J. Math. Anal. Appl. 549~(2) (2025) Paper No. 129582, 36.
\newblock \href {https://doi.org/10.1016/j.jmaa.2025.129582} {\path{doi:10.1016/j.jmaa.2025.129582}}.
\newline\urlprefix\url{https://doi.org/10.1016/j.jmaa.2025.129582}

\bibitem{Duzaar-Mignione-Steffen-2011}
F.~Duzaar, G.~Mingione, K.~Steffen, \href{https://doi-org.uniovi.idm.oclc.org/10.1090/S0065-9266-2011-00614-3}{Parabolic systems with polynomial growth and regularity}, Mem. Amer. Math. Soc. 214~(1005) (2011) x+118.
\newblock \href {https://doi.org/10.1090/S0065-9266-2011-00614-3} {\path{doi:10.1090/S0065-9266-2011-00614-3}}.
\newline\urlprefix\url{https://doi-org.uniovi.idm.oclc.org/10.1090/S0065-9266-2011-00614-3}

\bibitem{Feng-Parviainen-Sarsa-2023}
Y.~Feng, M.~Parviainen, S.~Sarsa, A systematic approach on the second order regularity of solutions to the general parabolic {$p$}-{L}aplace equation, Calc. Var. Partial Differential Equations 62~(7) (2023) Paper No. 204.

\bibitem{Berselli-Ruzicka-2022}
L.~C. Berselli, M.~Ruzicka, \href{https://doi.org/10.1007/s00526-022-02247-y}{Natural second-order regularity for parabolic systems with operators having {$(p,\delta)$}-structure and depending only on the symmetric gradient}, Calc. Var. Partial Differential Equations 61~(4) (2022) Paper No. 137, 49.
\newblock \href {https://doi.org/10.1007/s00526-022-02247-y} {\path{doi:10.1007/s00526-022-02247-y}}.
\newline\urlprefix\url{https://doi.org/10.1007/s00526-022-02247-y}

\bibitem{Cianchi-Maz'ya-2019-1}
A.~Cianchi, V.~G. Maz'ya, \href{https://doi-org.uniovi.idm.oclc.org/10.1007/s12220-019-00213-3}{Second-order regularity for parabolic {$p$}-{L}aplace problems}, J. Geom. Anal. 30~(2) (2020) 1565--1583.
\newblock \href {https://doi.org/10.1007/s12220-019-00213-3} {\path{doi:10.1007/s12220-019-00213-3}}.
\newline\urlprefix\url{https://doi-org.uniovi.idm.oclc.org/10.1007/s12220-019-00213-3}

\bibitem{DHHR-2011}
L.~Diening, P.~Harjulehto, P.~H\"ast\"o, M.~Ruzicka, \href{https://doi.org/10.1007/978-3-642-18363-8}{Lebesgue and {S}obolev spaces with variable exponents}, Vol. 2017 of Lecture Notes in Mathematics, Springer, Heidelberg, 2011.
\newblock \href {https://doi.org/10.1007/978-3-642-18363-8} {\path{doi:10.1007/978-3-642-18363-8}}.
\newline\urlprefix\url{https://doi.org/10.1007/978-3-642-18363-8}

\bibitem{HH-2019}
P.~Harjulehto, P.~H\"ast\"o, \href{https://doi.org/10.1007/978-3-030-15100-3}{Orlicz spaces and generalized {O}rlicz spaces}, Vol. 2236 of Lecture Notes in Mathematics, Springer, Cham, 2019.
\newblock \href {https://doi.org/10.1007/978-3-030-15100-3} {\path{doi:10.1007/978-3-030-15100-3}}.
\newline\urlprefix\url{https://doi.org/10.1007/978-3-030-15100-3}

\bibitem{Grisvard-2011}
P.~Grisvard, \href{https://doi.org/10.1137/1.9781611972030.ch1}{Elliptic problems in nonsmooth domains}, Vol.~69 of Classics in Applied Mathematics, Society for Industrial and Applied Mathematics (SIAM), Philadelphia, PA, 2011, reprint of the 1985 original [MR0775683], With a foreword by Susanne C. Brenner.
\newblock \href {https://doi.org/10.1137/1.9781611972030.ch1} {\path{doi:10.1137/1.9781611972030.ch1}}.
\newline\urlprefix\url{https://doi.org/10.1137/1.9781611972030.ch1}

\bibitem{Ar-Shm-RACSAM-2023}
R.~Arora, S.~Shmarev, \href{https://doi.org/10.1007/s13398-022-01346-x}{Existence and regularity results for a class of parabolic problems with double phase flux of variable growth}, Rev. R. Acad. Cienc. Exactas F\'is. Nat. Ser. A Mat. RACSAM 117~(1) (2023) Paper No. 34, 48.
\newblock \href {https://doi.org/10.1007/s13398-022-01346-x} {\path{doi:10.1007/s13398-022-01346-x}}.
\newline\urlprefix\url{https://doi.org/10.1007/s13398-022-01346-x}

\bibitem{Simon-1987}
J.~Simon, \href{https://doi.org/10.1007/BF01762360}{Compact sets in the space {$L^p(0,T;B)$}}, Ann. Mat. Pura Appl. (4) 146 (1987) 65--96.
\newblock \href {https://doi.org/10.1007/BF01762360} {\path{doi:10.1007/BF01762360}}.
\newline\urlprefix\url{https://doi.org/10.1007/BF01762360}

\bibitem{Chipot-Oliveira-2019}
M.~Chipot, H.~B. de~Oliveira, \href{https://doi.org/10.1007/s00208-019-01803-w}{Some results on the {$p(u)$}-{L}aplacian problem}, Math. Ann. 375~(1-2) (2019) 283--306.
\newblock \href {https://doi.org/10.1007/s00208-019-01803-w} {\path{doi:10.1007/s00208-019-01803-w}}.
\newline\urlprefix\url{https://doi.org/10.1007/s00208-019-01803-w}

\end{thebibliography}
\end{document}